\pdfoutput=1
\documentclass{amsart}
\usepackage{todonotes}
\usepackage[T1]{fontenc}
\usepackage{graphicx, amsmath, amsthm,subcaption,amssymb}
\usepackage{tikz-cd}
\usepackage{spectralsequences}
\usepackage{hyperref}
\hypersetup{
  colorlinks   = true, 
  urlcolor     = blue, 
  linkcolor    = blue, 
  citecolor   = red 
}
\usepackage{adjustbox}
\usepackage{bbm}
\usepackage[backend=biber, style=numeric, maxbibnames=99, doi=true,url=false,
    giveninits=true, hyperref]{biblatex}
\renewbibmacro{in:}{}

\newcommand{\st}{\text{ : }}

\newcommand{\CP}{\mathbb{C}\mathbb{P}}

\newcommand{\Z}{{\mathbb  Z}}

\newcommand{\R}{{\mathbb R}}
\newcommand{\F}{{\mathbb F}}

\newcommand{\Boxover}[1]{\underset{#1}{\Box}}

\DeclareMathOperator{\Hom}{Hom}
\DeclareMathOperator{\Ext}{Ext}

\DeclareMathOperator{\Spec}{Spec}

\newcommand{\m}[1]{{\protect\underline{#1}}}

\newcommand{\Sp}{\mathrm {Sp}}

\mathchardef\mhyphen=45

\numberwithin{equation}{section}

\newtheorem{theorem}{Theorem}[section]
\newtheorem{lemma}[theorem]{Lemma}
\newtheorem{corollary}[theorem]{Corollary}

\newtheorem{proposition}[theorem]{Proposition}
\newtheorem{conjecture}[theorem]{Conjecture}
\newtheorem{question}[theorem]{Question}
\newtheorem*{theorem*}{Theorem}
\newtheorem*{proposition*}{Proposition}

\theoremstyle{remark}
\newtheorem{remark}[theorem]{Remark}
\newtheorem{example}[theorem]{Example}
\newtheorem{notation}[theorem]{Notation}

\theoremstyle{definition}

\newtheorem{definition}[theorem]{Definition}
\newtheorem{construction}[theorem]{Construction}

\thanks{This material is based upon work supported by the National Science Foundation under Grant
No. DMS-1811189}
\begin{document}
\title{Chromatic defect, Wood's theorem, and higher Real $K$-theories}
\author{Christian Carrick}
\address{Mathematisches Institut, Endenicher Allee 60, Bonn, Germany}
\email{carrick@math.uni-bonn.de}
\begin{abstract}
Using Ravenel's Thom spectrum $\mathrm{X(n)}$, we introduce the concept of chromatic defect, which measures how far a spectrum is from being complex-orientable. We compute the chromatic defect of various examples of interest, such as finite spectra, the Real Johnson--Wilson spectra $\mathrm{ER(n)}$, fixed points  of Morava $\mathrm{E}$-theories (with respect to finite subgroups of the Morava stabilizer group), and the connective image of $\mathrm{J}$ spectrum. Moreover, an obstruction theory is developed for determining chromatic defect. Having finite chromatic defect is closely related to the existence of analogues of the classical Wood equivalence. We show that such equivalences exist in a wide generality and use them to construct $\Z$-indexed Adams--Novikov towers.
\end{abstract}
\maketitle
\tableofcontents
\section{Introduction}


The chromatic approach to stable homotopy uses the complex cobordism spectrum $\mathrm{MU}$ to detect phenomena in stable homotopy. The $\mathrm{MU}$-homology $\mathrm{MU}_*E$ of a spectrum $E$ is often much more computable than its stable homotopy $\pi_*E$. However, the former is the input to the Adams--Novikov spectral sequence, which can be used to recover information about the latter. This is an extremely powerful tool that detects nilpotence and periodicity phenomena, as conjectured by Ravenel \cite{ravorange} and proven by Devinatz--Hopkins--Smith \cite{dhs}. When $E$ is complex-orientable, however, there is essentially no difference between $\pi_*E$ and $\mathrm{MU}_*E$; the former is recovered as the comodule primitives in the latter, and thus $E$ has no complexity from the point of view of chromatic homotopy theory. 

There are many useful spectra $E$ that are not complex-orientable, but are in some sense only finitely many steps away from being complex-orientable. Their complexity is thus small from the point of view of chromatic homotopy, which is reflected in the computability of their Adams--Novikov spectral sequence. The clearest example of this is $\mathrm{ko}$, connective real $K$-theory. The classical Wood's theorem \cite{wood} states that there is an equivalence
\[\mathrm{ko}\otimes C(\eta)\simeq\mathrm{ku}\]
where $C(\eta)$ is the cofiber of the Hopf map $\eta:\mathbb{S}^1\to\mathbb{S}^0$, and $\mathrm{ku}$ is connective complex $K$-theory. The spectrum $\mathrm{ku}$ is complex-orientable and $\mathrm{ko}$ is not. However, $C(\eta)$ is a 2-cell complex with chromatic type zero, in the sense of the thick subcategory theorem \cite{ravorange}.

Wood's theorem is useful for calculations involving $\mathrm{ko}$, and it also implies that $\mathrm{ko}$ is in the thick tensor ideal generated by $\mathrm{MU}$, so that for example the Adams--Novikov spectral sequence of $\mathrm{ko}$ has a horizontal vanishing line on a finite page and $\mathrm{ko}$ satisfies the conditions of the telescope conjecture at all heights. We study analogues of Wood's theorem and say that a spectrum $E$ is \emph{Wood-type} if there is a finite $\mathrm{BP}$-projective $F$ such that $E\otimes F$ is complex-orientable (\ref{def:wood}). 

\subsubsection*{Chromatic Defect} Using Bott periodicity, Ravenel introduced a  filtration
\[\mathbb S=\mathrm{X(1)}\to \cdots\to \mathrm{X(n)}\to\cdots\to\mathrm{X(\infty)}=\mathrm{MU}\]
of $\mathrm{MU}$, where $\mathrm{X(n)}$ is a certain Thom spectrum over $\Omega\mathrm{SU(n)}$ 
\cite{ravgreen}. The $\mathrm{X(n)}$'s provide a convenient interpolation between stable homotopy and $\mathrm{MU}$-homology, and quite a lot is known about this filtration. In fact, the Devinatz--Hopkins--Smith proof of the nilpotence theorem proceeds by induction downward along this filtration: they show that each of the maps $\mathrm{X(n)}\to \mathrm{X(n+1)}$ detects nilpotence, then by compactness conclude that $\mathbb S\to \mathrm{MU}$ detects nilpotence \cite{dhs}.

Recovering stable homotopy from $\mathrm{MU}$-homology entails infinite descent along the Ravenel filtration. In general, calculating $\mathrm{X(n)}$-homology is much more difficult than calculating $\mathrm{MU}$-homology. Typically, calculating $\mathrm{X(n)}$-homology is as difficult as calculating stable homotopy as, in some sense, the $\mathrm{X(n)}$'s are too close to the sphere spectrum. 

We study theories $E$ with the property that $E\otimes\mathrm{X(n)}$ is complex-orientable for some $n<\infty$, and we say in this case that $E$ has \emph{chromatic defect} $\le n$ (\ref{def:defect}). When $E$ has chromatic defect $\le n$, the role of $\mathrm{MU}$-homology may be replaced by that of $\mathrm{X(n)}$-homology, and thus the stable homotopy of $E$ may be recovered by a \emph{finite} descent along Ravenel's filtration. As $E$ has chromatic defect 1 if and only if $E$ is itself complex-orientable, chromatic defect measures the failure of $E$ to be complex-orientable in a precise sense. The condition of $E$ having finite chromatic defect is closely related to that of $E$ being Wood-type; for example $\mathrm{ko}$ has chromatic defect 2. In fact, finite chromatic defect is a necessary condition for being Wood-type (\ref{cor:woodimpliesfinitedefect}).

Ravenel has studied extensively the process of recovering stable homotopy by descent along his filtration from a computational point of view \cite[Section 7.1]{ravgreen}. In fact, in an unpublished work \cite{hhreo}, Hill--Hopkins--Ravenel use this filtration to compute the homotopy groups of the fixed points of Morava $E$-theories with respect to a natural $C_p$-action at chromatic height $n=2(p-1)$. We compute chromatic defect in several cases of interest, which determines precisely the stage of Ravenel's filtration at which these descent computations must begin, and we focus on more general fixed point spectra of Morava $E$-theories with respect to finite subgroups of the Morava stabilizer group.

\subsubsection*{$\mathrm{K(n)}$-local homotopy and the Morava stabilizer group} Let $\mathrm{K(n)}$ be the $n$-th Morava $K$-theory. By a theorem of Devinatz--Hopkins, the $\mathrm{K(n)}$-local sphere can be described as the homotopy fixed-point spectrum \cite{devhop}
\[L_{\mathrm{K(n)}}\mathbb S\simeq \mathrm{E}(k,\Gamma)^{h\mathbb G_n},\]
where $\mathrm{E}(k,\Gamma)$ is a height $n$ Morava $E$-theory, for $\Gamma$ a height $n$ formal group over a perfect field $k$ of characteristic $p$, and $\mathbb G_n=\mathrm{Aut}(\Gamma)\rtimes\mathrm{Gal}(k/\F_p)$ the corresponding Morava stabilizer group. It was an observation of Ravenel \cite{ravarf} and Hopkins--Miller \cite{hopmiller} that $L_{\mathrm{K(n)}}\mathbb S$ is  approximated by the fixed points of $\mathrm{E}(k,\Gamma)$ at \textit{finite} subgroups $G$ of $\mathbb G_n$, and that these theories are more computable. These theories
\begin{equation}\label{eq:eon}
\mathrm{EO}_n(G):=\mathrm{E}(k,\Gamma)^{hG}
\end{equation}
are known as the \emph{Hopkins--Miller higher real $K$-theories}.

We compute the chromatic defect of higher real $K$-theories using a construction of Hopkins, which associates a stack $\mathcal{M}_E$ to any homotopy commutative ring spectrum $E$. The stack $\mathcal{M}_E$ is the one corresponding to the Hopf algebroid
    \[(\mathrm{MU}_{2*}E,\mathrm{MU}_{2*}(\mathrm{MU}\otimes E)).\]
Via the unit map $\mathbb S\to E$, the stack $\mathcal{M}_E$ comes equipped with a canonical affine morphism $p_E:\mathcal{M}_E\to\mathcal{M}_{FG}(1)$ to the moduli stack of formal groups and strict automorphisms. This stack allows one to describe the $E_2$-page of the Adams--Novikov spectral sequence of $E$ from an algebro-geometric point of view, and it is the most essential tool in our analysis of chromatic defect. Indeed, applying Hopkins' stack construction to Ravenel's filtration produces the tower
\[\Spec(L)\to\cdots\to\mathcal{M}_{FG}(n)\to\cdots\to\mathcal{M}_{FG}(1)\]
where $L$ is the Lazard ring, and $\mathcal{M}_{FG}(n)$ is the moduli stack of formal group laws equipped with a coordinate through degree $n$. This often allows one to turn the question of whether $E$ has finite chromatic defect into a rigidity question about $\mathcal{M}_E$: namely, whether the automorphism groups of objects in $\mathcal{M}_E$ become trivial after fixing coordinates on the associated formal groups through some finite degree.

\subsection{Main results and outline of the paper}

\subsubsection*{Section 2} In the classical Wood equivalence $\mathrm{ko}\otimes C(\eta)\simeq\mathrm{ku}$, a priori the left hand side has no ring structure, as $C(\eta)$ does not admit a unital multiplication. Our definition of Wood-types thus requires a more flexible definition of complex-orientability for a spectrum $E$ that does not require that $E$ be a ring spectrum. Therefore, we make the following definition.

\begin{definition}[\ref{def:comporient}]\label{def:introdefcomporient}
    A spectrum $E$ is complex-orientable if the map
    \[E\otimes\sigma_k:E\otimes\mathbb S^{2k+1}\to E\otimes\CP^k\]
    is nullhomotopic for all $k\ge1$, where $\sigma_k$ is the attaching map for the top cell in $\CP^{k+1}$.
\end{definition}

\begin{remark}
    In the literature, complex orientations are defined only for ring spectra. We say a ring spectrum $E$ is complex-orientable if the tautological line bundle over $\CP^\infty$ is orientable with respect to $E$. In Section \ref{sec:2}, we verify that this definition of complex-orientability is equivalent to Definition \ref{def:introdefcomporient} for any homotopy associative ring spectrum.
\end{remark}

Quillen showed that a homotopy commutative ring spectrum $E$ is complex-orientable if and only if there exists a homotopy ring map $\mathrm{MU}\to E$ \cite{quillen}. We extend this result to complex-orientable spectra (in the sense of Definition \ref{def:introdefcomporient}) in the following way.

\begin{theorem}[\ref{thm:weakMU}]
A spectrum $E$ is complex-orientable if and only if it is a weak $\mathrm{MU}$-module; that is, the unit map $E\to \mathrm{MU}\otimes E$ admits a retraction.   
\end{theorem}

This flexible definition of complex-orientability results in a straightforward definition of Wood-types and chromatic defect that makes sense in complete generality. In what follows, we work $p$-locally at a fixed prime $p$. We say a finite spectrum $F$ is a finite $\mathrm{BP}$-projective if $\mathrm{BP}_*F$ is a projective $\mathrm{BP}_*$-module.

\begin{definition}[\ref{def:wood}, \ref{def:defect}]\label{def:introwood}
We say a spectrum $E$ is \emph{Wood-type} if there exists a finite $\mathrm{BP}$-projective $F$ such that $E\otimes F$ is complex-orientable. 
\end{definition}

\begin{definition}[\ref{def:defect}]
    We say a spectrum $E$ has \emph{chromatic defect} $\le n$ if $E\otimes \mathrm{X(n)}$ is complex-orientable. In this case, we use the notation $\Phi(E)\le n$.
\end{definition}

We tie the notion of Wood-types to that of chromatic defect with the following result.

\begin{theorem}[\ref{cor:woodimpliesfinitedefect}]\label{thm:introwoodimpliesdefect}
    If $E$ is Wood-type, then it has finite chromatic defect.
\end{theorem}

It can be difficult to say whether a certain spectrum $E$ is Wood-type since this requires producing a complex $F$ such that $E\otimes F$ is complex-orientable. We will see, however, that computing chromatic defect is tractable in many cases, and when it is infinite for example, this rules out the possibility of being Wood-type, see \ref{cor:finitespectradefect} and \ref{cor:jnotwoodtype} for example.

\subsubsection*{Section 3} The mod $p$ homology $H_*(\mathrm{X(n)};\F_p)$ is determined as an $\mathcal{A}_*$-comodule algebra via the Thom isomorphism. We use this to describe the Adams $E_2$-page of $\mathrm{X(n)}$ and give a May spectral sequence converging to it. This recovers known results about the homotopy groups of $\mathrm{X(n)}$, such as that the bottom odd-dimensional homotopy group of $\mathrm{X(n)}_{(p)}$ is
\[\pi_{2p^{m+1}-3}\mathrm{X(n)}_{(p)}=\Z/p\{\chi_{m+1}\}\]
where $m=\lfloor \log_p(n)\rfloor$ (see also \cite{beardsley} for this result). The classes $\chi_{m+1}$ are shown to give a sequence of obstructions for a homotopy associative ring spectrum $E$ to have chromatic defect $\le n$.

\begin{theorem}[\ref{prop:chi}]\label{prop:introchi}
    Let $E$ be a $p$-local homotopy associative ring spectrum. Then $\Phi(E)\le n$ if and only if 
    \[\chi_{m+1}=0\in\pi_{2p^{m+1}-3}E\otimes \mathrm{X(p^m)}\]
    for all $m\ge \lfloor \log_p(n)\rfloor$.
\end{theorem}

We calculate the relative dual Steenrod algebra $\mathcal{A}^{\mathrm{X(n)}_*}:=\pi_*(\F_p\otimes_{\mathrm{X(n)}}\F_p)$, and this allows us to show that finite spectra are too close to the sphere to have finite chromatic defect.

\begin{corollary}[\ref{cor:finitespectradefect}]
    If $F$ is a finite spectrum, then $\Phi(F)=\infty$.
\end{corollary}

Set $m:=\lfloor \log_p(n)\rfloor$. We show that, when the $\mathcal{A}_*$-comodule $H_*(\mathrm{X(n)};\F_p)$ is restricted to $\mathcal{A}(m)_*$, it splits as a sum of finite even comodules, all of which are isomorphic to $\mathcal{P}(m-1)_*$ (see Section \ref{sec:3} for these Hopf algebras). With an eye toward Wood equivalences, we thus construct finite complexes $F$ that are free over $\mathcal{P}(m-1)$. This uses the idempotent construction of Jeff Smith \cite{ravorange}, used to construct finite $\mathcal{A}(m)$-free type $m+1$ complexes, and our proof is essentially the same.

\subsubsection*{Section 4} The spectra $\mathrm{ko}$ and $\mathrm{tmf}$ are our prototypical examples of Wood-types, the former via $C(\eta)$ in the classical Wood equivalence, and the latter (at the prime 2) via the complex $D\mathcal{A}(1)$ (see \cite{akhiltmf}). These are also prototypical examples of fp spectra in the sense of Mahowald--Rezk \cite{mr}, that is, a $p$-complete bounded below spectrum $E$ that admits an isomorphism
\[H_*(E;\F_p)\cong\mathcal{A}_*\Boxover{\mathcal{A}(n)_*}M\]
for some $n<\infty$ and $M$ a finite $\mathcal{A}(n)_*$-comodule.

We show that the Wood equivalences for $\mathrm{ko}$ and $\mathrm{tmf}$ actually follow from certain evenness conditions that hold for a much broader class of fp spectra. This is obtained by using the Adams spectral sequence to sharpen the obstruction statement of Proposition \ref{prop:introchi} to an algebraic context. We use the $\mathcal{P}(n)$-free complexes of Section \ref{sec:3} to obtain the following. Here, we use the following standard notation for certain exterior quotient Hopf algebras of the dual Steenrod algebra.
\[\mathcal{E}(n)_*:=\begin{cases}
    E(\xi_1,\ldots,\xi_{n+1})&p=2\\E(\tau_0,\ldots,\tau_n)&p>2
\end{cases}\]

\begin{theorem}[\ref{cor:fpeven}]\label{thm:intothmobstructions}
    Let $E$ be a homotopy associative ring spectrum that is an fp spectrum, so that $H_*E\cong\mathcal{A}_*\Boxover{\mathcal{A}(n)_*}M$, and suppose that
    \[\Ext^{s,t}_{\mathcal{E}(n)_*}(\F_p,M)\]
    is concentrated in even stems $t-s$. Then $\Phi(E)\le p^n$ and $E$ is Wood-type.
\end{theorem}

Theorem \ref{thm:intothmobstructions} is useful because computing $\Ext$ over the exterior Hopf algebra $\mathcal{E}(n)_*$ is much more tractable than over $\mathcal{A}(n)_*$. We actually prove more specific statements by tracking the degrees of the obstructions $\chi_{m+1}$. When $E$ as above is assumed to have finite chromatic defect, we can prove a partial converse to Theorem \ref{thm:introwoodimpliesdefect} that does not necessarily require evenness conditions.

\begin{notation}
    For a spectrum $X$, we let $\mathrm{ASS}(X)$ and $\mathrm{ANSS}(X)$ denote the Adams and Adams--Novikov spectral sequences of $X$, respectively.
\end{notation}

\begin{theorem}[\ref{thm:fpwoodsplittings}]
    Let $E$ be a homotopy associative ring spectrum that is an fp spectrum, and suppose $E$ has finite chromatic defect. If $\mathrm{ASS}(E\otimes \mathrm{BP})$ collapses on the $E_2$-page, then $E$ is Wood-type.
\end{theorem}

Giving ourselves the mod 2 homology of $\mathrm{ko}$ and $\mathrm{tmf}$, this gives a new proof of the following result of Hopkins \cite{tmfbook}.  

\begin{corollary}[\ref{example:kotmfwoodanddefect}]
    The spectra $\mathrm{ko}$ and $\mathrm{tmf}$ are Wood-type. The chromatic defect of $\mathrm{ko}$ is 2, and the chromatic defect of $\mathrm{tmf}$ is 4.
\end{corollary}

\begin{remark}
    The spectra $\mathrm{ko}$ and $\mathrm{tmf}$ are good connective models of certain Hopkins--Miller higher real $K$-theories (see \cite[Section 1.1]{chr} for a discussion of this). In \cite{BHSZ}, Beaudry--Hill--Shi--Zeng construct good connective models of the spectra $\mathrm{EO}_{2^{n-1}m}(C_{2^n})$ known as the $\mathrm{BP}^{((G))}\langle m\rangle$'s. In joint work with Mike Hill, we showed that these are fp spectra \cite{eohpaper}. The determination of their chromatic defect can thus be approached via the methods of Section \ref{sec:4}, and we intend to return to this in future work.
\end{remark}
 
\subsubsection*{Section 5} When $E$ has finite chromatic defect, the $E_2$-page of its Adams--Novikov spectral sequence simplifies substantially. For example, the $E_2$-page in general is determined by cohomology over the group scheme $\Spec(\Z[b_1,b_2,\ldots])$ corepresenting power series $x+b_1x^2+\cdots$. However, when $E$ has chromatic defect $\le n$, the $E_2$-page is determined by cohomology over the subgroup scheme $\Spec(\Z[b_1,b_2,\ldots,b_{n-1}])$, which has finite Krull dimension. Hopkins' stacks framework allows for these sorts of change-of-rings isomorphisms to take a more conceptual form. We use this to verify our ``finite subgroups'' philosophy on chromatic defect in the following way.

\begin{proposition}[\ref{prop:defectimpliesfiniteaut}]
    Let $E$ be a homotopy commutative ring spectrum with finite chromatic defect. Then, for any algebraically closed field $k$ and any $x\in\mathcal{M}_E(k)$ such that $p_E(x)\in \mathcal{M}_{FG}(1)(k)$ has finite height, the image of the homomorphism
    \[\psi:\mathrm{Aut}_{\mathcal{M}_E(k)}(x)\to \mathrm{Aut}_{\mathcal{M}_{FG}(1)(k)}(p_E(x))\]
    is a finite subgroup.
\end{proposition}

In other words, at finite height geometric points, the morphism $p_E:\mathcal{M}_E\to\mathcal{M}_{FG}(1)$ lands in finite subgroups. We provide various converses to this when $\mathcal{M}_E$ has certain strong finiteness conditions. These conditions hold for example for the $\mathrm{EO}_n(G)$'s and also for $\mathrm{ko}$ and $\mathrm{tmf}$. In the following, we suppose $E\otimes \mathrm{X(n)}$ is $\mathrm{MU}$ nilpotent-complete for all $n$ sufficiently large, for example if $E$ is connective or $\mathrm{MU}$-nilpotent.

\begin{proposition}[\ref{prop:finiteautimpliesdefect}]
    Let $E$ be a homotopy commutative ring spectrum, and suppose there is a faithfully flat, finite morphism $\Spec(R)\to \mathcal{M}_E$, for some Noetherian commutative ring $R$. Then $E$ has finite chromatic defect. 
\end{proposition}

Finally, we move our stacks analysis to the $\mathrm{K(n)}$-local category to compute the chromatic defect of the $K(n)$-local sphere.

\begin{theorem}[\ref{thm:K(n)localsphere}]
    The spectrum $L_{K(n)}\mathbb{S}$ does not have finite chromatic defect.
\end{theorem}

\begin{remark}
   Let $\mathrm{j}$ be the connective image of $\mathrm{J}$ spectrum. Since
   \[L_{\mathrm{K(1)}}\mathrm{j}\simeq L_{\mathrm{K(1)}}\mathbb S,\] it follows also that $\mathrm{j}$ has infinite chromatic defect (see Theorem \ref{thm:defectj}). This is an important example because $\mathrm{j}$ is an fp spectrum, and in fact $\mathrm{MU}$-nilpotent, but it does not have finite chromatic defect and thus is not Wood-type.
\end{remark}

\subsubsection*{Section 6} We study chromatic defect in two examples where a ring spectrum $E$ admits an equivalence $\mathcal{M}_E\simeq\Spec(R)/G$ to a quotient stack by a finite group. Working at the prime 2, let $\mathrm{E(n)}$ denote the height $n$ Johnson--Wilson theory. The spectrum $\mathrm{E(n)}$ admits a natural lift to a $C_2$-spectrum called $\mathrm{E}_{\mathbb R}(n)$, constructed by Hu--Kriz \cite{HK}, where $C_2$ acts by complex-conjugation. The fixed point spectra $\mathrm{ER(n)}:=\mathrm{E}_{\mathbb R}(n)^{hC_2}$ have been studied extensively by Kitchloo--Wilson; they used these theories to prove new nonimmersion results for real projective spaces \cite{KW1} \cite{KW2}.

\begin{theorem}[\ref{thm:defectern}]
    The chromatic defect of $\mathrm{ER(n)}$ is $2^n$.
\end{theorem}

This generalizes the result of Hopkins at height 1 \cite{tmfbook}, which states that $\mathrm{KO}\otimes \mathrm{X(2)}$ is complex-orientable, and it is a result of Atiyah \cite{atiyah} that $\mathrm{ER(1)}\simeq \mathrm{KO}$. At higher heights, some care is needed since $\mathrm{ER(n)}$ is not known to admit a ring structure (see \cite{kitchloo}). To compute the chromatic defect of $\mathrm{ER(n)}$, we thus prove the following, which may be of independent interest. 

\begin{theorem}[\ref{thm:defectern}]
    For all $m\ge 2^n$, the spectrum $\mathrm{ER(n)}\otimes \mathrm{X(m)}$ has a (Landweber exact) homotopy commutative ring structure with respect to which the restriction map $\mathrm{ER(n)}\otimes \mathrm{X(m)}\to \mathrm{E(n)}\otimes \mathrm{X(m)}$ is one of ring spectra.
\end{theorem}

We turn to the Hopkins--Miller theories. Fixing as before a height $n$ formal group $\Gamma$ over a perfect field $k$ of characteristic $p$, we let $G$ be a finite subgroup of $\mathbb G_n=\mathrm{Gal}(k/\F_p)\rtimes\mathrm{Aut}(\Gamma)$ the corresponding Morava stabilizer group. By the Goerss--Hopkins--Miller theorem \cite{gh04}, $G$ acts on $\mathrm{E}(k,\Gamma)$ by $\mathbb E_\infty$-ring maps. The chromatic defect of the fixed points $\mathrm{EO}_n(G)$ will be expressed in terms of the standard valuation on the endomorphism ring $\mathrm{End}(\Gamma)$, which we recall here. Every nonzero endomorphism $f(x)$ of $\Gamma$ may be expressed uniquely as $f(x)=g(x^{p^k})$, where $g(x)\in \mathrm{End}(\Gamma)^\times$ is invertible. The valuation $\nu$ is then defined by $\nu(f)=k/n$,
normalized by the height $n$, so that $p\in\mathrm{End}(\Gamma)$ is a uniformizer. In the statement below, we let $\pi:\mathbb G_n\to \mathrm{Aut}(\Gamma)$ denote the projection map, where we identify the underlying set of the semidirect product as a cartesian product. 

\begin{theorem}[\ref{thm:defecteon}]
    Let $N(G):=n\cdot\max\{\nu(\pi(g)-1):e\neq \pi(g)\}_{g\in G}$, where $e$ is the identity element of $G$. The chromatic defect of $\mathrm{EO}_n(G)$ is $p^{N(G)}$.
\end{theorem}

When the height $h$ is of the form $p^{n-1}(p-1)m$, there is a $C_{p^n}$-subgroup of $\mathbb G_h$ (see Section \ref{sec:6} for more details), and we have the following in this case.

\begin{example}[\ref{cor:eonexample1}, \ref{cor:eonexample2}]
    The chromatic defect of $\mathrm{EO}_{p^{n-1}(p-1)m}(C_{p^n})$ is $p^{p^{n-1}m}$.
\end{example}

To apply Ravenel's method of descent to compute $\pi_*\mathrm{EO}_n(G)$ as in \cite{hhreo}, for example, the above theorem implies that the computation begins by tensoring with $\mathrm{X(p^{N(G)})}$.

\begin{remark}
    In fact, our methods also imply that the $\mathrm{MU}$-homology of the $\mathrm{ER(n)}$'s and the $\mathrm{EO}_n(G)$'s is even and torsion-free. This implies these theories are \emph{quasi-syntomic} in the sense of Hahn--Raksit--Wilson \cite{hahnraksitwilson}.
\end{remark}

\subsubsection*{Section 7} One can conceive of variations on our definition of Wood-types $E$, for example by asking the complex $F$ to be just a type zero instead of a $\mathrm{BP}$-projective, or simply the condition of $E$ being $\mathrm{BP}$-nilpotent. Neither of these imply finite chromatic defect, but they do guarantee strong conditions on the Adams--Novikov spectral sequence of $E$, such as a horizontal vanishing line on a finite page. The Adams--Novikov spectral sequence of a Wood-type, however, is further restricted by way of a spectral sequence we introduce called the \emph{$\Z$-indexed Adams--Novikov spectral sequence} of a Wood-type. The relationship between the ANSS and the $\Z\mathrm{-ANSS}$ of a Wood-type is very similar to that of the HFPSS and the Tate SS for a $G$-spectrum $E$ with vanishing Tate spectrum $E^{tG}$.

Mahowald--Rezk introduced a $\Z$-indexed Adams spectral sequence for an fp spectrum $E$, which extends its Adams spectral sequence to a full plane spectral sequence converging to $\pi_*L_n^fE$ when $E$ is fp type $n$. Our definition of Wood-types is chosen in part to make their construction work instead with the Adams--Novikov spectral sequence. Rather than $L_n^fE$, the $\Z\mathrm{-ANSS}$ of a Wood-type always converges to zero.

\begin{theorem}[\ref{thm:ZANSS}]
    The $\Z\mathrm{-ANSS}$ of a Wood-type $E$ has the following properties:
    \begin{enumerate}
        \item The $\Z\mathrm{-ANSS}$ is independent of the choice of finite $\mathrm{BP}$-projective $F$ from the $E_2$-page on.
        \item The natural map
    \[\mathrm{ANSS}(E)\to\Z\mathrm{-ANSS}(E)\]
    is an isomorphism on $E_2$-pages in positive filtrations and an epi in filtration zero. 
    \item There is a one-to-one correspondence along the map
    \[\mathrm{ANSS}(E)\to\Z\mathrm{-ANSS}(E)\]
    of differentials whose source is in nonnegative filtrations.
    \item The $\Z\mathrm{-ANSS}$ converges to zero.
    \end{enumerate} 
\end{theorem}

This can be very useful for example for determining differentials and vanishing lines in the ANSS of a Wood-type (cf. \cite{dannyvanishing}). We run the ANSS and $\Z\mathrm{-ANSS}$ in detail for $\mathrm{ko}$ and use this to deduce the famous $d_3$ therein. We also give various descriptions of the $E_2$-page of the $\Z\mathrm{-ANSS}$ of a Wood-type in terms of Tate cohomology in $\mathrm{BP}_*\mathrm{BP}$-comodules.

\subsection{Some numerology and questions} For a $p$-local homotopy commutative ring spectrum $E$, there are  many numerical measures of $E$ coming from chromatic homotopy that are often related. One has, for instance,

\begin{enumerate}
    \item The chromatic height of $E$:
    \[\mathrm{ht}(E)=\min\{n\ge0\st E\simeq L_nE\}\]
    \item The $\mathrm{BP}$-nilpotence exponent of $E$ (see \cite{mnn}).
    \item The chromatic defect $\Phi(E)$ of $E$.
    \item For a Wood-type $E$, the minimum $n\ge1$ such that there exists a finite $\mathrm{X(n)}$-projective $F$ such that $E\otimes F$ is complex-orientable.
    \item The orientation order of Bhattacharya--Chatham \cite{hoodprasiteo}:
    \[\Theta(E)=\min\{n\ge1\st \xi^{\oplus n} \text{ is }E\text{-orientable}\} \]
    where $\xi$ is the tautological line bundle on $\CP^\infty$.
\end{enumerate}

At first glance, chromatic height does not fit so well in this list: if $E$ is complex-orientable, then all the quantities (2)-(5) are equal to 1 and thus are not sensitive to the chromatic height of $E$. However, when $E$ is the fixed points of a complex--oriented theory $R$, the quantities (2)-(5) will often be functions of the chromatic height of $R$, as in our Theorems \ref{thm:defecteon} and \ref{thm:defectern}
 for example. In fact, our results suggest a heuristic that, in many cases, the chromatic defect of $R^{hG}$ should be $\le p^{\mathrm{ht}(R)/(p-1)}$.

For $\mathrm{ko}$ and $\mathrm{tmf}$ at the prime 2, it can be shown that all the quantities (3)-(5) coincide and they are a lower bound for (2). In general, when $E$ is a Wood-type, (4) is an upper bound for (3) by the proof of \ref{cor:woodimpliesfinitedefect}.

\begin{question}
    Suppose $E$ has finite chromatic defect. Is $\Phi(E)$ a lower bound for the $\mathrm{BP}$-nilpotence exponent of $E$?
\end{question}

\begin{question}
    Does there exist a spectrum with finite chromatic defect that is not Wood type?
\end{question}

When $E$ does not have finite chromatic defect, and is therefore not Wood-type, the quantities above don't seem to be closely related. For instance, $\Phi(\mathrm{j})=\infty$, but $\mathrm{j}$ is $\mathrm{BP}$-nilpotent and therefore has finite $\mathrm{BP}$-nilpotence exponent.\\

There are of course many other numerical quantities one may assign to $E$. One may speculate that various chromatic height-shifting phenomena also result in shifts of chromatic defect in certain cases.

\begin{question}
    How does chromatic defect interact with chromatic redshift?
\end{question}

For instance, the computations of Angelini-Knoll--Ausoni--Rognes in \cite{kko} suggest that while $\Phi(\mathrm{ko})=2$, it may be the case that $\Phi(\mathrm{K}(\mathrm{ko}))=4$. Strong forms of the Ausoni--Rognes chromatic redshift conjecture state that the algebraic $\mathrm{K}$-theory of an $\mathbb E_1$-ring spectrum of fp type $n$ is fp of type $n+1$; the methods in Section \ref{sec:4} may then be useful for this question.\\

We finish by remarking that the quantities (3) and (5) coincide for $\mathrm{ER(n)}$, by our Theorem \ref{thm:defectern} and work of Kitchloo--Wilson \cite{KW1}. Tying these quantities together for $\mathrm{EO}_n$-theories would allow for our Theorem \ref{thm:defecteon} to shed light on conjectures of Bhattacharya--Chatham on the orientation orders of $\mathrm{EO}_n$-theories (see for example \cite[Conjecture 1.13]{hoodprasiteo}). The author investigates this in work in progress with Prasit Bhattacharya and Yang Hu.

\subsection{Acknowledgments}
We would like to thank the anonymous referee for their thorough reading and many comments and suggestions, which substantially improved this paper.

We would also like to acknowledge helpful conversations with and comments from Mike Hill, Lennart Meier, Jack Davies, Shaul Barkan, Ishan Levy, and Robert Burklund. We thank Guy Boyde for suggesting the terminology ``chromatic defect''.

\section{Orientability and Ravenel's \texorpdfstring{$\mathrm{X(n)}$}{X(n)}}\label{sec:2}
In this section, we revisit the notion of a complex-orientable spectrum and extend the usual definition for ring spectra to all spectra. The resulting class of complex-orientable spectra is in particular closed under tensoring with an arbitrary spectrum and under taking retracts. This gives us flexible notions of Wood-types and chromatic defect. 

\subsection{Complex-orientability}\label{sec:complexorientations}

\begin{definition}
    Let $\sigma_k$ be the map
    \[S^{2k+1}\to S^{2k+1}/S^1\cong\CP^k\]
    where $S^{2k+1}\subset\mathbb C^{k+1}$ is given its usual $\mathbb{S}^1\subset\mathbb C^\times$ action. 
\end{definition}

\begin{remark}
    The cofiber of $\sigma_k$ is $\CP^{k+1}$. In other words, $\sigma_k$ is the attaching map for the top cell in $\CP^{k+1}$. For simplicity, since we work stably, we will not distinguish between $\sigma_k$ and its double desuspension $\Sigma^{-2}\sigma_k:\mathbb S^{2k-1}\to\Sigma^{-2}\CP^k$.
\end{remark}

\begin{definition}\label{def:comporient}
    We say a spectrum $E$ is complex-orientable if the map \[E\otimes\sigma_k:E\otimes \mathbb S^{2k+1}\to E\otimes\CP^k\]
    is nullhomotopic for all $k\ge 1$.
\end{definition}

As we will see, when $E$ is a homotopy associative ring spectrum, this is equivalent to the classical definition of complex orientability, as in \cite{ravgreen}. The above definition is much more flexible, however, as it does not require $E$ to be a ring spectrum. Moreover, it is closed under taking retracts and tensoring with an arbitrary spectrum. To be in line with this flexibility, we will work with weak forms of modules in the homotopy category, so we take care now to define our terms.

\begin{definition}
    We say $E$ is a \textit{homotopy associative/commutative ring spectrum} if it is an associative/commutative monoid in the symmetric monoidal 1-category $\mathrm{Ho}(\Sp)$. A (left/right) \textit{homotopy $E$-module} is a (left/right) $E$-module in $\mathrm{Ho}(\Sp)$.
\end{definition}

\begin{definition}\label{def:weakmodule}
    An \textit{$\mathbb{E}_0$-algebra} in $\Sp$ is the data of a map of spectra $\eta_E:\mathbb{S}\to E$. For an $\mathbb{E}_0$-algebra $E$, a \textit{weak $E$-module} is a spectrum $M$ such that there exists $m:E\otimes M\to M$ making the following diagram commute
    \[
    \begin{tikzcd}
        M\arrow[r,"\eta_E"]\arrow[dr,equal]&E\otimes M\arrow[d,"m"]\\
        &M
    \end{tikzcd}
    \]
    up to homotopy.
\end{definition}

\begin{remark}
    $M$ is a weak left $E$-module if and only if it is a weak right $E$-module, hence we will just say weak $E$-module.
\end{remark}

\begin{proposition}
    If $E$ is an $\mathbb{E}_0$-algebra that is complex-orientable in the sense of Definition \ref{def:comporient}, then the unit map $\eta_E:\mathbb{S}\to E$ extends to a map $\Sigma^{-2}\CP^\infty\to E$. Conversely, if $E$ is a homotopy associative ring spectrum, and $\eta_E$ extends over $\Sigma^{-2}\CP^\infty$, then $E$ is complex-orientable.
\end{proposition}
\begin{proof}
    If $E\otimes \sigma_k=0$ for all $k\ge1$, then there is a splitting 
    \[E\otimes\Sigma^{-2}\CP^\infty\simeq\bigoplus\limits_{k\ge0}\Sigma^{2k}E\]
     indexed by the cells $\{b_0,b_1,\ldots\}$ of $\Sigma^{-2}\CP^\infty$. The composite
     \[\Sigma^{-2}\CP^\infty\xrightarrow{\eta_E}E\otimes\Sigma^{-2}\CP^\infty\simeq \bigoplus\limits_{k\ge0}\Sigma^{2k}E\to E\]
     extends $\eta_E$, where the last map is projection onto $b_0$.

    Conversely, if the unit map of $E$ extends over $\Sigma^{-2}\CP^\infty$, and $E$ is a homotopy associative ring spectrum, then the corresponding class $x\in\pi_{-2}F(\CP^{k}_+,E)$ defines a map
    \[\bigoplus\limits_{0\le n\le {k}}\Sigma^{-2n}E\to F(\CP^{k}_+,E)\]
    which, on the $n$-th component is given by
    \[\Sigma^{-2n}E\xrightarrow{x^n}E\otimes F(\CP^{k}_+,E)\to F(\CP^{k}_+,E)\]
    using the multiplication on $E$ to define both maps in the sequence. Filtering the lefthand side in $n$ and the righthand side via the cellular filtration of $\CP^{k}$, it suffices to show the given map induces an equivalence upon taking associated graded. This follows from the fact that $H^*(\CP^{k};\Z)\cong\Z[x]/x^{k+1}$, see for example \cite[Lecture 4, Proposition 7]{lurielecture}. This implies in particular that $E\otimes \CP^{k-1}\to E\otimes\CP^{k}$ admits a splitting, so that $E\otimes\sigma_{k-1}$ is null.
\end{proof}

It is a classical result of Quillen that a homotopy commutative ring spectrum $E$ is complex-orientable if and only if there is a homotopy ring map $\mathrm{MU}\to E$. The proof breaks down if $E$ is not homotopy commutative as the rings $E^*(\mathrm{BU}(n))$ will not necessarily be commutative. This prevents one from constructing a system of Thom classes $u_n\in E^{2n}(\mathrm{MU}(n))$ for the universal bundles that is compatible with the tensor product of bundles. We have the following replacement in the associative case.

\begin{proposition}\label{prop:orientableringisweakmodule}
    Let $E$ be a complex-orientable homotopy associative ring spectrum. Then $E$ is a weak $\mathrm{MU}$-module.
\end{proposition}
\begin{proof}
    The differentials in the Atiyah-Hirzebruch spectral sequence (AHSS)
    \[E_2=H_*(\Sigma^{-2}\CP^\infty;E_*)\implies E_*(\Sigma^{-2}\CP^\infty)\]
    come from the boundary maps in the cellular filtration for $\CP^\infty$. Definition \ref{def:comporient} assumes that these become zero after tensoring with $E$, so this AHSS collapses on $E_2$, and the universal coefficients theorem computes the $E_2$-page
    \[H_*(\Sigma^{-2}\CP^\infty;E_*)\cong E_*\{b_0,b_1,\ldots\}\]
    as a left $E_*$-module.

    The map $\Sigma^{-2}\CP^\infty\to \mathrm{MU}$ induces a map of AHSS's, which takes the form
    \[E_*\otimes H_*(\Sigma^{-2}\CP^\infty;\Z)\to E_*\otimes H_*(\mathrm{MU};\Z)\]
    on the $E_2$ page, as maps of left $E_*$-modules. The map $H_*(\Sigma^{-2}\CP^\infty;\Z)\to H_*(\mathrm{MU};\Z)$ exhibits the latter as 
    \[H_*(\mathrm{MU};\Z)=\mathrm{Sym}(H_*(\Sigma^{-2}\CP^\infty;\Z))/(b_0-1)=\Z[b_1,b_2,\ldots]\]
    Since $E$ is a homotopy associative ring spectrum, the $E$-based AHSS for $\mathrm{MU}$ is a multiplicative spectral sequence. It is also a spectral sequence of $E_*$-modules, so by use of the Leibniz rule, we see it must collapse on $E_2$ as the $b_i$'s are permanent cycles.
    
    Let $M$ be a monomial basis for $\Z[b_1,b_2,\ldots]$ and consider the map of homotopy left $E$-modules
    \[E\{M\}\to E\otimes \mathrm{MU}\]
    given by lifting the $b_i$'s and using the multiplication in $E\otimes \mathrm{MU}$. Giving elements of $M$ their corresponding AHSS filtration, this map becomes a map of filtered left homotopy $E$-modules. Since the AHSS for $E\otimes \mathrm{MU}$ collapses, the map on associated graded is an equivalence, so the map is an equivalence since the filtrations are bounded below and exhaustive. Now the map
    \[m:E\otimes \mathrm{MU}\simeq E\{M\}\to E\]
    given by projection onto the summand indexed by $1\in M$ exhibits $E$ as a weak $\mathrm{MU}$-module.
\end{proof}

    
    

\begin{theorem}\label{thm:weakMU}
    A spectrum $E$ is complex-orientable in the sense of Definition \ref{def:comporient} if and only if $E$ is a weak $\mathrm{MU}$-module.
\end{theorem}
\begin{proof}
    If $E$ is complex-orientable, then so is $\mathrm{End}(E)=F(E,E)$ since
    \[F(E,E)\otimes \sigma_k\simeq F(E,E\otimes\sigma_k)\]
    using that the domain and codomain of $\sigma_k$ are dualizable. By Proposition \ref{prop:orientableringisweakmodule}, $\mathrm{End}(E)$ is a weak homotopy $\mathrm{MU}$-module. $E$ is naturally an $\mathrm{End}(E)$-module, and in particular a weak $\mathrm{End}(E)$-module.

    The only if direction of the claim then follows from the following observation. If $\eta_R:\mathbb{S}\to R$ and $\eta_T:\mathbb{S}\to T$ are $\mathbb{E}_0$-algebras, and $M$ is a weak $R$-module and $R$ is a weak $T$-module, then $M$ is a weak $T$-module. In the following diagram
    \[
    \begin{tikzcd}
        M\arrow[r,"\eta_T\otimes 1"]\arrow[d,"\eta_R\otimes 1"]&T\otimes M\arrow[d,"1\otimes\eta_R\otimes 1"]\\
        R\otimes M\arrow[r,"\eta_T\otimes1\otimes1"]\arrow[dr,equal]&T\otimes R\otimes M\arrow[d,"m_R\otimes 1"]\\
        &R\otimes M\arrow[d,"m_M"]\\
        &M
    \end{tikzcd}
    \]
    the clockwise composite exhibits $M$ as a weak $T$-module, where $m_R$ is a weak $T$-module structure on $R$ and $m_M$ is a weak $R$-module structure on $M$.

    Conversely, if $E$ is a weak $\mathrm{MU}$-module, then the map $E\otimes\sigma_k$ is a retract of the map $MU\otimes E\otimes\sigma_k$, which is null as $\mathrm{MU}$ is complex-orientable.
\end{proof}

\begin{corollary}\label{cor:complexorientansscollapse}
    If $E$ is complex-orientable, the Adams-Novikov spectral sequence of $E$ collapses on the zero-line, i.e. it is concentrated in filtration zero from the $E_2$-page on.
\end{corollary}
\begin{proof}
    The ANSS of $E$ comes from the coaugmented cosimplicial spectrum 
\[E\to \mathrm{MU}\otimes E\implies \mathrm{MU}\otimes \mathrm{MU}\otimes E\Rrightarrow \mathrm{MU}\otimes\mathrm{MU}\otimes\mathrm{MU}\otimes E\cdots\]
    If $E$ is a weak $\mathrm{MU}$-module, this admits a $(-1)$-st codegeneracy, so that it is equivalent to the constant cosimplicial object at $E$.
\end{proof}

\begin{remark}
  If $E$ is a weak $\mathrm{MU}$-module, it is in particular $\mathrm{MU}$-nilpotent, i.e. $E$ is in the thick tensor ideal of $\Sp$ generated by $\mathrm{MU}$. In fact, $E$ is a weak $\mathrm{MU}$-module if and only if $E$ is $\mathrm{MU}$-nilpotent of exponent 1, in the sense of \cite[Part 1, Section 4]{mnn}. As we will discuss in Section \ref{sec:7}, this is also equivalent to asking $E$ to be an $\mathrm{MU}$-injective in the sense of $\mathrm{MU}$-based Adams resolutions, which gives another proof of the preceding corollary. 
\end{remark}


\subsection{\texorpdfstring{$\mathrm{X(n)}$}{X(n)}-orientations and chromatic defect} Theorem \ref{thm:weakMU} tells us that the condition of being a weak homotopy $\mathrm{MU}$-module is determined by the attaching maps for $\CP^\infty$. We may ask if there is a spectrum playing a similar role for $\CP^n$, i.e. for only the first $n-1$ attaching maps of $\CP^\infty$. This is exactly Ravenel's Thom spectrum $\mathrm{X(n)}$, which we now introduce. For a more thorough introduction to the $\mathrm{X(n)}$'s, we strongly recommend Hopkins' thesis \cite{hopthesis}.

\begin{definition}
    For $n\ge 1$, the spectrum $\mathrm{X(n)}$ is the Thom spectrum
    \[\mathrm{Thom}(\Omega \mathrm{SU(n)}\to \Omega \mathrm{SU}\simeq \mathrm{BU})\]
    where the equivalence $\Omega \mathrm{SU}\simeq \mathrm{BU}$ is Bott periodicity.
\end{definition}

\begin{remark}
    The equivalence $\Omega \mathrm{SU}\simeq \mathrm{BU}$ may be chosen to be a double loop map, so that $\mathrm{X(n)}$ acquires the structure of an $\mathbb{E}_2$-algebra in $\Sp$. It is known that $\mathrm{X(n)}$ is not $\mathbb{E}_3$ for $n>1$ \cite[Example 5.31]{lawsonen}.
\end{remark}

The space $\Omega \mathrm{SU(n)}$ admits a cell structure with even cells and $\CP^{n-1}$ as a subcomplex. Moreover, a Serre spectral sequence computation shows that
\[H_*(\Omega \mathrm{SU(n)};\Z)\cong\mathrm{Sym}(\tilde{H}_*(\CP^{n-1};\Z))=\Z[b_1,\ldots,b_{n-1}]\]
as a ring, with its product from the double loop space structure on $\Omega \mathrm{SU(n)}$. The composite $\CP^{n-1}\to\Omega \mathrm{SU(n)}\to \mathrm{BU}$ classifies $L-1$, where $L$ is the tautological line bundle on $\CP^{n-1}$, and the Thom spectrum of $L-1$ is $\Sigma^{-2}\CP^n$. Applying the Thom isomorphism, one finds that 
\[H_*(\mathrm{X(n)};\Z)\cong \mathrm{Sym}(H_*(\Sigma^{-2}\CP^{n};\Z))/(b_0-1)=\Z[b_1,\ldots,b_{n-1}]\]
as a ring. With these facts, the proofs of the previous subsection now may be repeated to show the following.

\begin{proposition}\label{prop:weakX(n)}
    A spectrum $E$ is a weak $\mathrm{X(n)}$-module if and only if $E\otimes\sigma_k=0$ for $1\le k\le n-1$.
\end{proposition}

The primary interest of this paper is to investigate spectra that are -- in some sense -- only finitely many steps away from being complex-orientable. Phrased as in Definition \ref{def:comporient}, we are interested in spectra $E$ that become complex-orientable after forcing finitely many of the maps $E\otimes\sigma_k$ to be null, say for $k<n$. Proposition \ref{prop:weakX(n)} tells us in some sense that, up to taking retracts, this passes through forming the tensor product $E\otimes\mathrm{X(n)}$. We thus make the following definition.

\begin{definition}\label{def:defect}
    A spectrum $E$ has \textit{chromatic defect} $\Phi(X)\le n$ if $E\otimes\mathrm{X(n)}$ is complex-orientable. We say $\Phi(X)=n$ if $n$ is the smallest positive integer with this property.
\end{definition}

\begin{remark}\label{rmk:defectupward}
    The inequality appearing in the above definition is justified in the following sense. If $E\otimes\mathrm{X(n)}$ is complex-orientable, then so is $E\otimes \mathrm{X(m)}$ for any $m\ge n$. This follows from the definition of complex-orientability as $\mathrm{X(m)}$ is an $\mathrm{X(n)}$-module. In fact, the same argument shows that if $R\to S$ is a map of homotopy associative ring spectra, then $\Phi(R)\ge \Phi(S)$
\end{remark}

\begin{example}\label{example:kotmfdefect}
    We have $\mathrm{X(1)}=\mathbb{S}$, so $E$ has chromatic defect $\Phi(E)=1$ if and only if $E$ is complex-orientable. We give two nontrivial examples due to Hopkins \cite[Chapter 9]{tmfbook}
    \begin{itemize}
        \item $\Phi(\mathrm{ko})=\Phi(\mathrm{KO})=2$
        \item $\Phi(\mathrm{tmf})=\Phi(\mathrm{Tmf})=\Phi(\mathrm{TMF})=4$
    \end{itemize}
    We will return frequently to these examples. The spectra $\mathrm{ko}$ and $\mathrm{tmf}$ are prototypical examples of fp spectra, in the sense of Mahowald--Rezk \cite{mr}, and we will investigate in Section \ref{sec:4} to what extent fp spectra admit finite chromatic defect.
\end{example}

If one localizes at a prime $p$, the spectrum $\mathrm{MU}$ splits
\[\mathrm{MU}_{(p)}\simeq \mathrm{BP}[x_i:i\neq p^k-1]\]
as a sum of shifts of the $p$-primary Brown-Peterson spectrum $\mathrm{BP}$, with $|x_i|=2i$. Just as with $\mathrm{MU}$, maps of homotopy ring spectra $\mathrm{X(m)}\to \mathrm{X(m)}$ are in bijection with the set of polynomials 
\[x+b_1x^2+\cdots+b_{n-1}x^m\]
where $b_i\in\pi_{2i}\mathrm{X(m)}$. By the Hurewicz theorem, the map $\mathrm{X(m)}\to \mathrm{MU}$ induces an isomorphism in $\pi_i$ for $i\le 2m-2$. It follows that the Quillen idempotent $\epsilon$ on $\mathrm{MU}_{(p)}$ that defines $\mathrm{BP}$ restricts to an idempotent homotopy ring map $\epsilon:\mathrm{X(m)}_{(p)}\to \mathrm{X(m)}_{(p)}$.

\begin{definition}\label{def:T(n)}
    We let $\mathrm{T(n)}$ denote the spectrum $\mathrm{X(p^n)}_{(p)}[\epsilon^{-1}]$.
\end{definition}


\begin{remark}
    $\mathrm{T(n)}$ inherits the structure of a homotopy commutative ring spectrum from $\mathrm{X(p^n)}$. Beardsley--Lawson showed in fact that $\mathrm{T(n)}$ is  an $\mathbb{E}_1$-summand of $\mathrm{X(p^n)}$ \cite{beardsleylawson}, but it is not known if it admits more structure (see \cite{quigleyknoll}).
\end{remark}

Just as with $\mathrm{BP}$, $\mathrm{X(m)}$ splits $p$-locally. To describe this splitting, we use the following notation.

\begin{definition}
    Let $E$ be a homotopy associative ring spectrum, and let $\{x_1,x_2,\ldots\}$ be a graded set. We let $E[x_1,x_2,\ldots]$ denote the free $E$-module on the graded set given by the standard monomial basis of the polynomial ring $\Z[x_1,x_2,\ldots]$.
\end{definition}

\begin{proposition}\label{prop:T(n)splitting}
    For $p^n\le m<p^{n+1}$, there is a splitting of $\mathrm{T(n)}$-modules
    \[\mathrm{X(m)}_{(p)}\simeq \mathrm{T(n)}[x_i:i\neq p^k-1,i<m]\]
\end{proposition}
\begin{proof}
    See \cite[6.5.1]{ravgreen}.
\end{proof}

For the same reasons that $\mathrm{BP}$ is often easier to work with than $\mathrm{MU}$, it is often more convenient to work one prime at a time and phrase chromatic defect in terms of the $\mathrm{T(n)}$'s.

\begin{definition}
    For $E$ a $p$-local spectrum, we let
    \[\Phi_p(E)=\min\{n\ge0\st E\otimes \mathrm{T(n)}\text{ is complex-orientable}\}.\]
\end{definition}

In particular, \[\Phi_p(E)=\lfloor\log_p\Phi(E)\rfloor\]
    which follows from the previous proposition.

The filtration of $\mathrm{BP}$ by the $\mathrm{T(n)}$'s is especially nice for making inductive arguments due to the existence of a nice cell structure on $\mathrm{T(n+1)}$ as a module over the $\mathbb{E}_1$-ring $\mathrm{T(n)}$. This was an essential ingredient in the Devinatz--Hopkins--Smith proof of the nilpotence theorem (see \cite[Proposition 1.5]{dhs}). We will also see a more structured description of this filtration due to Beardsley in Theorem \ref{thm:beardsley}.

\begin{construction}\label{const:T(n)cellstructure} One may use the Thom isomorphism to compute $H_*(\mathrm{X(n)};\Z)$, using that $\mathrm{X(n)}$ is the Thom spectrum of a complex bundle. Applying the Quillen idempotent, one may compute that $H_*(\mathrm{T(n)};\Z)=\Z[t_1,\ldots,t_n]$ with $|t_i|=2(p^i-1)$ (see \cite[Corollary 1.3.8]{hopthesis}). 

This allows one to form the following diagram of $\mathrm{T(n)}$-modules
\[ 
\begin{tikzcd}
T(n)\arrow[d,"t_{n+1}^0"]&\Sigma^{|t_{n+1}|}\mathrm{T(n)}\arrow[d,"t_{n+1}^1"]&\Sigma^{2|t_{n+1}|}\mathrm{T(n)}\arrow[d,"t_{n+1}^2"]&\\
T(n+1)\arrow[r]&X_{(1)}\arrow[r]&X_{(2)}\arrow[r]&\cdots
\end{tikzcd}
\]
where $X_{(k)}$ is the cofiber of $t_{n+1}^{k-1}$
, and $t_{n+1}^k$ is the map of $\mathrm{T(n)}$-modules adjoint to the map $\mathbb S^{k|t_{n+1}|}\to X_{(k)}$ provided by the Hurewicz theorem, using that 
\[H_*(X_{(k)};\Z_{(p)})=H_*(\mathrm{T(n)};\Z_{(p)})\{t_{n+1}^m\st m\ge k\}\]
Using the Hurewicz theorem, we see that $X_{(\infty)}=0$, hence defining 
\[X^{(k)}=\mathrm{fib}(\mathrm{T(n+1)}\to X_{(k+1)})\]
we have an exhaustive cell structure
\[X^{(0)}\to X^{(1)}\to\cdots\to X^{(\infty)}=\mathrm{T(n+1)}\]
on the $\mathrm{T(n)}$-module $\mathrm{T(n+1)}$ whose associated graded is the free module $\mathrm{T(n)}[t_{n+1}]$.
\end{construction}

\begin{proposition}\label{prop:T(n)toBPredundancy}
    Let $E$ be any spectrum with $\Phi_p(E)\le n\le m\le\infty$. There is an equivalence of $\mathrm{T(n)}$-modules
    \[E\otimes \mathrm{T(m)}\simeq E\otimes \mathrm{T(n)}[t_{n+1},t_{n+2},\ldots,t_m]\]
\end{proposition}
\begin{proof}
    The decomposition of the $\mathrm{T(n)}$-module follows inductively from the cell structure of Construction \ref{const:T(n)cellstructure} along with the fact that each of the cofiber sequences therein split after tensoring with $\mathrm{BP}$. Indeed, this latter fact follows from the evenness of $\mathrm{BP}\otimes\mathrm{T(n)}$.
\end{proof}

We finish this section with the following conjecture. 

\begin{conjecture}\label{conj}
    Any spectrum with finite chromatic defect satisfies the condition of the telescope conjecture at all heights and primes.
\end{conjecture}

\begin{remark}
    In unpublished work, Robert Burklund proves that for any $k$ and $n$, there is a compact $\mathrm{Tel}(k)$-local spectrum $F_{n,k}$ that is a module over $L_{\mathrm{Tel}(k)}\mathrm{X(n)}$, where $\mathrm{Tel}(k)$ is the $v_k$-telescope on a finite type $k$ complex. This result would give a proof for the above conjecture.

Indeed, if $E$ has finite chromatic defect, then $E$ satisfies the telescope conjecture at height $k$ if and only if  $L_{\mathrm{Tel}(k)}(E\otimes F_{n,k})$ satisfies the telescope conjecture at height $k$, by use of the thick subcategory theorem. Since $F_{n,k}$ is a $L_{\mathrm{Tel}(k)}\mathrm{X(n)}$-module, it follows that $L_{\mathrm{Tel}(k)}(E\otimes F_{n,k})$ is a retract of $L_{\mathrm{Tel}(k)}(E\otimes X(n)\otimes F_{n,k})$. The property of satisfying the telescope conjecture at height $k$ is closed under retracts, and $L_{\mathrm{Tel}(k)}(E\otimes X(n)\otimes F_{n,k})$ satisfies this property since it is complex-orientable by assumption.

Note that having infinite chromatic defect does not prevent a spectrum satisfying the condition of the telescope conjecture at all heights. For example, the connective image of $J$ spectrum $\mathrm{j}$ satisfies the condition of the telescope conjecture at all heights since it is $\mathrm{MU}$-nilpotent, but it has infinite chromatic defect by Theorem \ref{thm:defectj}.
\end{remark}

\subsection{Wood-types} In this section, we work $p$-locally. First we recall the following standard definition.

\begin{definition}
    A \textit{finite $\mathrm{BP}$-projective} is a finite spectrum $F$ such that $\mathrm{BP}
_*F$ is a projective $\mathrm{BP}_*$-module.
\end{definition}

In fact, we may as well work just with finite $\mathrm{BP}$-frees, due to the following lemma.

\begin{lemma}\label{lemma:BPfree}
    Every finite $\mathrm{BP}$-projective $F$ is finite $\mathrm{BP}$-free. That is, $\mathrm{BP}\otimes F$ is a finite free $\mathrm{BP}$-module.
\end{lemma}
\begin{proof}
    By finite typeness, it suffices to prove the claim after $p$-completion. Since $\mathrm{BP}_*F$ is projective, one uses Brown representability to exhibit $\mathrm{BP}\otimes F$ as a retract of $\mathrm{BP}\{T\}$, for some finite graded indexing set $T$.
    
    Since the $\F_p$-based ASS of $\mathrm{BP}$ collapses on $E_2$, it follows that that of $\mathrm{BP}\otimes F$ does as well. It follows from the indecomposability of the $\mathcal{A}_*$-comodule $H_*(\mathrm{BP};\F_p)$ that the inclusion $H_*(\mathrm{BP}\otimes F;\F_p)\to H_*(\mathrm{BP}\{T\};\F_p)$ is of the form $H_*(\mathrm{BP};\F_p)\{T'\}\to H_*(\mathrm{BP};\F_p)\{T\}$ for some $T'\subset T$. The corresponding statement thus holds on $E_\infty$-pages and on homotopy groups.
\end{proof}

 \begin{remark}
    Lemma \ref{lemma:BPfree} may also be proven inductively using a cell structure on $F$. In fact, all of the examples we will discuss concern finite complexes $F$ with only even-dimensional cells, which are automatically $\mathrm{BP}$-free since $\mathrm{BP}_*$ is concentrated in even degrees.
\end{remark}

\begin{definition}\label{def:wood}
    A $p$-local spectrum $E$ is said to be \emph{Wood-type} if there exists a finite $\mathrm{BP}$-projective $F$ such that $E\otimes F$ is complex-orientable.
\end{definition}

\begin{example}\label{example:mnneon}
    Fix a height $n$ formal group $\Gamma$ over a perfect field $k$ of characteristic $p$ and a finite subgroup $G\subset\mathbb G_n$ of the corresponding Morava stabilizer group, and we let $\mathrm{EO}_n(G)$ denote the fixed points of $E(k,\Gamma)$ with respect to $G$. It is a theorem of Meier--Naumann--Noel that $\mathrm{EO}_n(G)$ is Wood-type (see \cite[Appendix B]{meiernaumannnoel}).
\end{example}

\begin{example}\label{example:woodtmf}
    The above example gives higher height generalizations of Wood's theorem for the periodic theory $\mathrm{KO}$, which becomes $\mathrm{EO}_1(C_2)$ after 2-completion. Wood equivalences for connective theories like $\mathrm{ko}$ are harder to come by, and we revisit this in Section \ref{sec:4} in the context of fp spectra.

    Some known examples beyond $\mathrm{ko}$ include the equivalence
    \[\mathrm{tmf}\otimes D\mathcal{A}(1)\simeq \mathrm{tmf}_1(3)\]
    at the prime 2 and the equivalence 
    \[\mathrm{tmf}\otimes X_2\simeq \mathrm{tmf}_1(2)\]
     at the prime 3, where $X_2$ is the 8-skeleton of $\mathrm{T(1)}$ (see \cite{akhiltmf}). The spectra $\mathrm{tmf}_1(3)$ and $\mathrm{tmf}_1(2)$ are complex-orientable, so these equivalences imply that $\mathrm{tmf}$ is Wood-type.
\end{example}

As before, the thick subcategory theorem implies the following.

\begin{proposition}
    If $E$ is Wood-type, then $E$ is $\mathrm{BP}$-nilpotent.
\end{proposition}

The notion of Wood-types is closely tied to that of chromatic defect via the following useful fact.

\begin{proposition}
    Every finite $\mathrm{BP}$-free $F$ is a finite $\mathrm{T(n)}$-free for some $0\le n<\infty$.
\end{proposition}
\begin{proof}
If $F$ is a $\mathrm{BP}$-free, there is an equivalence of $\mathrm{BP}$-modules $\mathrm{BP}\otimes F\simeq \mathrm{BP}\{T\}$ for some finite graded indexing set $T$. For each $\alpha\in T$, the corresponding map $\mathbb S^{|\alpha|}\to \mathrm{BP}\otimes F$ factors through $\mathrm{T(n)}\otimes F$ for some $n$, since $\mathrm{T(n)}\to \mathrm{BP}$ is a $(2p^{n+1}-4)$-equivalence. Since $T$ is finite, we may choose $n$ such that such a factorization exists for all $\alpha\in T$, and this defines a map
\[\mathrm{T(n)}\{T\}\to \mathrm{T(n)}\otimes F\]
Applying $\mathrm{BP}\otimes_{\mathrm{T(n)}}-$, this map becomes an equivalence, thus it is an equivalence as $\mathrm{BP}$ is a free $\mathrm{T(n)}$-module by Proposition \ref{prop:T(n)toBPredundancy}.
\end{proof}

This gives a necessary condition for $E$ to be Wood-type.

\begin{corollary}\label{cor:woodimpliesfinitedefect}
    If $E$ is Wood-type, then $E$ has finite chromatic defect.
\end{corollary}
\begin{proof}
     By the proposition, there is some finite $\mathrm{T(n)}$-free $F$ such that $E\otimes F$ is complex-orientable. It follows that $E\otimes \mathrm{T(n)}$ is a retract of the complex-orientable spectrum $E\otimes F\otimes \mathrm{T(n)}$.
\end{proof}

\begin{corollary}\label{cor:EOnfinitedefect}
    For any $\mathrm{E}(k,\Gamma)$ and $G\subset\mathbb{G}_n$ as in Example \ref{example:mnneon}, the spectrum $\mathrm{EO}_n(G)$ has finite chromatic defect.
\end{corollary}

In Section \ref{sec:6} we will return to this example and compute the chromatic defect $\Phi(\mathrm{EO}_n(G))$ precisely in terms of the valuation on the  ring $\mathrm{End}(\Gamma)$.

\section{The \texorpdfstring{$\mathrm{T(n)}$}{T(n)}'s and \texorpdfstring{$\F_p$}{Fp}-homology}\label{sec:3}
In this section, we use the Thom isomorphism  to analyze the $\mathrm{T(n)}$'s and chromatic defect from the point of view of the Adams spectral sequence.

\subsection{The Adams spectral sequence of \texorpdfstring{$\mathrm{T(n)}$}{T(n)}} We fix a prime $p$ and let $\mathrm{T(n)}$ be the summand of $X(p^n)_{(p)}$ as in Definition \ref{def:T(n)}. The Thom isomorphism implies that the map $\mathrm{T(n)}\to \F_p$ induces an injection of $\mathcal{A}_*$-comodule algebras
\[H_*(\mathrm{T(n)};\F_p)\to H_*(\F_p;\F_p)=\mathcal{A}_*\]
with image $\F_2[\zeta_1^2,\ldots,\zeta_n^2]$ for $p=2$ and $\F_p[\zeta_1,\ldots,\zeta_n]$ for $p$ odd. This can be described as a coinduced comodule.

\begin{proposition}\label{prop:HT(n)}
    There is an isomorphism of $\mathcal{A}_*$-comodule algebras
    \[
    H_*(\mathrm{T(n)};\F_p)\cong
    \begin{cases}
    \mathcal{A_*}\Boxover{\mathcal{A}_*/(\xi_1^2,\ldots,\xi_n^2)}\F_2&p=2\\
    \mathcal{A_*}\Boxover{\mathcal{A}_*/(\xi_1,\ldots,\xi_n)}\F_p&p>2
    \end{cases}
    \]
\end{proposition}
The quotient Hopf algebra $\mathcal{A}_*/(\xi_1^2,\ldots,\xi_n^2)$ at $p=2$ is given by
\[E(\xi_1,\ldots,\xi_n)\otimes P(\xi_{n+1},\xi_{n+2},\ldots)\]
and, at odd primes, $\mathcal{A}_*/(\xi_1,\ldots,\xi_n)$ is given by
\[E(\tau_1,\tau_2,\ldots)\otimes P(\xi_{n+1},\xi_{n+2},\ldots)\]
A change-of-rings isomorphism then gives the following.

\begin{corollary}
    The $E_2$-page of the Adams spectral sequence for $\mathrm{T(n)}$ is given by
    \[\Ext_{\mathcal{A}_*}(\F_p,H_*\mathrm{T(n)})\cong
    \begin{cases}
        \Ext_{E(\xi_1,\ldots,\xi_n)\otimes P(\xi_{n+1},\xi_{n+2},\ldots)}(\F_2,\F_2)&p=2\\
        \Ext_{E(\tau_1,\tau_2,\ldots)\otimes P(\xi_{n+1},\xi_{n+2},\ldots)}(\F_p,\F_p)&p>2
    \end{cases}\]
\end{corollary}

We now define an increasing filtration on these quotient Hopf algebras to obtain a corresponding May spectral sequence. This is the usual May filtration on $\mathcal{A}_*$ (see \cite[Section 3.2]{ravgreen}) projected to the quotient. For $p=2$, this is obtained by giving $\xi_i^{2^j}$ filtration $2i-1$, and for $p>2$ by giving $\xi_i^{2^j}$ and $\tau_{i-1}$ filtration $2i-1$, and extending multiplicatively. At $p=2$, the associated graded of this filtration is a tensor product of primitively generated exterior algebras. At odd primes, it is a tensor product of primitively generated exterior algebras and primitively generated truncated polynomial algebras of height $p$. As in the case for the sphere, we deduce the following.

\begin{proposition}
    There is a May spectral sequence converging to the $E_2$-page of the Adams spectral sequence for $\mathrm{T(n)}$.
    \begin{itemize}
        \item When $p=2$, the May SS has signature
        \[E_1^{s,t,w}=\F_2[h_{i,j}\st j=0\text{ if }i\le n]\implies \Ext^{s,t}_{\mathcal{A}_*}(\F_2,H_*\mathrm{T(n)})\]
        where $h_{i,j}=[\xi_i^{2^j}]$. In $(s,t,w)$ tridegrees, we have $|h_{i,j}|=(1,2^j(2^i-1),2i-1)$.
        \item When $p>2$, the May SS has signature
        \[E_1^{s,t,w}=E(h_{i,j}\st i> n)\otimes\F_p[b_{i,j}\st i>n]\otimes\F_p[a_i]\implies \Ext^{s,t}_{\mathcal{A}_*}(\F_p,H_*\mathrm{T(n)})\]
        where $h_{i,j}=[\xi_i^{p^j}]$, $a_i=[\tau_{i}]$, and $b_{i,j}$ is the $p$-fold Massey power of $h_{i,j}$, which is represented explicitly by the cocycle
        \[\sum\limits_{k=1}^p\frac{1}{p}\binom{p}{k}[\xi_i^{p^{jk}}|\xi_i^{p^{j(p-k)}}]\]
        The $(s,t,w)$ tridegrees of the generators are $|h_{i,j}|=(1,2p^j(p^i-1),2i-1)$, $|a_i|=(1,2p^i-1,2i-1)$, and $|b_{i,j}|=(2,2p^{pj}(p^i-1),4i-2)$.
    \end{itemize}
\end{proposition}

This spectral sequence may be used to compute the $E_2$-page of the Adams spectral sequence for $\mathrm{T(n)}$ through a range. We plot in Figure \ref{May} the $E_1$-page for $\mathrm{T(1)}$ at $p=2$. When working with the May spectral sequence, it can be convenient to ignore the May filtration degree, and plot the spectral sequence in Adams bigrading $(t-s,s)$ for a class in $\Ext^{s,t}$, where $s$ is the cohomological degree, and $t$ is the internal degree. In this grading, all May differentials have the signature of an Adams $d_1$. By direct use of the coproduct formula and Nakamura's lemma \cite{nakamura}, one can compute some May differentials for $\mathrm{T(1)}$ in low stems.

\begin{sseqdata}[ name = MayX(2), Adams grading, classes = {fill, show name=below},
grid = go, xrange ={0}{9},yrange={0}{8},xscale=0.7,yscale=0.7,x tick step =2, y tick step =2,run off differentials = {->},struct lines = black ]
\class(0,0)
\class[name=h_{2,0}](2,1)
\structline[red](0,0)(2,1)
\DoUntilOutOfBoundsThenNMore{2}{
    \class(\lastx+2,\lasty+1)
    \structline[red]
}
\class(0,1)
\DoUntilOutOfBoundsThenNMore{2}{
    \class(\lastx+2,\lasty+1)
    \structline[red]
}
\class(0,2)
\DoUntilOutOfBoundsThenNMore{2}{
    \class(\lastx+2,\lasty+1)
    \structline[red]
}
\class(0,3)
\DoUntilOutOfBoundsThenNMore{2}{
    \class(\lastx+2,\lasty+1)
    \structline[red]
}
\class(0,4)
\DoUntilOutOfBoundsThenNMore{2}{
    \class(\lastx+2,\lasty+1)
    \structline[red]
}
\class(0,5)
\DoUntilOutOfBoundsThenNMore{2}{
    \class(\lastx+2,\lasty+1)
    \structline[red]
}
\class(0,6)
\DoUntilOutOfBoundsThenNMore{2}{
    \class(\lastx+2,\lasty+1)
    \structline[red]
}
\class(0,7)
\DoUntilOutOfBoundsThenNMore{2}{
    \class(\lastx+2,\lasty+1)
    \structline[red]
}
\class(0,8)
\DoUntilOutOfBoundsThenNMore{2}{
    \class(\lastx+2,\lasty+1)
    \structline[red]
}
\class(0,10)
\class(2,10)
\class(4,10)
\class(6,10)
\class(8,10)
\structline(0,0)(0,8)
\structline(2,1)(2,8)
\structline(4,2)(4,8)
\structline(6,3)(6,8)
\structline(8,4)(8,8)
\structline(10,5)(10,8)
\structline(12,6)(12,8)
\structline(14,7)(14,8)
\structline(0,8)(0,10)
\structline(2,8)(2,10)
\structline(4,8)(4,10)
\structline(6,8)(6,10)
\structline(8,8)(8,10)

\class[name=h_{2,1}](5,1)
\DoUntilOutOfBoundsThenNMore{2}{
    \class(\lastx+2,\lasty+1)
    \structline[red]
}
\class(5,2)
\DoUntilOutOfBoundsThenNMore{2}{
    \class(\lastx+2,\lasty+1)
    \structline[red]
}
\class(5,3)
\DoUntilOutOfBoundsThenNMore{2}{
    \class(\lastx+2,\lasty+1)
    \structline[red]
}
\class(5,4)
\DoUntilOutOfBoundsThenNMore{2}{
    \class(\lastx+2,\lasty+1)
    \structline[red]
}
\class(5,5)
\DoUntilOutOfBoundsThenNMore{2}{
    \class(\lastx+2,\lasty+1)
    \structline[red]
}
\class(5,6)
\DoUntilOutOfBoundsThenNMore{2}{
    \class(\lastx+2,\lasty+1)
    \structline[red]
}
\class(5,7)
\DoUntilOutOfBoundsThenNMore{2}{
    \class(\lastx+2,\lasty+1)
    \structline[red]
}
\class(5,8)
\DoUntilOutOfBoundsThenNMore{2}{
    \class(\lastx+2,\lasty+1)
    \structline[red]
}
\class(5,10)
\class(7,10)
\class(9,10)
\structline(5,1)(5,8)
\structline(5,8)(5,10)
\structline(7,2)(7,8)
\structline(7,8)(7,10)
\structline(9,3)(9,8)
\structline(9,8)(9,10)
\class(5,9)
\class(5,11)

\class[name=h_{3,0}](6,1)
\DoUntilOutOfBoundsThenNMore{2}{
    \d1
    \class(\lastx+2,\lasty+1)
    \structline[red]
}
\class(6,2)
\DoUntilOutOfBoundsThenNMore{2}{
    \d1
    \class(\lastx+2,\lasty+1)
    \structline[red]
}
\class(6,3)
\DoUntilOutOfBoundsThenNMore{2}{
    \d1
    \class(\lastx+2,\lasty+1)
    \structline[red]
}
\class(6,4)
\DoUntilOutOfBoundsThenNMore{2}{
    \d1
    \class(\lastx+2,\lasty+1)
    \structline[red]
}
\class(6,5)
\DoUntilOutOfBoundsThenNMore{2}{
    \d1
    \class(\lastx+2,\lasty+1)
    \structline[red]
}
\class(6,6)
\DoUntilOutOfBoundsThenNMore{2}{
    \d1
    \class(\lastx+2,\lasty+1)
    \structline[red]
}
\class(6,7)
\DoUntilOutOfBoundsThenNMore{2}{
    \d1
    \class(\lastx+2,\lasty+1)
    \structline[red]
}
\class(6,8)
\DoUntilOutOfBoundsThenNMore{1}{
    \d1
    \class(\lastx+2,\lasty+1)
    \structline[red]
}
\structline(6,1)(6,2)
\structline(6,2)(6,3,2)
\structline(6,3,2)(6,4,2)
\structline(6,4,2)(6,5,2)
\structline(6,5,2)(6,6,2)
\structline(6,6,2)(6,7,2)
\structline(6,7,2)(6,8,2)
\class(6,9)
\structline(6,8,2)(6,9,2)

\structline(8,2)(8,3)
\structline(8,3)(8,4,2)
\structline(8,4,2)(8,5,2)
\structline(8,5,2)(8,6,2)
\structline(8,6,2)(8,7,2)
\structline(8,7,2)(8,8,2)
\class(8,9)
\structline(8,8,2)(8,9,2)
\end{sseqdata}

\begin{figure}[!htbp]
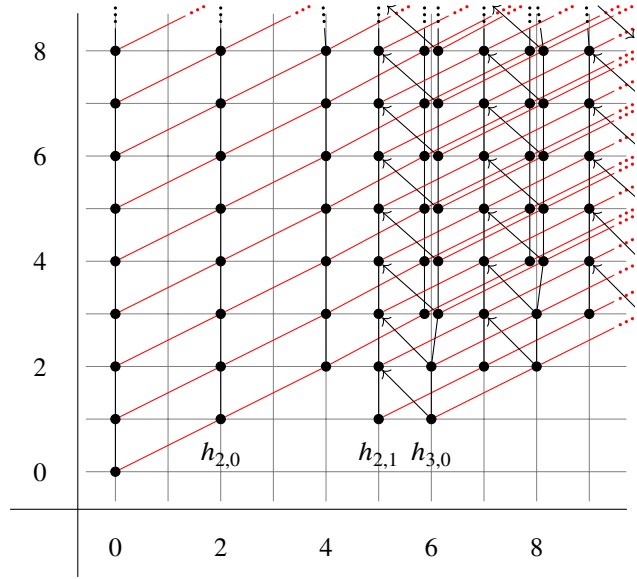

\centering
\printpage[name = MayX(2)]
\caption{The $E_1$-page of the May SS for $\mathrm{T(1)}$ at $p=2$}
\label{May}
\end{figure}

\begin{proposition}
    In the May SS for $\mathrm{T(1)}$ at $p=2$, we have the following differentials
    \begin{align*}
        d_1(h_{3,0})&=h_{1,0}h_{2,1}\\
        d_1(h_{4,0})&=h_{1,0}h_{3,1}+h_{2,0}h_{2,2}\\
        d_2(h_{3,0}^2)&=h_{1,0}^2h_{2,2}
    \end{align*}
\end{proposition}

\begin{sseqdata}[ name = ASSX(2), Adams grading, classes = {fill, show name=below},
grid = go, xrange ={0}{14},yrange={0}{8},xscale=0.7,yscale=0.7,x tick step =2, y tick step =2,run off differentials = {->},struct lines = black ]
\class(0,0)
\class[name=v_1](2,1)
\structline[red](0,0)(2,1)
\DoUntilOutOfBoundsThenNMore{2}{
    \class(\lastx+2,\lasty+1)
    \structline[red]
}
\class(0,1)
\DoUntilOutOfBoundsThenNMore{2}{
    \class(\lastx+2,\lasty+1)
    \structline[red]
}
\class(0,2)
\DoUntilOutOfBoundsThenNMore{2}{
    \class(\lastx+2,\lasty+1)
    \structline[red]
}
\class(0,3)
\DoUntilOutOfBoundsThenNMore{2}{
    \class(\lastx+2,\lasty+1)
    \structline[red]
}
\class(0,4)
\DoUntilOutOfBoundsThenNMore{2}{
    \class(\lastx+2,\lasty+1)
    \structline[red]
}
\class(0,5)
\DoUntilOutOfBoundsThenNMore{2}{
    \class(\lastx+2,\lasty+1)
    \structline[red]
}
\class(0,6)
\DoUntilOutOfBoundsThenNMore{2}{
    \class(\lastx+2,\lasty+1)
    \structline[red]
}
\class(0,7)
\DoUntilOutOfBoundsThenNMore{2}{
    \class(\lastx+2,\lasty+1)
    \structline[red]
}
\class(0,8)
\DoUntilOutOfBoundsThenNMore{2}{
    \class(\lastx+2,\lasty+1)
    \structline[red]
}
\class(0,9)
\structline(0,0)(0,8)
\structline(2,1)(2,8)
\structline(4,2)(4,8)
\structline(6,3)(6,8)
\structline(8,4)(8,8)
\structline(10,5)(10,8)
\structline(12,6)(12,8)
\structline(14,7)(14,8)
\structline(0,8)(0,9)
\structline(2,8)(2,9)
\structline(4,8)(4,9)
\structline(6,8)(6,9)
\structline(8,8)(8,9)
\structline(10,8)(10,9)
\structline(12,8)(12,9)
\structline(14,8)(14,9)

\class[name=\chi_2](5,1)
\DoUntilOutOfBoundsThenNMore{2}{
    \class(\lastx+2,\lasty+1)
    \structline[red]
}

\class[name=\chi_2^2](10,2)
\DoUntilOutOfBoundsThenNMore{2}{
    \class(\lastx+2,\lasty+1)
    \structline[red]
}
\class[name=h_{2,2}](11,1)
\DoUntilOutOfBoundsThenNMore{2}{
    \class(\lastx+2,\lasty+1)
    \structline[red]
}
\class[name=h_{3,1}](13,1)
\DoUntilOutOfBoundsThenNMore{2}{
    \class(\lastx+2,\lasty+1)
    \structline[red]
}
\structline(13,1)(13,2)

\class(11,2)
\DoUntilOutOfBoundsThenNMore{2}{
    \class(\lastx+2,\lasty+1)
    \structline[red]
}

\structline(11,1)(11,2)
\structline(13,2)(13,3)
\structline(15,2)(15,4)

\end{sseqdata}

\begin{figure}[!htbp]
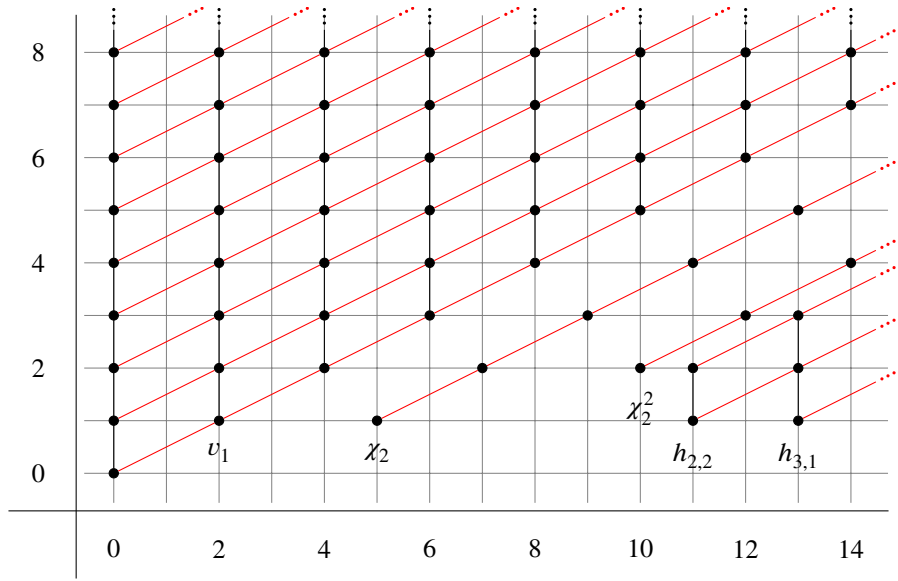

\centering
\printpage[name = ASSX(2)]
\caption{The $E_2$-page of the Adams SS for $\mathrm{T(1)}$ at $p=2$}
\label{Adams}
\end{figure}

The May SS for $\mathrm{T(n)}$ in general exhibits a similar pattern. The following comes by a degree check on the generators of the $E_1$-page, and the differential follows from the coproduct formula.

\begin{proposition}\label{prop:MaycalculationT(n)}
    The inclusion
    \[\F_2[h_{i,0},h_{n+1,1}\st i\le n+2]\to E_1\mathrm{-May}(\mathrm{T(n)})\]
    is an isomorphism in stems $t-s\le 2^{n+2}-2=\mathrm{stem}(h_{n+2,0})$, at $p=2$. There is a differential $d_1(h_{n+2,0})=h_{1,0}h_{n+1,1}$.
    
    The inclusion
    \[E(h_{n+1,0})\otimes\F_p[a_i\st i\le n+2]\to E_1\mathrm{-May}(\mathrm{T(n)})\]
    is an isomorphism in stems $t-s\le 2p^{n+1}-2=\mathrm{stem}(a_{n+1})$, for $p>2$. There is a differential $d_1(a_{n+1})=h_{1,0}h_{n+1,0}$.
\end{proposition}
\begin{proof}
    We give the proof for $p=2$, the odd primary case being analogous. The inclusion in question is the canonical inclusion
    \[\F_2[h_{i,0},h_{n+1,1}\st i\le n+2]\to\F_2[h_{i,j}\st j=0\text{ if }i\le n]\]
    In particular, the polynomial generators not in the image of this inclusion consist of $h_{n+1,j}$ where $j>1$, $h_{n+2,j}$ where $j>0$, and $h_{n+k,j}$ for $k>2$ and for all $j$. The $t-s$ degree of each of these generators exceeds that of $h_{n+2,0}$, so these generators do not contribute to the range in question.

    The May differential $d_1(h_{n+2,0})=h_{1,0}h_{n+1,1}$ follows from the coproduct formula
    \[\Delta(\xi_{n+2})=\xi_{n+2}\otimes 1+1\otimes\xi_{n+2}+\xi_{n+1}^2\otimes\xi_1+\sum\limits_{j=2}^{n+1}\xi_{n+2-j}^{2^j}\otimes \xi_{j}\]
    so that, in the cobar complex, $d([\xi_{n+2}])=[\xi_{n+1}^2|\xi_1]$ modulo higher May filtration terms.
\end{proof}

\begin{corollary}\label{cor:firstnonzeroT(n)}
    The first nonzero odd homotopy group of $\mathrm{T(n)}$ is
    \[\pi_{2p^{n+1}-3}\mathrm{T(n)}=\Z/p\]
    A generator is detected by $h_{n+1,1}$ at $p=2$ and $h_{n+1,0}$ at $p>2$. The same is true for $\mathrm{X(m)}_{(p)}$ whenever $p^n\le m<p^{n+1}$.
\end{corollary}
\begin{proof}
    By Proposition \ref{prop:MaycalculationT(n)}, the $2p^{n+1}-3$ stem of the $E_2$-page of the May SS for $\mathrm{T(n)}$ has a single nonzero bidegree $(2p^{n+1}-3,1)$ generated by $h_{n+1,1}$ when $p=2$ and by $h_{n+1,0}$ when $p>2$. The $2p^{n+1}-4$ stem is $h_{1,0}$-torsion free and $h_{1,0}h_{n+1,1}=0$ (resp. $h_{1,0}h_{n+1,0}=0$ at $p>2$), so the generator of the $2p^{n+1}-3$ stem cannot support a differential. The May SS thus collapses on $E_2$ in this stem, and the same argument shows the generator cannot support an Adams differential.

    The statement about $\mathrm{X(m)}$ then follows immediately from the splitting of Proposition \ref{prop:T(n)splitting}.
\end{proof}

In fact, the attaching maps $\sigma_k$ discussed in the previous section identify for us a canonical generator of $\pi_{2p^{m+1}-3}\mathrm{T(n)}$. Since $\mathrm{T(n)}$ is a weak $\mathrm{X(p^{n+1}-1)}$-module, one has a splitting
\[\mathrm{T(n)}\otimes\Sigma^{-2}\CP^{p^{n+1}-1}\simeq \mathrm{T(n)}\{b_0,\ldots,b_{p^{n+1}-1}\}\]
Since the odd homotopy groups of $\mathrm{T(n)}$ vanish in degrees less than $2p^{n+1}-3$, the attaching map
\[\sigma_{p^{n+1}-1}:\mathbb{S}^{2p^{n+1}-3}\to \mathrm{T(n)}\otimes\Sigma^{-2}\CP^{p^{n+1}-1}\]
has nonzero component at $b_i$ if and only if $i=0$. This component defines an element of $\pi_{2p^{n+1}-3}\mathrm{T(n)}$ which we call $\chi_{n+1}$.

\begin{proposition}\label{prop:chi}
    The element $\chi_{n+1}$ is a generator of $\pi_{2p^{n+1}-3}\mathrm{T(n)}=\Z/p$.
\end{proposition}
\begin{proof}
    The above discussion shows that $\chi_{n+1}=0$ if and only if $\mathrm{T(n)}\otimes\sigma_{p^{n+1}-1}$ is null. If this is the case, then $\mathrm{T(n)}$ is a weak $\mathrm{X(p^{n+1})}$-module by Proposition \ref{prop:weakX(n)}, and therefore a weak $\mathrm{T(n+1)}$-module. However, the map $\mathrm{T(n+1)}\to \mathrm{BP}$ is $(2p^{n+2}-3)$-connected, hence the same is true for the map 
    \[\mathrm{T(n)}\otimes \mathrm{T(n+1)}\to \mathrm{T(n)}\otimes \mathrm{BP}\]
    The homotopy groups of $\mathrm{T(n)}\otimes \mathrm{BP}$ are concentrated in even degrees, hence we have $\pi_{2p^{n+1}-3}(\mathrm{T(n)}\otimes \mathrm{T(n+1)})=0$, contradicting that $\pi_{2p^{n+1}-3}\mathrm{T(n)}$ is a retract.
\end{proof}

\begin{remark}
    The source $h_{n+2,0}$ (resp. $a_{n+1}$ at $p>2$) of the May $d_1$ of Proposition \ref{prop:MaycalculationT(n)} detects the class 
    \[v_{n+1}\in \pi_{2(p^{n+1}-1)}\mathrm{T(n+1)}\cong \pi_{2(p^{n+1}-1)}\mathrm{BP}\]
    This implies in fact that $v_{n+1}\in \pi_{2(p^{n+1}-1)}(\mathrm{T(n+1)}/p)$ lifts to a class in 
    \[\pi_{2(p^{n+1}-1)}(\mathrm{T(n)}/p)\]
    whose Bockstein is $\chi_{n+1}$. This gives a Toda bracket description in $\mathrm{T(n+1)}$
    \[v_{n+1}\in\langle p,\chi_{n+1},1\rangle\]
    as one can show using, for example \cite[Lemma 4.6.1]{meierthesis}.
\end{remark}

The element $\chi_{n+1}$ gives a convenient criterion to compute the chromatic defect of a homotopy associative ring spectrum $E$, namely an explicit sequence of obstructions in $\mathrm{T(n)}$ homology to $\Phi_p(E)$ being $\le n$. We continue to work in the $p$-local setting.

\begin{proposition}\label{prop:defectringspectrum}
    Let $E$ be a $p$-local homotopy associative ring spectrum. Then $\Phi_p(E)\le n$ if and only if
    \begin{equation}\label{chizero}\chi_{m+1}=0\in\pi_{2p^{m+1}-3}(E\otimes \mathrm{T(m)})\end{equation}
    for all $m\ge n$.
\end{proposition}
\begin{proof}
    By Proposition \ref{prop:weakX(n)} and the discussion preceding Proposition \ref{prop:chi}, $E\otimes \mathrm{T(m)}$ is a weak $\mathrm{T(m+1)}$-module if and only if 
    \[E\otimes \mathrm{T(m)}\otimes\sigma_{p^{m+1}-1}=0\]
    whose only potentially nonzero component is $E\otimes \Sigma^{2p^{m+1}-3}\mathrm{T(m)}\xrightarrow{E\otimes\chi_{m+1}} E\otimes \mathrm{T(m)}$. Since $E\otimes \mathrm{T(m)}$ is a homotopy associative ring spectrum, this is null if and only if \ref{chizero} holds. By induction, if \ref{chizero} holds for all $m\ge n$, then $E\otimes \mathrm{T(m)}$ is a weak $\mathrm{T(m)}$-module for all $m\ge n$, hence complex-orientable.
\end{proof}

\subsection{The chromatic defect of finite spectra} We begin this section with a toy example, on which we will elaborate to show that the chromatic defect of any finite spectrum is infinite. In this example, we will show that the finite spectrum $C(\eta)$ is not complex-orientable; i.e. that $\Phi(C(\eta))>1$. The fact that $C(\eta)$ has infinite chromatic defect is the claim that $\Phi(C(\eta))>n$ for all $n$, which follows by a completely analogous argument in the category of $\mathrm{X(n)}$-modules; see Corollary \ref{cor:finitespectradefect}.

We work 2-locally throughout this example. We have $\mathrm{X(1)}=\mathbb{S}$, and $\mathrm{T(0)}=\mathrm{X(1)}_{(p)}=\mathbb{S}_{(p)}$. The attaching map $\sigma_1:\mathbb{S}^1\to \mathbb{S}$ is by definition the Hopf map $\eta$, which corresponds to the element $\chi_1$, when $p=2$. The quotient $\mathbb{S}/\eta=C(\eta)$ does not admit a unital multiplication. Indeed, if it did, the endomorphism
\[\eta:\Sigma C(\eta)\to C(\eta)\]
must be null. Taking the cofiber, this would give an equivalence
\[C(\eta)\otimes C(\eta)\simeq C(\eta)\oplus\Sigma^2C(\eta)\]
but the mod 2 cohomology of the left hand side admits a nontrivial $\mathrm{Sq}^4$ action, while the right hand side does not.

We may instead embed the quotient $C(\eta)$ into a ring spectrum by forming the $\mathbb{E}_1$-quotient of $\mathbb{S}$ by $\eta$ as the pushout
\[
\begin{tikzcd}
    \mathrm{Free}_{\mathbb{E}_1}(\mathbb{S}^1)\arrow[r,"\overline{\eta}"]\arrow[d,"\overline{0}"']&\mathbb{S}\arrow[d]\\
    \mathbb{S}\arrow[r]&\mathbb{S}/\!\!/\eta
\end{tikzcd}
\]
in the category of $\mathbb{E}_1$-rings, where $\overline{0}$ and $\overline{\eta}$ denote the adjoint maps with respect to the free-forget adjunction. This quotient is in fact $\mathrm{T(1)}$.

\begin{proposition}\label{prop:alpha1}
    There is an equivalence of $\mathbb{E}_1$-rings
    \[\mathbb{S}/\!\!/\eta\to \mathrm{T(1)}\]
    $2$-locally, and an equivalence of $\mathbb{E}_1$-rings
    \[\mathbb{S}/\!\!/\alpha\to \mathrm{T(1)}\]
    $p$-locally for odd primes $p$.
\end{proposition}
\begin{proof}
    We give the proof for $p=2$, the odd primary case being analogous. Choosing a nullhomotopy of $\eta$ in $\pi_*\mathrm{T(1)}$ determines a map
    $\mathbb{S}/\!\!/\eta\to \mathrm{T(1)}$ by universal property, and since both sides are finite type, we may check that is an isomorphism in mod 2 cohomology. Since $\eta$ is zero in mod 2 homology, when we tensor this pushout diagram up to $\F_2$, we have a pushout
\[
\begin{tikzcd}
    \mathrm{Free}_{\mathbb{E}_1-\F_2}(\mathbb{S}^1)\arrow[rr,"\mathrm{Free}_{\mathbb{E}_1-\F_2}(0)"]\arrow[d,"\mathrm{Free}_{\mathbb{E}_1-\F_2}(0)"']&&\mathrm{Free}_{\mathbb{E}_1-\F_2}(*)\arrow[d]\\
    \mathrm{Free}_{\mathbb{E}_1-\F_2}(*)\arrow[rr]&&\mathbb{S}/\!\!/\eta\otimes\F_2
\end{tikzcd}
\]
Since the Free functor preserves pushouts, one has
\[\pi_*(\mathbb{S}/\!\!/\eta\otimes\F_2)=\pi_*(\mathrm{Free}_{\mathbb{E}_1-\F_2}(\mathbb{S}^2))=\F_2[x_2]\]
where $|x_2|=2$. 

We claim now that in mod $2$ cohomology, $\mathrm{Sq}^2$ of the nonzero class in degree $0$ is equal to the nonzero class in degree $2$ for both $H^*(\mathbb{S}/\!\!/\eta)$ and $H^*(\mathrm{T(1)})$. With this in place, by linearity over the Steenrod algebra, the map $\mathbb{S}/\!\!/\eta\to \mathrm{T(1)}$ must induce an isomorphism on $H^{\le 2}(-;\F_2)$ and therefore also on $H_{\le 2}(-;\F_2)$. Since the map induced on homology is a ring map, it is therefore an isomorphism.

To establish the claim about $\mathrm{Sq}^2$, note that for a ring spectrum $R$, there is an extension of the unit map $\mathbb{S}\to R$ over $C(\eta)$. If $R$ is connective and  $H^0(R;\F_2)=\F_2$, then the map $C(\eta)\to R$ must induce an isomorphism on $H^0(-;\F_2)$ since it extends the unit. But then since $\mathrm{Sq}^2$ is nonzero on $H^0(C(\eta);\F_2)$ it must be nonzero on $H^0(R;\F_2)$. Note now that $\mathbb{S}/\!\!/\eta$ and $\mathrm{T(1)}$ satisfy the conditions on $R$.
\end{proof}

\begin{remark}
    It follows that if $E$ is any $p$-local spectrum, then $E$ admits the structure of a $\mathrm{T(1)}$-module if and only if $E\otimes\eta$ is null for $p=2$ and if and only if $E\otimes\alpha_1$ is null if $p>2$. Indeed, if $E\otimes \eta$ (resp. $E\otimes\alpha_1$ for $p>2)$ is null, then the $\mathbb{E}_1$-ring map $\mathbb{S}\to \mathrm{End}(E)$ factors canonically through an $\mathbb{E}_1$-ring map $\mathrm{T(1)}\to \mathrm{End}(E)$ by  universal property, giving $E$ a $\mathrm{T(1)}$-module structure. Conversely, the map $\eta$ (resp. $\alpha_1$ for $p>2$) acts by zero on any $\mathrm{T(1)}$-module.
\end{remark}

We finish the toy example by noting the following: if $C(\eta)$ were complex-orientable, then the map $\eta$ must act by zero on $C(\eta)$. The above proposition would then provide a ring map
\[f:\mathrm{T(1)}\to\mathrm{End}(C(\eta))\]
This map $f$ would send $1\mapsto 1$ in mod $2$ homology. However, in $H_*\mathrm{T(1)}=\F_2[x_2]$, there is a coaction formula
\[\psi(x_2^{2^n})=\zeta_1^{2^{n+1}}\otimes 1+1\otimes x_2^{2^n}\]
Since $H_*(f)$ is a map of comodule algebras, this implies that $x_2^{2^n}$ must be sent to a nonzero element for all $n$, contradicting the boundedness of $H_*\mathrm{End}(C(\eta))$.

We adapt this argument now to a general finite spectrum and to the more general notion of finite chromatic defect. This hinges on a theorem of Beardsley, which may be shown by adapting Proposition \ref{prop:alpha1} to the category of $\mathrm{X(p^n)}_{(p)}$-modules, with $\chi_{n+1}$ in place of $\eta$ \cite{beardsley}.

\begin{theorem}[Beardsley]\label{thm:beardsley}
    Let $p^n\le m<p^{n+1}$. If $m<p^{n+1}-1$, then 
    \[\mathrm{X(m+1)}_{(p)}\simeq \mathrm{X(m)}_{(p)}[b_{m}]\]
    is the free $\mathbb{E}_1$-$\mathrm{X(m)}_{(p)}$-algebra on a class in degree $2m$.
    
    If $m=p^{n+1}-1$, $\mathrm{X(m+1)}_{(p)}$ is the free $\mathbb{E}_1$-$\mathrm{X(m)}_{(p)}$-algebra with a nullhomotopy of $\chi_{n+1}$. That is, there is a pushout
\[
\begin{tikzcd}
    \mathrm{Free}_{\mathbb{E}_1-\mathrm{X(p^{n+1}-1)}_{(p)}}(\mathbb{S}^{2p^{n+1}-3})\arrow[r,"\overline{\chi_{n+1}}"]\arrow[d,"\overline{0}"']&\mathrm{X(p^{n+1}-1)}_{(p)}\arrow[d]\\
    \mathrm{X(p^{n+1}-1)}_{(p)}\arrow[r]&\mathrm{X(p^{n+1})}
\end{tikzcd}
\]
    in $\mathbb{E}_1$-$\mathrm{X(p^{n+1}-1)}_{(p)}$-algebras.
\end{theorem}

Essentially, Beardsley's theorem says if one attaches an $\mathbb{E}_1$-$\mathrm{T(n)}$-cell to $\mathrm{T(n)}$ to kill $\chi_{n+1}$, one gets $\mathrm{T(n+1)}$. However, since $\mathrm{T(n)}$ is not known to be an $\mathbb{E}_2$-ring, it doesn't make sense to speak of $\mathbb{E}_1$-$\mathrm{T(n)}$-algebras, so we must state things in terms of the $\mathbb{E}_2$-rings $\mathrm{X(m)}_{(p)}$.

    

The discussion following Proposition \ref{prop:alpha1} may be repeated in the category of $\mathrm{X(p^n)}_{(p)}$-modules, using the relative dual Steenrod algebra. We will only need the following piece of information.

\begin{lemma}\label{lemma:coaction}
    Let $p^n\le m<p^{n+1}$. The map
    \[\mathcal{A}_*=\pi_*(\F_p\otimes\F_p)\to\pi_*(\F_p\otimes_{\mathrm{X(m)}}\F_p)=:\mathcal{A}_*^{\mathrm{X(m)}}\]
    sends $\zeta_{n+1}^{p^k}$ to a nonzero coalgebra primitive for all $k$, when $p>2$. When $p=2$, the map sends $\zeta_{n+1}^{2^k}$ to a nonzero coalgebra primitive for all $k>0$.
\end{lemma}
\begin{proof}
    We have an isomorphism
    \[\pi_*(\F_p\otimes_{\mathrm{X(m)}}\F_p)\cong\pi_*((\F_p\otimes\F_p)\otimes_{\F_p\otimes \mathrm{X(m)}}\F_p)\cong\mathcal{A}_*\otimes_{H_*\mathrm{X(m)}}\F_p\]
    Indeed, the map $H_*\mathrm{T(n)}\to\mathcal{A}_*$ is identified with the flat inclusion
    \[\F_2[\zeta_1^2,\ldots,\zeta_n^2]\to\F_2[\zeta_1,\zeta_2,\ldots]\]
    Moreover, $H_*\mathrm{T(n)}$ is flat over $H_*\mathrm{X(m)}$ as it is a retract of a free module. This implies $\mathcal{A}_*$ is flat over $H_*\mathrm{X(m)}$, and the above isomorphism follows from collapse of the corresponding Kunneth spectral sequence.
    
    The splitting $H_*\mathrm{X(m)}\cong H_*\mathrm{T(n)}[x_i:i\neq p^k-1,i<m]$ implies that
\[\mathcal{A}_*\otimes_{H_*\mathrm{X(m)}}\F_p\cong(\mathcal{A}_*\otimes_{H_*\mathrm{T(n)}}\F_p)[x_i:i\neq p^k-1,i<m]\]
It follows from Proposition \ref{prop:HT(n)} that 
\[\mathcal{A}_*\otimes_{H_*\mathrm{T(n)}}\F_p\cong \begin{cases}E(\xi_1,\ldots,\xi_n)\otimes P(\xi_{n+1},\xi_{n+2},\ldots)&p=2\\E(\tau_1,\tau_2,\ldots)\otimes P(\xi_{n+1},\xi_{n+2},\ldots)&p>2\end{cases}\]
from which the claim now follows.
\end{proof}

\begin{proposition}\label{prop:noweakmodules}
    Suppose $p^n\le m<p^{n+1}$, and $p^{n+1}\le m+k\le \infty$. There are no nontrivial compact $\mathrm{X(m)}_{(p)}$-modules that are weak $\mathrm{X(m+k)}_{(p)}$-modules.
\end{proposition}
\begin{proof}
    In what follows, we implicitly work $p$-locally and drop this from the notation. If $F$ is a compact $\mathrm{\mathrm{X(m)}}$-module that is a weak $\mathrm{X(m+k)}$-module, then by Theorem \ref{thm:beardsley} there is an $\mathbb{E}_1$-$\mathrm{\mathrm{X(m)}}$-algebra map
    \[f:\mathrm{X(p^{n+1})}\to \mathrm{End}_{\mathrm{X(m)}}(F)\]
    as $0=\chi_{n+1}\in \pi_*\mathrm{X(m+k)}$. 

    The relative homology
    \[H^{\mathrm{X(m)}}_*(\mathrm{End}_{\mathrm{X(m)}}(F))=\pi_*(\F_p\otimes_{\mathrm{X(m)}}\mathrm{End}_{\mathrm{X(m)}}(F))\]
    is a comodule over the relative dual Steenrod algebra $\mathcal{A}_*^{\mathrm{X(m)}}$. These relative homology groups are bounded above as $\mathrm{End}_{\mathrm{X(m)}}(F)$ is compact, and $H^{\mathrm{X(m)}}_*(\mathrm{X(m)})$ is bounded above. 
    
    Via the map of Hopf algebras $\mathcal{A}_*\to \mathcal{A}_*^{\mathrm{X(m)}}$, any $\mathcal{A}_*$-comodule may be regarded as an $\mathcal{A}_*^{\mathrm{X(m)}}$-comodule, and the composite
    \[H_*\mathrm{T(n+1)}\to H_*(\mathrm{X(p^{n+1})})\to H_*^{\mathrm{X(m)}}(\mathrm{X(p^{n+1})})\xrightarrow{f_*}H_*^{\mathrm{X(m)}}(\mathrm{End}_{\mathrm{X(m)}}(F))\]
    is a map of $\mathcal{A}_*^{\mathrm{X(m)}}$-comodules sending $1\mapsto 1$. By Lemma \ref{lemma:coaction}, one has the coaction formula
    \[\psi(\zeta_{n+1}^{p^k})=1\otimes \zeta_{n+1}^{p^k}+\zeta_{n+1}^{p^k}\otimes 1\]
    in the $\mathcal{A}_*^{\mathrm{X(m)}}$-comodule $H_*\mathrm{T(n+1)}$ at odd primes, and similarly at $p=2$ with $\zeta_{n+1}^2$ in place of $\zeta_{n+1}$.
    
    The above composite must therefore send $\zeta_{n+1}^{p^k}$ (resp. $\zeta_{n+1}^{2^{k+1}}$ at $p=2$) to a nonzero element for all $k\ge0$. This gives a contradiction if $1\neq0\in \pi_0(\mathrm{End}_{\mathrm{X(m)}}(F))$ by boundedness.
\end{proof}

\begin{corollary}\label{cor:finitespectradefect}
     If $F$ is a nontrivial finite spectrum, then $\Phi(F)=\infty$. Working $p$-locally, if $F$ is a nontrivial finite $p$-local spectrum, then $\Phi_p(F)=\infty$.
\end{corollary}
\begin{proof}
    If $F$ is a nontrivial finite spectrum, then $F_{(p)}$ is nontrivial for some prime $p$, so the latter claim implies the former. The $\mathrm{X(m)}_{(p)}$-module $F\otimes \mathrm{X(m)}_{(p)}$ is compact, so if $F\otimes \mathrm{X(m)}_{(p)}$ is complex-orientable, it is a weak $\mathrm{MU}$-module, contradicting Proposition \ref{prop:noweakmodules} if $m<\infty$.
\end{proof}

\subsection{\texorpdfstring{$\mathcal{P}(n)$}{P(n)}-free complexes} An important fact in stable homotopy is the existence of type $n+1$ finite complexes. To construct a type $n+1$ finite complex, it suffices to construct a finite complex $F$ such that $H^*(F;\F_p)$ is free over the subalgebra $\mathcal{A}(n)$ of the Steenrod algebra $\mathcal{A}$. The Steenrod algebra has an even variant $\mathcal{P}$ and corresponding subalgebras $\mathcal{P}(n)$. We mimic the $\mathcal{A}(n)$ case to construct even cell complexes whose cohomologies are free over $\mathcal{P}(n)$. We discuss also the structure of $H^*\mathrm{T(n)}$ as a module over $\mathcal{A}(n)$; this will allow us to more closely relate chromatic defect with Wood-types. 

\begin{definition}
We define the following Hopf algebras.
\begin{itemize}
    \item Let $\mathcal{A}(n)$ be the subalgebra of the Steenrod algebra $\mathcal{A}$ generated by $\{\mathrm{Sq}^{2^i}: 0\le i\le n\}$ for $p=2$ and $\{\beta,\mathrm{P}^{p^i}\st 0\le i\le n-1\}$ for $p>2$.
    \item Let $\mathcal{P}=\mathcal{A}/\langle \mathrm{Sq}^1\rangle$ where $\langle \mathrm{Sq}^1\rangle$ is the 2-sided ideal generated by $\mathrm{Sq}^1$ for $p=2$, and let $\mathcal{P}$ be the subalgebra of $\mathcal{A}$ generated by the $\mathrm{P}^i$'s for $p>2$.
    \item Let $\mathcal{P}(n)$ be the subalgebra of $\mathcal{P}$ generated by $\{\mathrm{Sq}^{2^i}: 1\le i\le n+1\}$ for $p=2$ and $\{\mathrm{P}^{p^i}\st 0\le i\le n\}$ for $p>2$.
\end{itemize}

 \end{definition}

The algebras $\mathcal{A}$ and $\mathcal{P}$ are both Hopf algebras and $\mathcal{A}(n)$ and $\mathcal{P}(n)$ are subHopf algebras. The description of their duals is standard, and we recall this below. As is standard with the Steenrod algebra and its dual, we use the notation $\mathcal{A}(n)_*$ and $\mathcal{P}(n)_*$ to denote the dual of $\mathcal{A}(n)$ and $\mathcal{P}(n)$ respectively.

\begin{proposition}\label{prop:A(n)duals}
    The inclusion $\mathcal{A}(n)\to \mathcal{A}$ is dual to the quotient
    \begin{align*}
        &\F_2[\xi_1,\xi_2,\ldots]\to\F_2[\xi_1,\ldots,\xi_{n+1}]/(\xi_1^{2^{n+1}},\xi_2^{2^{n}},\ldots,\xi_{n+1}^2)&p=2\\
        &\F_p[\xi_1,\xi_2,\ldots]\otimes E(\tau_0,\tau_1,\ldots)\to\F_p[\xi_1,\ldots,\xi_{n}]/(\xi_1^{p^{n}},\xi_2^{p^{n-1}},\ldots,\xi_{n}^p)\otimes E(\tau_0,\ldots,\tau_n)&p>2
    \end{align*}
    The inclusion $\mathcal{P}(n)\to \mathcal{P}$ is dual to the quotient
    \begin{align*}
        &\F_2[\xi_1^2,\xi_2^2,\ldots]\to\F_2[\xi_1^2,\ldots,\xi_{n+1}^2]/({(\xi_1^2)}^{2^{n+1}},{(\xi_2^2)}^{2^{n}},\ldots,{(\xi_{n+1}^2)}^2)&p=2\\
        &\F_p[\xi_1,\xi_2,\ldots]\to\F_p[\xi_1,\ldots,\xi_{n+1}]/(\xi_1^{p^{n+1}},\xi_2^{p^{n}},\ldots,\xi_{n+1}^p)&p>2
    \end{align*}
\end{proposition}

\begin{remark}
    When $p=2$, the Hopf algebra $\mathcal{P}(n)$ is sometimes referred to as $D\mathcal{A}(n)$ -- or the ``double'' of $\mathcal{A}(n)$ -- as there is a degree-doubling isomorphism of Hopf algebras $\mathcal{A}\cong \mathcal{P}$ that restricts to a degree-doubling isomorphism of Hopf algebras $\mathcal{A}(n)\cong \mathcal{P}(n)$.
\end{remark}

The key fact that allows one to construct $\mathcal{A}(n)$- and $\mathcal{P}(n)$-free complexes is the following theorem of Adams--Margolis \cite{adamsmargolis} at $p=2$ and Miller--Wilkerson \cite{millerwilkerson} at $p>2$ (see also \cite[Section 6.2]{ravorange}). We let $P_t^s$ be the element of $\mathcal{A}$ dual to $\xi_t^{p^s}\in \mathcal{A}_*$ with respect to the monomial basis, and for $p>2$ we let $Q_t$ denote the element dual to $\tau_t$.

\begin{theorem}\label{thm:freemodulessteenrod}
    Let $\mathcal{B}$ be a sub Hopf algebra of the Steenrod algebra $\mathcal{A}$. At $p=2$, a $\mathcal{B}$-module $M$ is free if and only if the Margolis homologies $H_*(M;P_t^s)=0$ for all $P_t^s\in\mathcal{B}$, and for $p>2$, $M$ is free if and only if $H_*(M;P_t^s)=H_*(M;Q_t)=0$ for all $P_t^s,Q_t\in\mathcal{B}$.
\end{theorem}

The theorem applies verbatim to $\mathcal{B}=\mathcal{A}(n)$ for any prime $p$ and to $\mathcal{B}=\mathcal{P}(n)$ for $p>2$. For $p=2$, the theorem also applies to $\mathcal{B}=\mathcal{P}(n)\subset \mathcal{P}$ because the degree-doubling isomorphism of Hopf algebras $\mathcal{A}\cong\mathcal{P}$ sends $P_t^s$ to $P_t^{s+1}$. As outlined by Ravenel in \cite{ravorange}, Jeff Smith used idempotents in the group algebra of the symmetric group to construct finite $\mathcal{A}(n)$-free complexes by way of Theorem \ref{thm:freemodulessteenrod}, and we adapt this argument to $\mathcal{P}(n)$.

\begin{proposition}\label{prop:steenrodaction}
    Let $M$ be an $\mathcal{A}$-module on which $P_t^0$ acts nontrivially for $1\le t\le n$. Then there is an $N>>0$ and an idempotent $e\in\Z_{(p)}[\Sigma_N]$ such that 
    \[H_*(eM^{\otimes N};P_t^s)=0\]
    for $s+t\le n+1$ and $(s,t)\neq (0,n+1)$ when $p=2$, and for $s+t\le n$ when $p>2$.
\end{proposition}
\begin{proof}
    By \cite[Appendix C.3.1]{ravorange}, each $P_s^t$ and $Q_t$ generates a subalgebra of $\mathcal{A}$ of the form $E(x)=\F_p[x]/x^2$ or $T(x)=\F_p[x]/x^{p^n}$. If $M$ is a free module over $E(x)$ or $T(x)$, the corresponding Margolis homology $H_*(M;x)$ vanishes. The claim now follows immediately as in the proof of \cite[Appendix C.3.2]{ravorange}, replacing $H^*(X)$ throughout with $M$.
\end{proof}

Taking $M=H^*\CP^{p^{n+1}}$, it follows from the Kunneth formula that a retract of a smash power of $\CP^{p^{n+1}}$ is free over $\mathcal{P}(n)$.

\begin{theorem}\label{thm:P(n)freecomplexes}
There is a finite complex $F$ such that $F$ is a retract of $(\CP^{p^{n+1}})^{\otimes N}$ for some $N$, and $H^*(F;\F_p)$ is a free $\mathcal{P}(n)$-module.     
\end{theorem}
\begin{proof}
    For $p>2$, $P_t^0$ acts nontrivially on $H^*(\CP^{p^{n+1}};\F_p)$ for $t\le n+1$, hence by combining Proposition \ref{prop:steenrodaction} and Theorem \ref{thm:freemodulessteenrod}, there exists a retract $F$ of $(\CP^{p^{n+1}})^{\otimes N}$ for some $N$ such that $H^*(F;\F_p)$ is free over $\mathcal{P}(n)$.

    For $p=2$, $P_t^0$ acts nontrivially on $H^*(\mathbb{RP}^{2^{n+1}};\F_2)$ for $t\le n+1$, and hence there exists a retract $F$ of $(\mathbb{RP}^{2^{n+1}})^{\otimes N}$ for some $N$ such that $H^*(F;\F_2)$ is free over $\mathcal{A}(n)$. Now, there is a degree-doubling symmetric-monoidal equivalence of categories $\mathrm{Mod}(\mathcal{A})\simeq\mathrm{Mod}(\mathcal{P})$ which sends $H^*(\mathbb{RP}^{2^{n+1}};\F_2)$ to $H^*(\mathbb{CP}^{2^{n+1}};\F_2)$. Since the degree-doubling isomorphism $\mathcal{A}\cong\mathcal{P}$ sends $\mathcal{A}(n)$ to $\mathcal{P}(n)$, it follows that using the same idempotent to split $(\mathbb{CP}^{2^{n+1}})^{\otimes N}$ gives a $\mathcal{P}(n)$-free summand.
\end{proof}

\begin{remark}
This theorem constructs \emph{some} complex $F$ whose cohomology is free over $\mathcal{P}(n)$, but it can be difficult to work with $F$ in a direct manner. However, in some cases one can do better via a direct construction. The $2$-cell complex $\mathbb{S}/\eta$ at $p=2$ and $\mathbb{S}/\alpha_1$ at $p>2$ has cohomology free of rank 1 over $\mathcal{P}(0)$. At $p=2$, one may construct a complex called $D\mathcal{A}(1)$ which is free of rank 1 over $\mathcal{P}(1)$, see \cite{akhiltmf}.
\end{remark}

These complexes may be used to construct Wood equivalences for certain fp spectra in a wide generality, which we will study in Section \ref{sec:4}. The main fact we use is that, as an $\mathcal{A}(n)_*$-comodule, $\mathcal{P}(n-1)_*$ is coinduced from an exterior algebra. In the following, recall that $\mathcal{E}(n)_*$ denotes the quotient Hopf algebra of the dual Steenrod algebra given by $E(\xi_1,\ldots,\xi_{n+1})$ when $p=2$ and $E(\tau_0,\ldots,\tau_{n})$ when $p>2$.

\begin{proposition}\label{prop:A(n)toE(n)cor}
    Suppose that $M$ is an $\mathcal{A}_*$-comodule concentrated in even degrees and that $M$ is cofree as a $\mathcal{P}(n-1)_*$-comodule, i.e. that there is an isomorphism of $\mathcal{P}(n-1)_*$-comodules $M\cong \mathcal{P}(n-1)_*\otimes V$ for some $\F_p$-vector space $V$. Then there is an isomorphism of $\mathcal{A}(n)_*$-comodules
    \[M\cong \mathcal{A}(n)_*\Boxover{\mathcal{E}(n)_*}V\]
\end{proposition}
\begin{proof}
    Note that an $\mathcal{A}(n)_*$-comodule concentrated in even degrees is the same data as a $\mathcal{P}(n-1)_*$-comodule concentrated in even degrees. The inclusion of $\mathcal{A}(n)_*$-comodules $\mathcal{P}(n-1)_*\to \mathcal{A}(n)_*$ factors through an isomorphism
    \[\mathcal{P}(n-1)_*\to\mathcal{A}(n)_*\Boxover{\mathcal{E}(n)_*}\F_p\]
    of $\mathcal{A}(n)_*$-comodules, as follows from Proposition \ref{prop:A(n)duals}. One now uses the isomorphism
    \[(\mathcal{A}(n)_*\Boxover{\mathcal{E}(n)_*}\F_p)\otimes V\cong \mathcal{A}(n)_*\Boxover{\mathcal{E}(n)_*}V\]
\end{proof}

Note that the $\mathcal{P}(n-1)$-free complexes constructed in Theorem \ref{thm:P(n)freecomplexes} are in fact $\mathrm{T(n)}$-projectives, since $\CP^{p^{n}}$ is $\mathrm{T(n)}$-free. This relationship is reflected on homology as follows. When $H_*\mathrm{T(n)}$ is restricted to $\mathcal{A}(n)_*$, since it is concentrated in even degrees, it is equivalently a $\mathcal{P}(n-1)_*$-comodule. Just like our $\mathcal{P}(n-1)$-free complexes, $H_*\mathrm{T(n)}$ is cofree as a $\mathcal{P}(n-1)_*$-comodule. 

\begin{proposition}\label{prop:homologyT(n)overA(n)}
    There is an isomorphism of $\mathcal{A}(n)_*$-comodules
    \[H_*\mathrm{T(n)}\cong \mathcal{P}(n-1)_*\otimes\F_p[t_1^{2^n},t_2^{2^{n-1}},\ldots,t_n^2]\cong \mathcal{A}(n)_*\Boxover{\mathcal{E}(n)_*}\F_p[t_1^{2^n},t_2^{2^{n-1}},\ldots,t_n^2]\]
    where both appearances of $\F_p[t_1^{2^n},t_2^{2^{n-1}},\ldots,t_n^2]$ have trivial coactions.
\end{proposition}
\begin{proof} As before, we regard all $\mathcal{A}(n)_*$-comodules concentrated in even degrees equivalently as $\mathcal{P}(n-1)_*$-comdules. The map of $\mathcal{A}(n)_*$-comodule algebras 
    \[H_*\mathrm{T(n)}\to H_*\F_p=\mathcal{A}_*\to \mathcal{A}(n)_*\]
    has image $\mathcal{P}(n-1)_*$, and thus factors through a surjection of $\mathcal{P}(n-1)_*$-comodules $H_*\mathrm{T(n)}\to \mathcal{P}(n-1)_*$. By the Milnor--Moore theorem (see e.g. \cite[A1.1.20]{ravgreen}), there is an isomorphism of $\mathcal{P}(n-1)_*$-comodules $H_*\mathrm{T(n)}\cong \mathcal{P}(n-1)_*\otimes (H_*\mathrm{T(n)}\Boxover{\mathcal{P}(n-1)_*}\F_p)$, where the latter has the extended comodule structure (i.e. $H_*\mathrm{T(n)}\Boxover{\mathcal{P}(n-1)_*}\F_p$ is simply regarded as a vector space). 
    
    One may now use Proposition \ref{prop:A(n)duals} to compute that
    \[H_*\mathrm{T(n)}\Boxover{\mathcal{P}(n-1)_*}\F_p\cong \F_p[t_1^{2^n},t_2^{2^{n-1}},\ldots,t_n^2]\]
    to obtain the first claimed isomorphism and directly apply Proposition \ref{prop:A(n)toE(n)cor} to obtain the second.
\end{proof}

\begin{remark}
The splitting of the proposition cannot be made into an isomorphism of $\mathcal{A}_*$-comodules because $H^*\mathrm{T(n)}$ supports arbitrarily many $\mathrm{Sq}^{2^i}$'s for $p=2$ and $P^{p^i}$'s for $p>2$.
\end{remark}

\section{Chromatic defect and Wood equivalences for fp spectra}\label{sec:4}
In this section, we will discuss the behavior of chromatic defect on a special class of spectra called fp spectra, which we now recall.

\begin{definition}(\cite[Proposition 3.2]{mr})\label{def:fpspectra}
    A $p$-complete, bounded below spectrum $E$ is said to be an fp spectrum if it satisfies one of the following equivalent conditions
    \begin{itemize}
        \item There exists a finite $p$-local spectrum $F$ such that $E\otimes F$ is $\pi$-finite. That is $|\pi_*(E\otimes F)|<\infty$.
        \item There exists a finite $p$-local spectrum $F$ such that $E\otimes F$ is a finite sum of shifts of $\F_p$.
        \item The exists a finite $\mathcal{A}(n)_*$-comodule $M$ and an isomorphism of $\mathcal{A}_*$-comodules $H_*(E;\F_p)\cong\mathcal{A}_*\Boxover{\mathcal{A}(n)_*}M$.
        \item $H^*(E;\F_p)$ is a finitely presented $\mathcal{A}$-module.
    \end{itemize}
\end{definition}

One often considers fp spectra, such as $\mathrm{ko}$ and $\mathrm{tmf}$ for example, because they are connective spectra with strong finiteness properties that make them amenable to computation. The isomorphisms in Definition \ref{def:fpspectra} give rise to change-of-rings isomorphisms for the Adams spectral sequence (ASS) that give one considerable computational control. The resulting computations are interesting in their own right (see \cite{brunerrognes} for example) and allow for naturality arguments to deduce behavior in the Adams spectral sequence for more complicated theories, such as the sphere. Moreover these theories often provide connective models for key finite-height chromatic spectra of interest, such as the $\mathrm{EO}_n(G)$'s, which are otherwise inaccessible from the point of view of the Adams spectral sequence because their mod $p$ homology vanishes.

It is difficult, however, to say things in general about the behavior Adams--Novikov spectral sequence of an fp spectrum, such as the existence of vanishing lines or change-of-rings isomorphisms. We discuss some results of this form in this section in certain nice cases of fp spectra, by giving obstructions in the Adams spectral sequence to an fp spectrum having finite chromatic defect and by using the $\mathcal{P}(n)$-free complexes of Theorem \ref{thm:P(n)freecomplexes} to give sufficient conditions for an fp spectrum to be Wood-type.

\subsection{Chromatic defect for fp spectra}
Proposition \ref{prop:homologyT(n)overA(n)} allows us to use a change-of-rings isomorphism to identify the homology of $E\otimes \mathrm{T(n)}$ when $E$ is an fp spectrum.

\begin{proposition}\label{prop:CORwithfp}
Suppose $E$ is an fp spectrum so that $H_*(E;\F_p)\cong\mathcal{A}_*\Boxover{\mathcal{A}(n)_*}M$. There is an isomorphism of $\mathcal{A}_*$-comodules
\[H_*(E\otimes \mathrm{T(n)})\cong (\mathcal{A}_*\Boxover{\mathcal{E}(n)_*}M)\otimes \F_p[t_1^{2^n},t_2^{2^{n-1}},\ldots,t_n^2] \]
where $\F_p[t_1^{2^n},t_2^{2^{n-1}},\ldots,t_n^2]$ has trivial coaction. In particular, the ASS of $E\otimes\mathrm{T(n)}$ has signature
\[E_2=\Ext_{\mathcal{E}(n)_*}(\F_p,M)\otimes\F_p[t_1^{2^n},t_2^{2^{n-1}},\ldots,t_n^2]\implies \pi_*(E\otimes\mathrm{T(n)})\]
\end{proposition}
\begin{proof}
    The first claimed isomorphism follows directly from the Kunneth isomorphism and Proposition \ref{prop:homologyT(n)overA(n)}.
    A change of rings isomorphism (as in \cite[A1.3.13]{ravgreen}) for the Hopf algebra quotient $\mathcal{A}_*\to\mathcal{E}(n)_*$ now implies an isomorphism
    \[\Ext_{\mathcal{A}_*}(\F_p,(\mathcal{A}_*\Boxover{\mathcal{E}(n)_*}M)\otimes \F_p[t_1^{2^n},t_2^{2^{n-1}},\ldots,t_n^2])\cong\Ext_{\mathcal{E}(n)_*}(\F_p,M\otimes \F_p[t_1^{2^n},t_2^{2^{n-1}},\ldots,t_n^2])\]
    Finally, since $\F_p[t_1^{2^n},t_2^{2^{n-1}},\ldots,t_n^2]$ has trivial coaction, it follows directly from the cobar complex that 
\[\Ext_{\mathcal{E}(n)_*}(\F_p,M\otimes \F_p[t_1^{2^n},t_2^{2^{n-1}},\ldots,t_n^2])\cong \Ext_{\mathcal{E}(n)_*}(\F_p,M)\otimes \F_p[t_1^{2^n},t_2^{2^{n-1}},\ldots,t_n^2]\]
\end{proof}

The upshot of this change of rings isomorphism is that computing Ext over the exterior algebra $\mathcal{E}(n)_*$ is much easier than over $\mathcal{A}(n)_*$. If $E$ is a homotopy associative ring spectrum, we can fit the obstructions of Proposition \ref{prop:defectringspectrum} into this picture and obtain the following.

\begin{corollary}\label{cor:fpobstructions}
    If $E$ is an fp homotopy associative ring spectrum, so that $H_*(E;\F_p)\cong\mathcal{A}_*\Boxover{\mathcal{A}(n)_*}M$, and
    \[\Ext_{\mathcal{E}(n)_*}^{s,2p^{m+1}-3-2(p-1)*+s}(\F_p,M)=0\]
    for $s\ge2$ and $m\ge n$, then $\Phi_p(E)\le n$.
\end{corollary}
\begin{proof}
    The obstructions to having $\Phi_p(E)\le n$ lie in the $2p^{m+1}-3$ stems for $m\ge n$ by Proposition \ref{prop:defectringspectrum}, and the graded vector space $\F_p[t_1^{2^n},t_2^{2^{n-1}},\ldots,t_n^2]$ is concentrated in degrees $2(p-1)*$, so by Proposition \ref{prop:CORwithfp}, it remains only to explain why one may take $s\ge2$. This is because the obstruction $\chi_{m+1}$ is detected by $h_{m+1,1}$ when $p=2$ and $h_{m+1,0}$ when $p>2$ in the cobar complex for $H_*\mathrm{T(m)}$. The map 
    \[\Ext_{\mathcal{A}_*}(\F_p,H_*\mathrm{T(m)})\to \Ext_{\mathcal{A}(n)_*}(\F_p,H_*\mathrm{T(m)})\]
    sends this class to zero when $m\ge n$. Indeed, this follows from the fact that the quotient map $\mathcal{A}_*\to\mathcal{A}(n)_*$ sends $\xi_{m+1}^2\mapsto0$ when $p=2$ and $\xi_{m+1}\mapsto0$ when $p>2$ for $m\ge n$, by Proposition \ref{prop:A(n)duals}. Since $\chi_{m+1}$ has Adams filtration 1 in $\pi_*\mathrm{T(m)}$, it must therefore be detected in filtration $>1$ in $\pi_*E\otimes \mathrm{T(m)}$.
\end{proof}

\begin{remark}\label{rmk:defectandfpupward}
    In the preceding proposition and corollary, and throughout this section, we often fix an isomorphism $H_*(E;\F_p)\cong\mathcal{A}_*\Boxover{\mathcal{A}(n)_*}M$ for an fp spectrum $E$ and study the condition that $\Phi_p(E)\le n$. Asking that these two numbers agree does not result in much loss of generality: indeed, if there is an isomorphism $H_*(E;\F_p)\cong\mathcal{A}_*\Boxover{\mathcal{A}(n)_*}M$, then for any $m\ge n$, there is an isomorphism $H_*(E;\F_p)\cong\mathcal{A}_*\Boxover{\mathcal{A}(m)_*}M'$. Similarly, of course, if $\Phi_p(E)\le n$, then $\Phi_p(E)\le m$ as in Remark \ref{rmk:defectupward}
\end{remark}

\begin{example}
    Suppose that $E$ is an fp homotopy associative ring spectrum and that $M$ is concentrated in even degrees. It follows that $M$ has trivial coaction as an $\mathcal{E}(n)_*$-comodule, since the generators of $\mathcal{E}(n)_*$ are in odd degrees. Therefore 
    \[\Ext_{\mathcal{E}(n)_*}(\F_p,M)\cong \Ext_{\mathcal{E}(n)_*}(\F_p,\F_p)\otimes M\]
    and since $M$ is concentrated in even degrees $t$ and $s=0$ and $\Ext_{\mathcal{E}(n)_*}(\F_p,\F_p)$ is concentrated in bidegrees $(s,t)$ where $t-s$ is even, it follows that $\Ext^{s,t}_{\mathcal{E}(n)_*}(\F_p,M)=0$ bidegrees $(s,t)$ where $t-s$ is even. By Corollary \ref{cor:fpobstructions}, we then have that $\Phi_p(E)\le n$.

    We can in this way immediately recover Hopkins' examples from Example \ref{example:kotmfdefect} using that 
    \[H_*\mathrm{ko}=\mathcal{A}_*\Boxover{\mathcal{A}(1)_*}\F_2\]
    and 
    \[H_*\mathrm{tmf}=\mathcal{A}_*\Boxover{\mathcal{A}(2)_*}\F_2\]
\end{example}

\begin{example}
    One may take the spectrum $\mathrm{BP}_\R\langle 2\rangle^{C_2}$ -- where $\mathrm{BP}_\R\langle 2\rangle$ is any form of the second truncated Real Brown-Peterson spectrum (e.g. $\mathrm{tmf}_0(3)=\mathrm{tmf}_1(3)^{C_2}$ by \cite{hillmeier}) -- as an example where $M$ is not itself even but $\Ext_{\mathcal{E}(n)_*}(\F_p,M)$ is concentrated in even stems. Indeed, by \cite[Theorem 1.9]{chr}, we have a decomposition
    \[H_*(\mathrm{BP}_\R\langle 2\rangle^{C_2})\cong \mathcal{A}_*\Boxover{\mathcal{A}(2)_*}M\]
    where $M$ is a certain 10 dimensional comodule described in \emph{loc. cit.} It can be checked by hand or with an Ext resolver that $\Ext_{\mathcal{E}(2)_*}(\F_p,M)$ is even.
\end{example}

\subsection{Wood-type fp spectra}
The description of $H_*(E\otimes \mathrm{T(n)})$ of Proposition \ref{prop:CORwithfp} may be used to give an algebraic Wood equivalence for any fp spectrum $E$ with finite chromatic defect. In this section, we discuss these algebraic equivalences and a sufficient condition to lift them to the spectrum level. The algebraic equivalence is a simple consequence of Proposition \ref{prop:A(n)toE(n)cor}. In the following statement, we fix an $\mathcal{A}_*$-comodule structure extending the canonical $\mathcal{A}(n)_*$-comodule structure on $\mathcal{P}(n-1)_*$. Such structures exist by \cite[Theorem A]{mitchell}.

\begin{proposition}
    Let $E$ be an fp spectrum so that $H_*E\cong\mathcal{A}_*\Boxover{\mathcal{A}(n)_*}M$. Then there is an isomorphism of $\mathcal{A}_*$-comodules 
    \[H_*E\otimes \mathcal{P}(n-1)_*\cong\mathcal{A}_*\Boxover{\mathcal{E}(n)_*}M\]
\end{proposition}

In particular, by Proposition \ref{prop:homologyT(n)overA(n)}, $H_*E\otimes\mathcal{P}(n-1)_*$ is a retract of $H_*(E\otimes \mathrm{T(n)})$ via the projection $\F_p[t_1^{2^n},t_2^{2^{n-1}},\ldots,t_n^2]\to\F_p\{1\}$. Hence if $E$ has finite chromatic defect, $H_*E$ is Wood-type in the sense that $\mathcal{P}(n-1)_*$ is a finite comodule concentrated in even degrees, and $H_*(E)\otimes\mathcal{P}(n-1)_*$ is a weak $H_*\mathrm{MU}$-module. 

\begin{remark}\label{rmk:modtau}
    This algebraic Wood equivalence can be given a proper home in Hovey's category $\mathrm{Stable}(\mathcal{A})$, equivalently, by working mod $\tau$ in Pstragowski's category $\mathrm{Syn}_{\F_p}$ \cite{piotr}. However, we will not make use of this, and therefore comment on it only briefly.
    
    Indeed one may define complex-orientability in terms of the attaching maps $\sigma_k:\mathbb S^{2k+1,1}\to\nu\CP^k$ as in Definition \ref{def:comporient}, and one has an immediate analog of Theorem \ref{thm:weakMU} in both $\mathrm{Syn}_{\F_p}$ and $\mathrm{Stable}(\mathcal{A})$ with $\nu \mathrm{MU}$. According to the analogous definition of Wood-types using compact $\nu \mathrm{BP}$-projectives, if $E$ is an fp spectrum with finite chromatic defect, the above proposition then implies that $\nu E/\tau$ is Wood-type in $\mathrm{Stable}(\mathcal{A})$.
\end{remark}

To lift these algebraic Wood equivalences to the spectrum level, we will need two main ingredients. First we must lift the comodule $\mathcal{P}(n-1)_*$ to a finite $\mathrm{BP}$-projective; this is accomplished, up to taking shifts and direct sums, by the complexes of Theorem \ref{thm:P(n)freecomplexes}. Then we must lift the classes in $\F_p[t_1^{2^n},\ldots,t_n^2]$. Such lifts are guaranteed by the collapse of $\mathrm{ASS}(E\otimes \mathrm{BP})$.

\begin{theorem}\label{thm:fpwoodsplittings}
    Let $E$ be an fp homotopy associative ring spectrum with finite chromatic defect. If $\mathrm{ASS}(E\otimes \mathrm{BP})$ collapses on $E_2$, then $E$ is Wood-type.
\end{theorem}
\begin{proof}
    Since $E$ is an fp spectrum, we have an isomorphism $H_*E\cong\mathcal{A}_*\Boxover{\mathcal{A}(n)_*}M$ and we have assumed $\Phi_p(E)\le m$ for some $m$. As in Remark \ref{rmk:defectandfpupward}, we can take $m=n$. By Proposition \ref{prop:T(n)toBPredundancy}, we have a splitting $E\otimes \mathrm{BP}\simeq E\otimes \mathrm{T(n)}[t_{n+1},t_{n+2},\ldots]$ so that $\mathrm{ASS}(E\otimes \mathrm{BP})$ collapses on $E_2$ if and only if $\mathrm{ASS}(E\otimes \mathrm{T(n)})$ does. 

    Now we fix a finite complex $F$ with the property that $F$ is a retract of $(\CP^{p^n})^{\otimes N}$ for some $N$ and that $H^*(F;\F_p)$ is a free $\mathcal{P}(n-1)$-module, using Theorem \ref{thm:P(n)freecomplexes}. As a retract of $(\CP^{p^n})^{\otimes N}$, $F$ is a finite $\mathrm{BP}$-projective, and the AHSS computing $[F,\mathrm{T(n)}]$ collapses on $E_2$, as it is a retract of the corresponding spectral sequence for $(\CP^{p^n})^{\otimes N}$. Fixing a basis $\{b_1,\ldots,b_k\}$ of the free $\mathcal{P}(n-1)$-module $H^*(F;\F_p)$ with $|b_i|=n_i$, we thereby fix a map $\iota:F\to \mathrm{T(n)}\{b_1,\ldots,b_k\}$ lifting the element in
    \[\bigoplus_{j} H^{j}(F;\mathrm{T(n)}\{b_1,\ldots,b_k\}^{-j})\cong\bigoplus_{j}\Hom_\Z(H_{j}(F),\mathrm{T(n)}\{b_1,\ldots,b_k\}_{j})\]
    that sends $b_i\mapsto b_i$, and is zero on all other basis elements, using that $H_j(F)$ is a finitely generated free abelian group.

    One may determine the effect of the map $\iota$ in homology as follows. The composition
    \[\Sigma^{n_i}\mathcal{P}(n-1)_*\xrightarrow{b_i}H_*F\xrightarrow{\iota_*}H_*\mathrm{T(n)}\{b_1,\ldots,b_k\}\xrightarrow{b_j}\Sigma^{n_j}H_*\mathrm{T(n)}\]
    is either zero or an isomorphism in the bottom dimension of the target. Over $\mathcal{A}(n)_*$, the target is cofree on $\F_p[t_1^{2^n},\ldots,t_n^2]$, so when the above map is nonzero, its projection onto the cocylic summand indexed by $1$ must agree with that of the canonical inclusion $\Sigma^{n_i}\mathcal{P}(n-1)_*\to \Sigma^{n_i}H_*T(n)$ of Proposition \ref{prop:homologyT(n)overA(n)}. Using this map along with the ring structure of $H_*T(n)$, we have a map $\Sigma^{n_i}\mathcal{P}(n-1)_*[t_1^{2^n},\ldots,t_n^2]\to H_*T(n)\{b_i\}$ which is therefore surjective by Proposition \ref{prop:homologyT(n)overA(n)} and hence an isomorphism for dimension reasons. 

    Finally, since $\mathrm{ASS}(E\otimes \mathrm{T(n)})$ collapses on $E_2$, there are classes $t_1^{2^n},\ldots,t_n^2\in\pi_*(E\otimes \mathrm{T(n)})$ lifting the corresponding classes in homology provided by Proposition \ref{prop:CORwithfp}. Using these classes, the map $\iota$, and the multiplication on $E\otimes \mathrm{T(n)}$, we may define a map
    \[(E\otimes F)[t_1^{2^n},\ldots,t_n^2]\to E\otimes \mathrm{T(n)}\{b_1,\ldots,b_k\}\]
    that induces an isomorphism in homology by the above argument, and it is thus an equivalence. It follows that $E\otimes F$ is a retract of $E\otimes \mathrm{T(n)}\{b_1,\ldots,b_k\}$ and therefore complex-orientable.
\end{proof}

\begin{corollary}\label{cor:fpeven}
    Let $E$ be an fp homotopy associative ring spectrum so that $H_*E\cong\mathcal{A}_*\Boxover{\mathcal{A}(n)_*}M$ and suppose that
    \[\Ext_{\mathcal{E}(n)_*}^{s,t}(\F_p,M)=0\]
    whenever $t-s$ is odd. Then $\Phi_p(E)\le n$ and $E$ is Wood-type.
\end{corollary}
\begin{proof}
    The evenness conditions guarantee that $\Phi_p(E)\le n$ by Corollary \ref{cor:fpobstructions} and therefore $E\otimes \mathrm{BP}$ is a sum of shifts of $E\otimes\mathrm{T(n)}$ by Proposition \ref{prop:T(n)toBPredundancy}. The latter implies that if $\mathrm{ASS}(E\otimes \mathrm{T(n)})$ collapses, so does $\mathrm{ASS}(E\otimes \mathrm{BP})$. To see that $\mathrm{ASS}(E\otimes \mathrm{T(n)})$ collapses, note that by Proposition \ref{prop:CORwithfp}, $\Ext_{\mathcal{A}_*}^{s,t}(\F_p,H_*(E\otimes\mathrm{T(n)})=0$ unless $t-s$ is even, so there is no room for differentials.
\end{proof}

\begin{example}\label{example:kotmfwoodanddefect}
    As explained in the previous section, $\mathrm{ko}$, $\mathrm{tmf}$, and $\mathrm{BP}_\R\langle 2\rangle^{C_2}$ all satisfy the conditions of Corollary \ref{cor:fpeven}, and thus are seen to be Wood-type. As a nonexample, the connective image of $J$ spectrum $\mathrm{j}$ is an fp spectrum, but is not Wood-type as it does not have finite chromatic defect; we will show this in Section \ref{sec:5}.
\end{example}

\begin{remark}
    The conditions of the above Theorem \ref{thm:fpwoodsplittings}, however, are not necessary. For example, the spectrum $\mathrm{ku}/4$ is a Wood-type (complex-orientable, in fact) fp spectrum with the property that $\mathrm{ASS}(E\otimes \mathrm{BP})$ has nonzero differentials. It seems plausible that the conditions of the theorem may be relaxed to asking for collapse on a finite page.
\end{remark}

\section{Hopkins' Stacks and the \texorpdfstring{$\mathrm{X(n)}$}{X(n)}'s}\label{sec:5} An fp spectrum $E$ admits a change of rings isomorphism
\[\Ext_{\mathcal{A}_*}(\F_p,H_*X)\cong\Ext_{\mathcal{A}(n)_*}(\F_p,M)\]
on the $E_2$-page of the classical Adams spectral sequence of $E$, for some $n$ and $M$. We will see in this section that spectra with finite chromatic defect play a similar role with respect to the Adams--Novikov spectral sequence. In this setting, the relevant change of rings isomorphisms take a more conceptual form, via the language of stacks.

\subsection{Chromatic defect and the Adams--Novikov spectral sequence} Hopkins observed that many of the known change of rings isomorphisms in the Adams--Novikov spectral sequence could be reinterpreted and more conceptually derived in the language of stacks \cite[Chapter 9]{tmfbook}. To any homotopy commutative ring spectrum $E$, we may associate a stack $\mathcal{M}_E$ equipped with a canonical $\mathbb{G}_m$-action and a $\mathbb{G}_m$-equivariant affine morphism $p_E:\mathcal{M}_E\to\mathcal{M}_{FG}(1)$, where $\mathcal{M}_{FG}(1)$ is the moduli stack of formal group laws and strict isomorphisms, viewed as a $\mathbb{G}_m$-torsor over $\mathcal{M}_{FG}$, the moduli stack of formal group laws and \emph{all} isomorphisms.

\begin{definition}\label{def:stack}
    Let $E$ be a homotopy commutative ring spectrum. We let the $\mathbb{G}_m$-stack $\mathcal{M}_E$ be the one associated to the graded Hopf algebroid
    \[(\mathrm{MU}_{2*}E,\mathrm{MU}_{2*}(\mathrm{MU}\otimes E)).\]
\end{definition}

\begin{remark}
   It is often useful to mod out by the $\mathbb{G}_m$-action and get an affine morphism
\[\mathcal{M}_E/\mathbb G_m\to\mathcal{M}_{FG}\]
    We will instead primarily work $\mathbb{G}_m$-equivariantly so that when $E$ is complex-orientable, $\mathcal{M}_E$ is an affine scheme, as opposed to a stack of the form $\Spec(E_{2*})/\mathbb G_m$. This is convenient as we wish to capture more general, non even-periodic cases, such as those arising from Johnson--Wilson theories $\mathrm{E(n)}$.
\end{remark}

By Quillen's theorem on $\mathrm{MU}_*$ \cite{quillen}, one has a $\mathbb{G}_m$-equivariant equivalence $\mathcal{M}_{\mathbb{S}}\simeq M_{FG}(1)$, whereby the map $p_E$ is induced by the unit map $\mathbb{S}\to E$ for $E$ a homotopy commutative ring spectrum. More generally, \emph{any} spectrum $E$ gives rise to a pair of $\mathbb{G}_m$-equivariant quasicoherent sheaves $\mathcal{F}_0(E)$ and $\mathcal{F}_1(E)$ on $\mathcal{M}_{FG}(1)$, corresponding to the comodules $\mathrm{MU}_{2*}E$ and $\mathrm{MU}_{2*+1}E$. When $E$ is a homotopy-commutative ring spectrum, the stack $\mathcal{M}_E$ is realized as a relative Spec construction on the sheaf of algebras $\mathcal{F}_0(E)$, and the sheaves $\mathcal{F}_0(E)$ and $\mathcal{F}_1(E)$ are pushed forward from sheaves on $\mathcal{M}_E$. This gives our change of rings isomorphisms.

\begin{proposition}\label{prop:stackscor}
    There is a graded isomorphism of the $E_2$-page of $\mathrm{ANSS}(E)$
    \[E_2^{*,*}\cong H^*(\mathcal{M}_{FG}(1);\mathcal{F}_0(E)\oplus\mathcal{F}_1(E))\cong H^*(\mathcal{M}_E;\mathcal{O}_{\mathcal{M}_E}\oplus\mathcal{O}_1)\]
    where $\mathcal{O}_1$ is the $\mathbb{G}_m$-equivariant quasicoherent sheaf on $\mathcal{M}_E$ defined by the comodule $MU_{2*+1}E$.
\end{proposition}
\begin{proof}
    This follows by directly analyzing the cobar complex, whose cohomology is the $E_2$-page of the ANSS of $E$ and arises from the cosimplicial object
    \[\mathrm{MU}_*E\implies \mathrm{MU}_*\mathrm{MU}\otimes_{\mathrm{MU}_*} \mathrm{MU}_*E\Rrightarrow \cdots\]
    Since the Hopf algebroid $(\mathrm{MU}_*,\mathrm{MU}_*\mathrm{MU})$ is concentrated in even degrees, this cosimplicial object splits into an even piece
    \[\mathrm{MU}_{2*}E\implies \mathrm{MU}_*\mathrm{MU}\otimes_{\mathrm{MU}_*} \mathrm{MU}_{2*}E\Rrightarrow \cdots\]
    and an odd piece
    \[\mathrm{MU}_{2*+1}E\implies \mathrm{MU}_*\mathrm{MU}\otimes_{\mathrm{MU}_*} \mathrm{MU}_{2*+1}E\Rrightarrow \cdots\]
    The identification $\mathcal{M}_{FG}(1)\simeq \mathcal{M}_{(\mathrm{MU}_*,\mathrm{MU}_*\mathrm{MU})}$ gives the first isomorphism.

    For the second identification, the cohomology of the even piece is also $H^*(\mathcal{M}_E;\mathcal{O}_{\mathcal{M}_E})$ by definition of $\mathcal{M}_E$, and the odd piece is identified with the cohomology of a sheaf on $\mathcal{M}_E$ whose underlying $MU_{2*}E$-module is $MU_{2*+1}E$ via the isomorphism
    \[\mathrm{MU}_*\mathrm{MU}\otimes_{\mathrm{MU}_*} \mathrm{MU}_{2*+1}E\cong(\mathrm{MU}_*\mathrm{MU}\otimes_{\mathrm{MU}_*}MU_{2*}E)\otimes_{MU_{2*}E}MU_{2*+1}E\]
\end{proof}

\begin{example}\label{example:hopkinsstacks}
    When $E=\mathrm{ko}$, the stack $\mathcal{M}_{\mathrm{ko}}$ is given by the moduli stack of quadratic equations $\mathcal{M}_{\mathrm{quad}}$ with strict coordinate transformations. When $E=\mathrm{tmf}$, the stack $\mathcal{M}_{\mathrm{tmf}}$ is given by the moduli stack of cubic equations $\mathcal{M}_{\mathrm{cub}}$ with strict coordinate transformations (see \cite[Chapter 9]{tmfbook} or \cite{akhiltmf}). These spectra have the property that $MU_{2*+1}E=0$, so that for instance $\mathrm{ANSS}(\mathrm{tmf})$ takes the form
    \[H^*(\mathcal{M}_{\mathrm{cub}};\mathcal{O}_{\mathcal{M}_{\mathrm{cub}}})\implies\pi_*\mathrm{tmf}\]
    which, in this form, is sometimes called the elliptic spectral sequence.
\end{example}

The construction of Definition \ref{def:stack} has two features that are especially useful. It has the property that it relates vanishing lines on the $E_2$-page of the ANSS to cohomological vanishing of sheaves over $\mathcal{M}_E$, and it can often be used to determine the $E_2$-page of the ANSS of a tensor product of spectra in terms of the corresponding stacky pullback over $\mathcal{M}_{FG}(1)$.

\begin{proposition}\label{prop:affinenessandpullback}
    Let $E$ and $E'$ be homotopy commutative ring spectra.
    \begin{enumerate}
\item If $E$ is complex-orientable, $\mathcal M_E\simeq \mathrm{Spec}(E_{2*})$, and $p_E$ is the map classifying the formal group over $E_{*}$.
\item Suppose that $\mathrm{MU}_{2*}E'$ is a flat $\mathrm{MU}_*$-module, then one has a pullback square
\[
\begin{tikzcd}
\mathcal M_{E\otimes E'}\arrow[r]\arrow[d]&\mathcal M_{E'}\arrow[d,"p_{E'}"]\\
\mathcal M_{E}\arrow[r,"p_E"']&\mathcal M_{FG}(1)
\end{tikzcd}
\]
\end{enumerate}
\end{proposition}
\begin{proof}
    The first claim follows as in Corollary \ref{cor:complexorientansscollapse} from the $(-1)$-st codegeneracy furnished by the $\mathrm{MU}$-module structure in the coaugmented cosimplicial object
    \[E\to \mathrm{MU}\otimes E\implies \mathrm{MU}\otimes \mathrm{MU}\otimes E\Rrightarrow \mathrm{MU}\otimes\mathrm{MU}\otimes\mathrm{MU}\otimes E\cdots\]
    The second claim follows by a Hopf algebroid manipulation using that $p_{E'}$ is an affine morphism; see \cite[Proposition 6.1.6]{thesis} or \cite[1.6.6]{handbook}.
\end{proof}

Claim (1) in Proposition \ref{prop:affinenessandpullback} implies in particular that the stackiness of $\mathcal{M}_E$ is a measure of the failure of $E$ to be complex-orientable, and indeed our chromatic defect of Definition \ref{def:defect} gives a way of quantifying this. In fact, the definition of chromatic defect is chosen precisely to reflect a specific algebro-geometric measure of the stackiness of $\mathcal{M}_E$ over $\mathcal{M}_{FG}(1)$, and the connection comes via an identification of $\mathcal{M}_{\mathrm{X(n)}}$.

\begin{definition}\label{def:mjet}
We let $\mathcal M_{FG}(n)$ denote the moduli stack of formal groups together with an $n$-jet. That is, 
\[\mathcal M_{FG}(n)(R)=\begin{cases}\textbf{Objects}\text{: formal group laws over }R\\\textbf{Morphisms: }f:F\to G\text{ such that }f(x)\equiv x\mod x^{n+1}\end{cases}\]
\end{definition}

\begin{proposition}\label{prop:stackpullbackX(n)}
    There is an $\mathbb{G}_m$-equivariant equivalence of stacks $\mathcal M_{FG}(n)\simeq \mathcal M_{\mathrm{X(n)}}$ over $\mathcal{M}_{FG}(1)$. In particular, the map $\mathcal M_{FG}(n)\to \mathcal M_{FG}(1)$ is a faithfully flat affine morphism with the property that
\[
\begin{tikzcd}
\mathcal M_{E\otimes \mathrm{X(n)}}\arrow[r]\arrow[d]&\mathcal M_{FG}(n)\arrow[d]\\
\mathcal M_{E}\arrow[r]&\mathcal M_{FG}(1)
\end{tikzcd}
\]
is a pullback, for $E$ a homotopy commutative ring spectrum.
\end{proposition}
\begin{proof}
By the Thom isomorphism, one has an isomorphism of Hopf algebroids
\[(\mathrm{MU}_*\mathrm{X(n)},\mathrm{MU}_*(\mathrm{MU}\otimes \mathrm{X(n)}))\cong(\mathrm{MU}_*[b_1,\ldots,b_{n-1}],\mathrm{MU}_*\mathrm{MU}[b_1,\ldots,b_{n-1}])\]
The map of Hopf algebroids 
\[(\mathrm{MU}_*\mathrm{X(n)},\mathrm{MU}_*(\mathrm{MU}\otimes \mathrm{X(n)}))\to(\mathrm{MU}_*\mathrm{MU},\mathrm{MU}_*(\mathrm{MU}\otimes \mathrm{MU}))\]
is an inclusion, so we may compute $\eta_R(b_i)$ in the latter. In $\mathrm{MU}_*\mathrm{MU}$, the $b_i$ are by definition the coefficients of the canonical strict isomorphism $\eta_L^*F\to\eta_R^*F$ where $F$ is the universal formal group law over $\mathrm{MU}_*$. Note that a map of rings
\[\phi:\mathrm{MU}_*\mathrm{MU}\otimes_{\mathrm{MU}_*}\mathrm{MU}_*\mathrm{MU}\to R\]
corresponds to a sequence
\[F_1\xrightarrow{g}F_2\xrightarrow{f}F_3\]
of isomorphisms of formal group laws over $R$, where $F_i$ is the pushforward of $F_{\mathrm{MU}}$ along the the map $\mathrm{MU}\to \mathrm{MU}\otimes \mathrm{MU}\otimes \mathrm{MU}$ including $\mathrm{MU}$ as the $i$-th tensor factor. Since $\eta_R$ in the Hopf algebroid $(\mathrm{MU}_*\mathrm{MU},\mathrm{MU}_*(\mathrm{MU}\otimes \mathrm{MU}))$ is induced by the map
\[\mathrm{MU}\otimes \mathrm{MU}\simeq \mathrm{MU}\otimes \mathbb{S}\otimes \mathrm{MU}\to \mathrm{MU}\otimes \mathrm{MU}\otimes \mathrm{MU}\]
we have that $\phi(\eta_L(b_i))$ are the coefficients of $g$, and $\phi(\eta_R(b_i))$ are the coefficients of $f\circ g$. Therefore, if $\eta_L(b_i)=\eta_R(b_i)$ for $i\le m-1$, then $f(x)\equiv x\mod x^{n+1}$. It follows that the map
\[\mathcal M_{\mathrm{X(n)}}\to\mathcal M_{FG}(1)\]
factors through an equivalence onto the substack $\mathcal M_{FG}(n)$.
\end{proof}

\begin{corollary}\label{cor:finitephiaffine}
    Let $E$ be a homotopy commutative ring spectrum. If $\Phi(E)\le n$, then the pullback 
    \[\mathcal{M}_E\times_{\mathcal{M}_{FG}(1)}\mathcal{M}_{FG}(n)\]
    is an affine scheme. Conversely, if $E\otimes \mathrm{X(n)}$ is $\mathrm{MU}$ nilpotent complete, and the above pullback is an affine scheme, then $\Phi(E)\le n$.
\end{corollary}
\begin{proof}
    The first claim follows from claim (1) of Proposition \ref{prop:affinenessandpullback} and Proposition \ref{prop:stackpullbackX(n)}. For the second, since the above pullback is an affine scheme, we see that $\mathrm{ANSS}(\CP^k\otimes E\otimes \mathrm{X(n)})$ collapses on the zero line for all $k$ by Proposition \ref{prop:stackscor}. Since $\mathrm{MU}\otimes\sigma_k=0$ for all $k$, the class $\sigma_k$ has Adams--Novikov filtration $>0$, and hence must be zero in $\pi_*\CP^k\otimes E\otimes \mathrm{X(n)}$ since $\mathrm{ANSS}(\CP^k\otimes E\otimes \mathrm{X(n)})$ converges, as $E\otimes \mathrm{X(n)}$ is $\mathrm{MU}$ nilpotent complete. This implies that $E\otimes \mathrm{X(n)}$ is complex-orientable by definition since $E\otimes \mathrm{X(n)}$ is a ring spectrum.
\end{proof}

\subsection{Chromatic defect and quotient stacks} We would like sufficient algebro-geometric conditions on the morphism $\mathcal{M}_E\to\mathcal{M}_{FG}(1)$ that guarantee finite chromatic defect for $E$, by way of Corollary \ref{cor:finitephiaffine}. One can give various finiteness conditions to guarantee this, which apply in particular when $\mathcal{M}_E$ is a quotient stack by a finite group action and also in various connective cases such as $\mathcal{M}_{\mathrm{ko}}$ and $\mathcal{M}_{\mathrm{tmf}}$ from Example \ref{example:hopkinsstacks}. We begin with a necessary condition.

\begin{proposition}\label{prop:defectimpliesfiniteaut}
    Let $E$ be a homotopy commutative ring spectrum with finite chromatic defect. Then, for any algebraically closed field $k$ and any $x\in\mathcal{M}_E(k)$ such that $p_E(x)\in \mathcal{M}_{FG}(1)(k)$ has finite height, the image of the homomorphism
    \[\psi:\mathrm{Aut}_{\mathcal{M}_E(k)}(x)\to \mathrm{Aut}_{\mathcal{M}_{FG}(1)(k)}(p_E(x))\]
    is a finite subgroup.
\end{proposition}
\begin{proof}
    First suppose $k$ has characteristic zero; then $p_E(x)$ is isomorphic to $\widehat{\mathbb{G}_a}$, which has a trivial and therefore finite strict automorphism group. If $k$ has characteristic $p>0$, then since $k=\overline{k}$, Lazard's theorem implies that $p_E(x)$ is isomorphic to the Honda formal group law $F_n$ for some height $n<\infty$, all of whose automorphisms are defined over the finite field $\F_{p^n}$ (see \cite[Lecture 19]{lurielecture}).
    
    Since $E$ has finite chromatic defect, it follows from Corollary \ref{cor:finitephiaffine} that for some $N>0$, the stack
    \[\mathcal{M}_E\times_{\mathcal{M}_{FG}(1)}\mathcal{M}_{FG}(N)\]
    is an affine scheme, and therefore a discrete stack. This implies in particular that if $f(x)$ is in the image of $\psi$, and $f(x)\equiv x\mod x^{N+1}$, then $f(x)=x$. Suppose for the sake of contradiction that the image of $\psi$ is infinite. Since all the automorphisms of $F_n$ are defined over the finite field $\F_{p^n}$, it follows that for any $m$, there are two \textit{distinct} automorphisms $f_1(x)\neq f_2(x)$ in the image of $\psi$ such that $f_1(x)\equiv f_2(x)\mod x^{m+1}$. Indeed, were this not true, then each automorphism in the image of $\psi$ would be uniquely determined by its truncation mod $x^{m+1}$, which contradicts that the image of $\psi$ is infinite as there are only finitely many such polynomials over a finite field.

    Letting $m=N$, and choosing $f_1,f_2$ as above, the automorphism $g:=f_1\circ f_2^{-1}$ has the property that $g(x)\equiv x\mod x^{N+1}$, but $g(x)\neq x$, contradicting discreteness of the above pullback.
\end{proof}

To obtain a converse, one must ask for conditions on the morphism $p_E:\mathcal{M}_E\to\mathcal{M}_{FG}(1)$ to guarantee that
\[\mathcal{P}(n):=\mathcal{M}_E\times_{\mathcal{M}_{FG}(1)}\mathcal{M}_{FG}(n)\]
is an affine scheme for some $n$. If the image of the map$\psi$ appearing in Proposition \ref{prop:defectimpliesfiniteaut} is always finite, $\mathcal{P}(n)$ still may not be discrete. For instance, if $\mathcal{M}_E$ has height $\infty$ points or if the size of the image of $\psi$ is not uniformly bounded above, the stack $\mathcal{P}(n)$ may not be discrete. One strong condition that guarantees affineness, however, is when $\mathcal{M}_E$ is the quotient of a Noetherian ring by a finite group.

\begin{proposition}\label{prop:quotientstackhasfinitechromaticdefect}
    Let $E$ be a homotopy commutative ring spectrum such that $E\otimes \mathrm{X(n)}$ is $\mathrm{MU}$-nilpotent complete, for all $n$ sufficiently large. Suppose there is an equivalence $\mathcal{M}_E\simeq \Spec(R)/G$ for $G$ a finite group and $R$ a Noetherian commutative ring. Then $E$ has finite chromatic defect.
\end{proposition}
\begin{proof}
   Each $g\in G$ acts on the formal group over $R$ via a power series
   \[[g](x)=x+\sum\limits_{n\ge0}b_n(g)x^{n+1}\]
   and the image of any automorphism of $\Spec(R)/G$ in $\mathcal{M}_{FG}(1)$ is obtained from these power series by base change. For each $g$ and $n$, consider the ideal in $R$
   \[J_n(g):=(b_i(g):i<n))\]
   Since $R$ is Noetherian, the $J_n(g)$'s must stabilize for each $g$. It follows from the fact that $G$ is finite that, by taking $n$ large enough, if after base change to any ring $S$, $[g](x)\equiv x\mod x^{n+1}$, then $[g](x)=x$ over $S$. This implies the stack
    \[\mathcal{M}_E\times_{\mathcal{M}_{FG}(1)}\mathcal{M}_{FG}(n)\]
    is discrete. However, if $\mathcal{M}$ is any discrete stack with an affine morphism to $\Spec(R)/G$, then $\mathcal{M}$ is affine. This follows for example from \cite[Theorem A7.1.1]{katzmazur}; see also \cite[Lemma 6.2.6]{thesis}. The claim then follows from Corollary \ref{cor:finitephiaffine}.
\end{proof}

An identification $\mathcal{M}_E\simeq\Spec(R)/G$ suitable to apply Proposition \ref{prop:quotientstackhasfinitechromaticdefect} often corresponds to an identification $E=R^{hG}$, for $R$ a homotopy commutative ring spectrum with $G$-action. In fact, the former identification follows from the latter under the following assumptions.

\begin{proposition}\label{prop:quotientstackandhtpyfixedpoints}
Let $R\in \Sp^{BG}$ be a homotopy commutative ring spectrum with $G$-action so that $R$ is complex-orientable as a spectrum, and set $E=R^{hG}$. Suppose the map $p$ in the diagram
\[
\begin{tikzcd}
\mathrm{Spec}(R_{2*})\arrow[r]\arrow[d]&\mathcal M_{FG}(1)\\
\mathrm{Spec}(R_{2*})/G\arrow[ur,"p"']
\end{tikzcd}
\]
is affine. Then if $\mathrm{MU}^{\otimes s}\otimes E\simeq (\mathrm{MU}^{\otimes s}\otimes R)^{hG}$ for $s=1,2$, there is an equivalence of stacks $\mathcal M_{E}\simeq \mathrm{Spec}(R_{2*})/G$.
\end{proposition}
\begin{proof}
Set $\mathcal N:=\mathrm{Spec}(R_{2*})/G$. By affineness of $p$, we have a diagram of pullback squares\\

\adjustbox{scale=0.89,center}{
\begin{tikzcd}
\mathrm{Spec}\bigg(\Gamma(\mathcal N\times_{\mathcal M_{FG}(1)}\mathrm{Spec}(L)\times_{\mathcal M_{FG}(1)}\mathrm{Spec}(L);\mathcal O)\bigg)\arrow[r]\arrow[d]&\mathrm{Spec}\bigg(\Gamma(\mathcal N\times_{\mathcal M_{FG}(1)}\mathrm{Spec}(L);\mathcal O)\bigg)\arrow[d]\\
\mathrm{Spec}\bigg(\Gamma(\mathcal N\times_{\mathcal M_{FG}(1)}\mathrm{Spec}(L);\mathcal O)\bigg)\arrow[r]\arrow[d]&\mathcal N\arrow[d,"p"]\\
\mathrm{Spec}(L)\arrow[r]&\mathcal M_{FG}(1)
\end{tikzcd}
}
and
\[\Gamma(\mathcal N\times_{\mathcal M_{FG}(1)}\mathrm{Spec}(L);\mathcal O)\cong \Gamma(\mathcal N;p^*\mathcal F_{\mathrm{MU}})\cong H^0(G;\mathrm{MU}_{2*}R)=\mathrm{MU}_{2*}E\]
The last identification follows from the HFPSS
\[H^*(G;\mathrm{MU}_*R)\implies \mathrm{MU}_*E\]
using the equivalence $\mathrm{MU}\otimes E\simeq (\mathrm{MU}\otimes R)^{hG}$. Indeed, the higher cohomology on this $E_2$-page vanishes by a change of rings isomorphism using affineness of the pullback $\mathcal N\times_{\mathcal M_{FG}(1)}\mathrm{Spec}(L)$. Similar considerations apply to the pullback in the upper left corner. 

The top right stack is thereby identified with $\mathcal M_{\mathrm{MU}\otimes E}$, and the top left stack with $\mathcal M_{\mathrm{MU}\otimes \mathrm{MU}\otimes E}$. This gives the identification
\[\mathcal N\simeq\mathcal M_{((\mathrm{MU}\otimes E)_{2*},(\mathrm{MU}\otimes \mathrm{MU}\otimes E)_{2*})}=\mathcal M_{E}\]
\end{proof}

As remarked in \cite[Section 6.2]{mathewmeier}, if $R$ is a Landweber exact $\mathbb E_\infty$-ring, then the map $p$ in Proposition \ref{prop:quotientstackandhtpyfixedpoints} is affine precisely when, for every field valued point $x$ of $\mathrm{Spec}(R_{2*})$, the stabilizer of $x$ in $G$ acts faithfully on $p(x)$. Moreover, the map $E\to R$ must then be a faithful Galois extension so that the condition $\mathrm{MU}^{\otimes s}\otimes E\simeq (\mathrm{MU}^{\otimes s}\otimes R)^{hG}$ holds automatically by \cite[Theorem 5.10]{mathewmeier}, and we have the following.

\begin{corollary}\label{cor:fixedpointspectrafinite}
    Let $R$ be an even-periodic, Landweber exact $\mathbb E_\infty$-ring with $G$-action for a finite group $G$ and $R_0$ Noetherian. Suppose for every field valued point $x:\Spec(k)\to \Spec(R_0)$, the stabilizer of $x$ in $G$ acts faithfully on the corresponding formal group over $k$. Then $E=R^{hG}$ has finite chromatic defect.
\end{corollary}

For example, we recover immediately Corollary \ref{cor:EOnfinitedefect}. We finish by giving a more general converse to Proposition \ref{prop:defectimpliesfiniteaut}, for a class of stacks $\mathcal{M}_E$ that behave in many ways like a quotient by a finite group.

\begin{proposition}\label{prop:finiteautimpliesdefect}
    Let $E$ be a homotopy commutative ring spectrum such that $E\otimes \mathrm{X(n)}$ is $\mathrm{MU}$-nilpotent complete, for all $n$ sufficiently large. Suppose that there is a faithfully flat, finite morphism $q:\Spec(R)\to \mathcal{M}_E$ for some Noetherian commutative ring $R$. Then $E$ has finite chromatic defect. 
\end{proposition}
\begin{proof}
    The argument for Proposition \ref{prop:quotientstackhasfinitechromaticdefect} needs only slight modification. An automorphism in $\mathcal{M}_E$ is given by a point of the affine scheme
    \[\Spec(T)=\Spec(R)\times_{\mathcal{M}_E}\Spec(R)\]
    which is finite over $\Spec(R)$, and therefore Noetherian. The morphism $\mathcal{M}_E\to\mathcal{M}_{FG}(1)$ determines a morphism 
    \[\Spec(T)\to \Spec(L[b_1,b_2,\ldots])=\Spec(L)\times_{\mathcal{M}_{FG}(1)}\Spec(L)\]
    and we may consider the chain of ideals
    \[(b_1)\subset (b_1,b_2)\subset (b_1,b_2,b_3)\subset\cdots\]
in $T$, which must stabilize. As before we conclude that 
    \[\mathcal{P}(n):=\mathcal{M}_E\times_{\mathcal{M}_{FG}(1)}\mathcal{M}_{FG}(n)\]
    is discrete for some $n$. Under the conditions above, it follows that $\mathcal{P}(n)$ is an Artin stack, and pulling back the cover $q$ to $\mathcal{P}(n)$, we may apply \cite[Lemma 69.17.3]{stacks} to see that $\mathcal{P}(n)$ is affine.
\end{proof}

\begin{example}\label{example:tmfkofinitestacks}
    The stacks associated to $\mathrm{ko}$ and $\mathrm{tmf}$ as in Example \ref{example:hopkinsstacks} satisfy the conditions of Proposition \ref{prop:finiteautimpliesdefect}. Indeed, one has the finite flat covers discussed by Hopkins in \cite[Chapter 9]{tmfbook}. For $\mathrm{ko}$ this cover is the morphism $\mathcal{M}_{\mathrm{ku}}\to\mathcal{M}_{\mathrm{ko}}$, and for $\mathrm{tmf}$ this is given 2-locally by the morphism $\mathcal{M}_{\mathrm{tmf}_1(3)}\to\mathcal{M}_{\mathrm{tmf}}$, 3-locally by $\mathcal{M}_{\mathrm{tmf}_1(2)}\to\mathcal{M}_{\mathrm{tmf}}$, and $\mathcal{M}_{\mathrm{tmf}}$ is itself affine $p$-locally for $p>3$. These finite covers predict the Wood equivalences of Example \ref{example:woodtmf}.
\end{example}

\subsection{\texorpdfstring{$\mathrm{K(n)}$}{K(n)}-local stacks and the image of \texorpdfstring{$J$}{J} spectrum} Yet another very useful aspect of Hopkins' stack construction of Definition \ref{def:stack} is that it allows one to access information about various chromatic localizations of a homotopy commutative ring spectrum $E$ in terms of the stack $\mathcal{M}_E$. We refer the reader to \cite[Section 1.6]{handbook} for an excellent discussion of this aspect. Behrens' framework in \emph{loc. cit.} can be used to compute chromatic defect $\mathrm{K(n)}$- or $\mathrm{E(n)}$-locally, and we use this to verify that $\Phi(\mathrm{j})=\infty$.

Behrens extends Hopkins' stack definition to include more general descent spectral sequences into this framework, such as that of $\mathrm{Tmf}$, which do not coincide with the corresponding ANSS. We will not need this, however, and all our stacks are in the ordinary sense of Definition \ref{def:stack}. Behrens gives the following pullback formula following Proposition \ref{prop:affinenessandpullback} (2) for $\mathrm{K(n)}$-localizations, where as before we work instead with the $\mathbb{G}_m$-torsors over his stacks.

\begin{proposition}\label{prop:behrenspullback}
    Let $(\mathcal{M}_{FG})^{[n]}_{(p)}(1)$ be the the formal neighborhood of the locus of formal group laws in characteristic $p$ of exact height $n$ in $\mathcal{M}_{FG}(1)$. Let $E$ be a homotopy commutative ring spectrum such that $\mathrm{MU}_*E$ is flat over $\mathrm{MU}_*$. Then there is a pullback of (formal) stacks
    \[
\begin{tikzcd}
\mathcal M_{L_{\mathrm{K(n)}}E}\arrow[r]\arrow[d]&\mathcal (\mathcal{M}_{FG})^{[n]}_{(p)}(1)\arrow[d]\\
\mathcal M_E\arrow[r]&\mathcal M_{FG}(1)
\end{tikzcd}
\]
\end{proposition}

In particular, if $E=\mathrm{X(n)}$, one may place $\mathcal{M}_{FG}(n)$ in the bottom left corner, and if $L_{\mathrm{K(n)}}E$ is complex-orientable, the pullback will be affine. This yields the following. 

\begin{theorem}\label{thm:K(n)localsphere}
    The $K(n)$-local sphere $L_{K(n)}\mathbb{S}$ has infinite chromatic defect.
\end{theorem}
\begin{proof}
    Suppose $L_{K(n)}\mathbb{S}$ has finite chromatic defect. Then there exists a homotopy ring map $MU\to (L_{K(n)}\mathbb{S})\otimes X(m)$ for some $m<\infty$, and by composing, a homotopy ring map
    \[MU\to (L_{K(n)}\mathbb{S})\otimes X(m)\to L_{K(n)}X(m)\]
     By Proposition \ref{prop:behrenspullback}, it follows that the pullback
    \[\mathcal{M}_{FG}(n)\times_{\mathcal{M}_{FG}(1)}(\mathcal{M}_{FG})^{[n]}_{(p)}(1)\]
    is affine and therefore discrete. As in the proof of Proposition \ref{prop:defectimpliesfiniteaut}, this implies that the automorphism group of the Honda formal group law $F_n$ must be finite, but the Morava stabilizer group $\mathbb{G}_n$ is infinite.
\end{proof}

We finish the section with a connective version of this result at height 1, regarding Mahowald's connective image of $J$ spectrum $\mathrm{j}$. Working $p$-locally, one has Adams operations $\psi^{p+1}$ acting on $\mathrm{ko}$ when $p=2$ and on the Adams summand $\ell$ when $p>2$. When $p=2$, the map $\psi^3-1:\mathrm{ko}\to \mathrm{ko}$ factors through the connective cover $\tau_{\ge 4}\mathrm{ko}$, and when $p>2$, $\psi^{p+1}-1:\ell\to\ell$ factors through $\tau_{\ge 2p-2}\ell$ (see for example \cite{davis}).

\begin{definition}
    When $p=2$, we define
    \[\mathrm{j}=\mathrm{fib}(\psi^3-1:\mathrm{ko}\to \tau_{\ge 4}\mathrm{ko})\]
    and when $p>2$
    \[\mathrm{j}=\mathrm{fib}(\psi^{p+1}-1:\ell\to \tau_{\ge 2p-2}\ell)\]
\end{definition}

The spectrum $\mathrm{j}$ is an fp spectrum in the sense of Definition \ref{def:fpspectra} and in fact admits an $\mathbb E_\infty$-structure. One of its claims to fame is that it is a connective model of the $K(1)$-local sphere; that is $L_{\mathrm{K(1)}}\mathrm{j}=L_{\mathrm{K(1)}}\mathbb{S}$ as $\mathbb E_\infty$-rings at all primes (see for example \cite[Lemma 2.2]{hms}).

\begin{theorem}\label{thm:defectj}
    The connective image of $J$ spectrum $\mathrm{j}$ does not have finite chromatic defect, at any prime $p$.
\end{theorem}
\begin{proof}
    If $\mathrm{j}\otimes \mathrm{X(n)}$ were complex-orientable for some $n$, then so would be $L_{K(1)}\mathrm{j}=L_{K(1)}\mathbb{S}$, contradicting Theorem \ref{thm:K(n)localsphere}.
\end{proof}

Immediately from Corollary \ref{cor:woodimpliesfinitedefect}, we have the following.

\begin{corollary}\label{cor:jnotwoodtype}
    The connective image of $J$ spectrum $\mathrm{j}$ is not Wood-type. That is, there is no finite $\mathrm{BP}$-projective $F$ such that $\mathrm{j}\otimes F$ is complex-orientable.
\end{corollary}

\section{Real-oriented theories and higher real \texorpdfstring{$K$}{K}-theories}\label{sec:6} We compute the chromatic defect of the Real Johnson-Wilson theories $\mathrm{ER(n)}$ and Goerss-Hopkins-Miller theories $\mathrm{EO}_n(G)$, using Corollary \ref{cor:fixedpointspectrafinite}. As in Proposition \ref{prop:quotientstackandhtpyfixedpoints}, this comes down to understanding the quotient stacks $\Spec(\mathrm{E(n)}_*)/C_2$ and $\Spec((\mathrm{E}_n)_*)/G$ and in particular showing that the corresponding morphisms to $\mathcal{M}_{FG}(1)$ are affine.

\subsection{The Real Johnson-Wilson theories} The fixed points $\mathrm{ER(n)}=(\mathrm{E(n)})^{hC_2}$ with respect to the complex conjugation action give higher height versions of periodic real $K$-theory $\mathrm{KO}=\mathrm{ER(1)}$. At higher heights, however, the $C_2$-action on $\mathrm{E(n)}$ is not known to give an equivariant ring spectrum, and hence the fixed point spectrum $\mathrm{ER(n)}$ is not known to be a ring spectrum. Kitchloo--Lorman--Wilson show that $\mathrm{ER(n)}$ has a homotopy commutative ring structure, up to phantom maps \cite{kitchloo}, but no more than this is known at the moment. We will show that $\mathrm{ER(n)}\otimes\mathrm{X(2^n)}$ has homotopy groups concentrated in even degrees; its complex-orientability follows from Proposition \ref{prop:chi} if it is a homotopy associative ring spectrum, so we will instead verify a ring structure on $\mathrm{ER(n)}\otimes\mathrm{X(2^n)}$.

The lack of ring structure also makes it difficult to apply the methods of Mathew--Meier \cite{mathewmeier}, say, to obtain equivalences such as
\[\mathrm{MU}^{\otimes s}\otimes \mathrm{ER(n)}\simeq (\mathrm{MU}^{\otimes s}\otimes \mathrm{E(n)})^{hC_2}\]
as in Proposition \ref{prop:quotientstackandhtpyfixedpoints}. We can use a trick from genuine equivariant homotopy to get around this. 

\begin{lemma}\label{lemma:cofree}
Let $E$ be a Borel-complete genuine $C_p$-spectrum, $X$ a spectrum, and let $i_*:\Sp\to\Sp^{C_p}$ denote the unique symmetric monoidal colimit preserving functor. Then the natural map
\[E^{hC_p}\otimes X\to (E\otimes i_*X)^{hC_p}\] 
is an equivalence if and only if $E\otimes i_*X$ is also Borel complete. This is always true if $E$ is a module over a $C_p$-ring spectrum $R$ such that $\Phi^{C_p}(R)\simeq*$.
\end{lemma}

\begin{proof}
This is immediate from the fact that $E^{C_p}\simeq E^{hC_p}$ and $(-)^{C_p}$ commutes with colimits. For the second claim, note that $\Phi^{C_p}(E\otimes i_*X)$ and $(E\otimes i_*X)^{tC_p}$ are both modules over $\Phi^{C_p}(R)$.
\end{proof}

We now state the crucial algebro-geometric input necessary to compute $\Phi(\mathrm{ER(n)})$.

\begin{proposition}\label{prop:ERnmultiplication}
    Suppose that the pullback stack
    \[\Spec(\mathrm{E(n)}_*)/C_2\times_{\mathcal{M}_{FG}(1)}\mathcal{M}_{FG}(m)\]
    is an affine scheme. Then $\mathrm{ER(n)}\otimes \mathrm{X(m')}$ admits a Landweber exact homotopy commutative ring structure for all $m'\ge m$, and in particular $\Phi(\mathrm{ER(n)})\le m$.
\end{proposition}
\begin{proof}
    Let $\Spec(R_*)$ denote the affine scheme given by the above pullback; then the map $\Spec(R_*)\to\mathcal{M}_{FG}(1)$ is flat. Indeed this follows from the fact that $\mathcal{M}_{FG}(n)\to \mathcal{M}_{FG}(1)$ is flat and $\Spec(\mathrm{E(n)}_*)/C_2\to \mathcal{M}_{FG}(1)$ is flat. The former is Proposition \ref{prop:stackpullbackX(n)} and the latter follows from the facts that $\Spec(\mathrm{E(n)}_*)\to\Spec(\mathrm{E(n)}_*)/C_2$ is faithfully flat and $\Spec(\mathrm{E(n)}_*)\to \mathcal{M}_{FG}(1)$ is flat as $E(n)$ is Landweber exact.

    There is therefore a Landweber exact homotopy commutative ring spectrum $R$ whose formal group induces the above map $\Spec(R_*)\to\mathcal{M}_{FG}(1)$. We can identify the spectra $R$ and $\mathrm{ER(n)}\otimes\mathrm{X(m)}$ by producing an isomorphism of homology theories
    \[\mathrm{MU}_*(Z)\otimes_{\mathrm{MU}_*}R_*\to (\mathrm{ER(n)}\otimes\mathrm{X(m)})_*(Z)\]
    Since
    \[\Spec(\mathrm{E(n)}_*)/C_2\times_{\mathcal{M}_{FG}(1)}\mathcal{M}_{FG}(m)\]
    is an affine scheme, the homotopy fixed point spectral sequence
    \[H^*(C_2;(\mathrm{E(n)}\otimes\mathrm{X(m)})_*Z)\implies (\mathrm{ER(n)}\otimes\mathrm{X(m)})_*(Z)\]
    is concentrated on the zero-line, where we have used Lemma \ref{lemma:cofree} to identify the target of the spectral sequence. The natural map 
    \[(\mathrm{ER(n)}\otimes\mathrm{X(m)})_*(Z)\to ((\mathrm{E(n)}\otimes\mathrm{X(m)})_*(Z))^{C_2}\]
    is then an isomorphism. Note now that the map
    \[\Spec((\mathrm{E(n)}\otimes\mathrm{X(m)})_*)\to\Spec(R_*)\]
    is a $C_2$-torsor, as it is the base change of the $C_2$-torsor $\Spec(\mathrm{E(n)}_*)\to\Spec(\mathrm{E(n)}_*)/C_2$ along the map $\mathcal{M}_{FG}(m)\to\mathcal{M}_{FG}(1)$. This gives a natural identification
    \[((\mathrm{E(n)}\otimes\mathrm{X(m)})_*(Z))^{C_2}\cong (\mathrm{MU}_*Z\otimes_{\mathrm{MU}_*}(\mathrm{E(n)}\otimes\mathrm{X(m)})_*)^{C_2}\cong \mathrm{MU}_*Z\otimes_{\mathrm{MU}_*}R_*\]
    using in the first identification that $\mathrm{E(n)}\otimes\mathrm{X(m)}$ is Landweber exact.

    Since $\mathcal{M}_{FG}(m')\to\mathcal{M}_{FG}(m)$ is an affine morphism for $m'\ge m$, if the pullback in the statement of the proposition is an affine scheme, then so is 
    \[\Spec(\mathrm{E(n)}_*)/C_2\times_{\mathcal{M}_{FG}(1)}\mathcal{M}_{FG}(m')\]
    and the same argument now applies to $\mathrm{ER(n)}\otimes \mathrm{X(m')}$.
\end{proof}

Under the assumptions of Proposition \ref{prop:ERnmultiplication}, we can let $m=\infty$, and we deduce that $\mathrm{ER(n)}\otimes \mathrm{MU}$ is a Landweber exact homotopy commutative ring spectrum, and hence we can associate a stack to $\mathrm{ER(n)}$ exactly as in Definition \ref{def:stack}, we set
\[\mathcal{M}_{\mathrm{ER(n)}}:=\mathcal{M}_{(\mathrm{MU}_{2*}\mathrm{ER(n)},\mathrm{MU}_{2*}\mathrm{MU}\otimes_{\mathrm{MU}_{2*}}\mathrm{MU}_{2*}\mathrm{ER(n)})}\]
All the usual properties for $\mathcal{M}_E$ hold as well for $\mathcal{M}_{\mathrm{ER(n)}}$, such as Proposition \ref{prop:affinenessandpullback}, since these properties only ever make use of the ring structure on $\mathrm{MU}_{2*}E$.

\begin{proposition}\label{prop:stacker(n)}
    Assuming that, for some $m$, 
    \[\Spec(\mathrm{E(n)}_*)/C_2\times_{\mathcal{M}_{FG}(1)}\mathcal{M}_{FG}(m)\]
    is an affine scheme, there is an equivalence of $\mathbb{G}_m$-stacks
    \[\mathcal{M}_{\mathrm{ER(n)}}\simeq\Spec(\mathrm{E(n)}_*)/C_2\]
    over $\mathcal{M}_{FG}(1)$.
\end{proposition}
\begin{proof}
    Since the morphism $\mathcal{M}_{FG}(m)\to \mathcal{M}_{FG}(1)$ is faithfully flat, the assumptions imply that the map $p:\Spec(\mathrm{E(n)}_*)/C_2\to\mathcal{M}_{FG}(1)$ is affine. This, along with Lemma \ref{lemma:cofree}, imply that the conditions of Proposition \ref{prop:quotientstackandhtpyfixedpoints} are satisfied.
\end{proof}

\begin{remark}
    Hopkins gives a description of $\mathcal{M}_{\mathrm{KO}}$ as the moduli stack of nonsingular quadratic equations and strict coordinate transformations, as in Example \ref{example:hopkinsstacks}. In fact, this description falls immediately out of Proposition \ref{prop:stacker(n)} using that $\mathrm{KO}\simeq\mathrm{ER(1)}$. It is possible to make similar identifications to give a modular description of $\mathcal{M}_{\mathrm{ER(n)}}$ for all heights $n$, which we give in \cite{thesis} and intend to return to in future work. 
\end{remark}

We turn now to proving that the pullback in Proposition \ref{prop:ERnmultiplication} is indeed affine.

\begin{theorem}\label{thm:defectern}
    For all $m\ge2^n$, $\mathrm{ER(n)}\otimes\mathrm{X(m)}$ is a Landweber exact homotopy commutative ring spectrum. Moreover, $\Phi(\mathrm{ER(n)})=2^n$.
\end{theorem}
\begin{proof}
    We begin by showing that $\Phi(\mathrm{ER(n)})\le 2^n$; we need to show that the pullback 
    \[\mathcal{P}:=\Spec(\mathrm{E(n)}_*)/C_2\times_{\mathcal{M}_{FG}(1)}\mathcal{M}_{FG}(2^n)\]
     is an affine scheme. As in the proof of Proposition \ref{prop:quotientstackhasfinitechromaticdefect} following \cite[Theorem A7.1.1]{katzmazur}, since $\Spec(\mathrm{E(n)}_*)/C_2$ is a quotient stack, it suffices to show $\mathcal P$ is discrete. An object $P\in \mathcal P(R)$ has a nontrivial automorphism if and only if the nontrivial element $\gamma\in C_2$ acts on the formal group $F_P$ over $R$ by a power series $f(x)$ with the property that $f(x)\equiv x\mod x^{2^n+1}$.

     The element $\gamma$, however, acts by formal inversion on $F_P$. More specifically, the complex-conjugation action map $\gamma:\mathrm{MU}\to \mathrm{MU}$ induces the map on homotopy groups classifying the conjugate of the universal formal group law $F_{univ}$ by $-[-1]_{F_{univ}}(x)$ (see \cite[Example 11.19]{hhr}). Thus if $P$ has a nontrivial automorphism, then
\[[-1]_{F_P}(x)\equiv x\mod x^{2^n+1}\]
     Now since $v_n$ is a unit in $R$, and $\gamma^*F_P=F_P$, it follows from the fact that $\gamma(v_n)=-v_n$ that $2v_n=0\in R$, so $R$ is an $\F_2$-algebra, and in particular $-[-1]_{F_P}(x)=[-1]_{F_P}(x)$. This implies that
\[0=F_P(x,[-1]_{F_P}(x))\equiv F_P(x,x)\mod x^{2^n+1}\]
but the right hand side is $[2]_{F_P}(x)$, and $F_P$ has height $\le n$, a contradiction for a nonzero ring $R$.

To see that $\Phi(\mathrm{ER(n)})\ge 2^n$, suppose to the contrary that $\mathrm{ER(n)}\otimes \mathrm{X(2^n-1)}$ is complex orientable. By Proposition \ref{prop:stacker(n)} it follows that there is a pullback square
\[
\begin{tikzcd}
\mathcal M_{\mathrm{ER(n)}\otimes \mathrm{X(2^n-1)}}\arrow[r]\arrow[d]&\mathcal M_{FG}(2^n-1)\arrow[d]\\
\mathrm{Spec}(\mathrm{E(n)}_*)/C_2\arrow[r]&\mathcal M_{FG}(1)
\end{tikzcd}
\]
If the pullback $\mathcal M_{\mathrm{ER(n)}\otimes \mathrm{X(2^n-1)}}$ were affine, it would be discrete. This is a contradiction because we can consider the point of the pullback at the ring $\F_2[v_n^{\pm}]$ determined by the ring map $\mathrm{E(n)}_*\to\F_2[v_n^{\pm}]$ which kills $(2,v_1,\ldots,v_{n-1})$. This has a nontrivial automorphism because 
\[[-1]_{F_{univ}}(x)\equiv x\mod (2,v_1,\ldots,v_{n-1},v_nx^{2^n})\]
(see \cite[Proposition 3.5]{BHSZ})
\end{proof}

\subsection{Higher real \texorpdfstring{$K$}{K}-theories} We turn now to the determination of the chromatic defect of Goerss--Hopkins--Miller higher real $K$-theories. As before we fix a height $n$ formal group $\Gamma$ over a perfect field $k$ of characteristic $p$ and a finite subgroup $G\subset\mathbb G_n$ of the corresponding Morava stabilizer group, and we let $\mathrm{EO}_n(G)$ denote the fixed points of $\mathrm{E}(k,\Gamma)$ with respect to $G$. This follows a similar path to the previous subsection: we identify $\mathcal{M}_{\mathrm{EO}_n(G)}$ with a quotient stack and then use a hands-on stacks argument to determine the minimum $m$ such that the stack becomes affine after pulling back to $\mathcal{M}_{FG}(m)$. We have already made the first observation in the previous section, which we recall here. 

\begin{proposition}\label{prop:stackeo_n}
    There is a $\mathbb{G}_m$-equivariant equivalence of stacks
    \[\mathcal{M}_{\mathrm{EO}_n(G)}\simeq \Spec((\mathrm{E}_n)_*)/G\]
    over $\mathcal{M}_{FG}(1)$. 
\end{proposition}
\begin{proof}
    As in \cite[Theorem 5.10]{mathewmeier}, the stabilizer of any point $x\in \Spec((\mathrm{E}_n)_*)/G$ acts faithfully on the formal group since $G$ is a subgroup of the Morava stabilizer group, thus the morphism $\Spec((\mathrm{E}_n)_*)/G\to \mathcal{M}_{FG}(1)$ is affine, and $\mathrm{EO}_n(G)\to \mathrm{E}_n$ is a $G$-Galois extension. The latter implies that $(\mathrm{MU}^{\otimes s}\otimes \mathrm{E}_n)^{hG}\simeq \mathrm{MU}^{\otimes s}\otimes \mathrm{EO}_n(G)$, and Proposition \ref{prop:quotientstackandhtpyfixedpoints} applies.
\end{proof}

We will use the valuation on $\mathrm{End}(\Gamma)$ to calculate the chromatic defect of $\mathrm{EO}_n(G)$ in the following way. Fix a $p$-typical universal deformation $\tilde{\Gamma}$ of $\Gamma$, and an automorphism $g\in\mathrm{End}(\Gamma)^\times$ of $\Gamma$. By universal property, there is an induced isomorphism
\[[g]:\tilde{\Gamma}\to [g]^*\tilde{\Gamma}\]
of formal group laws over $W(k)[[u_1,\ldots,u_{n-1}]]$. Since we have chosen a $p$-typical coordinate, we may write
\[[g](x)=x+_{g^*\tilde \Gamma}\sum\limits_{i\ge 1} {^{g^*\tilde \Gamma}}t_i(g)x^{p^i}\]
and thus 
\[[g-1](x)=\sum\limits_{i\ge 1} {^{g^*\tilde \Gamma}} t_i(g)x^{p^i}\]
Working mod the maximal ideal $\mathfrak  m$ in $W(k)[[u_1,\ldots,u_{n-1}]]$ we have an expression
\[[g-1](x)=\sum\limits_{i\ge 1} {^\Gamma t_i(g)x^{p^i}}\]
so that $n\cdot\nu(g-1)$ is the minimum $i$ such that $t_i(g)\not\equiv 0$ mod $\mathfrak m$. Since $W(k)[[u_1,\ldots,u_{n-1}]]$ is a complete local ring, this is equivalent to asking that $t_i(g)$ be a unit. In the statement below, we let $\pi:\mathbb G_n\to \mathrm{Aut}(\Gamma)$ denote the projection map, where we identify the underlying set of the semidirect product as a cartesian product.

\begin{theorem}\label{thm:defecteon}
    Let $N(G):=n\cdot\max\{\nu(\pi(g)-1)\st e\neq \pi(g)\}_{g\in G}$, where $e$ is the identity element of $G$. Then $\Phi(\mathrm{EO}_n(G))=p^{N(G)}$.
\end{theorem}
\begin{proof}
    Our proof closely follows that of Theorem \ref{thm:defectern}, namely to see that $\Phi(\mathrm{EO}_n(G))\le p^{N(G)}$, we need to show that the stack
    \[\mathcal{P}:=\Spec((\mathrm{E}_n)_*)/G\times_{\mathcal{M}_{FG}(1)}\mathcal{M}_{FG}(p^{N(G)})\]
    is an affine scheme, for which it suffices to show in this case that it is discrete.  An object $P\in \mathcal P(R)$ has a nontrivial automorphism if and only if a nontrivial element $g\in G$ acts on the formal group $\Gamma_P$ on $R$ by a power series $[\pi(g)](x)$ with the property that $[\pi(g)](x)\equiv x\mod x^{p^{N(G)}+1}$. However, by definition of $N(G)$, we have that 
    \[[\pi(g)](x)=x+_{g^*\tilde \Gamma}\sum\limits_{i\ge 1} {^{g^*\tilde \Gamma}}t_i(g)x^{p^i}\]
    and $t_i(g)\not\equiv 0$ mod $\mathfrak m$ for some $i\le N(G)$, so that $t_i(G)$ is a unit in $(\mathrm{E}_n)_*$ and therefore in $R$, which gives a contradiction for a nonzero ring $R$.

    As before, to see that $\Phi(\mathrm{EO}_n(G))\ge p^{N(G)}$ it suffices to show that
    \[\Spec((\mathrm{E}_n)_*)/G\times_{\mathcal{M}_{FG}(1)}\mathcal{M}_{FG}(p^{N(G)}-1)\]
    is not discrete. A nontrivial automorphism in this stack is found over $k[u^{\pm}]$ at $\Gamma$ by taking $g$ so that $N(G)=n\nu(\pi(g)-1)$.
\end{proof}

Computing this valuation is a purely number-theoretic problem, and can be done with ease in many cases using general facts about valuations on division algebras, for which we refer the reader to \cite{serrelocal}. At a given height $n$ we let $F_n$ be the Honda formal group law over $\overline{\F_p}$. When $n=k(p-1)$, there is a tower of division algebras
\[\mathbb Q_p\subset\mathbb Q_p(\zeta_p)\subset \mathrm{End}(F_n)[1/p]\]
The element $\zeta_p\in \mathbb Q_p(\zeta_p)^\times$ has order $p$ and has positive valuation, therefore giving a copy of $C_p$ in $\mathrm{Aut}(F_n)$ and a $C_p$ action on $\mathrm{E}_n$.

\begin{corollary}\label{cor:eonexample1}
    We have $\Phi(\mathrm{EO}_{k(p-1)}(C_p))=p^k$.
\end{corollary}
\begin{proof}
    For all $0<k<p$, $\zeta^k-1$ is a uniformizer of $\mathcal O_{\mathbb{Q}_p(\zeta_p)}$, and since $\mathbb{Q}_p(\zeta_p)/\mathbb{Q}_p$ is totally ramified, we therefore have $\nu(\zeta^k-1)=\frac{1}{p-1}$, so that $N(C_p)=k$.
\end{proof}

More generally, $D=\mathrm{End}(F_h)[1/p]$ is the central $\mathbb{Q}_p$-division algebra with Hasse invariant $1/h$, and Serre showed that a field extension $K/\mathbb{Q}_p$ is contained in $D$ if and only if $[K:\mathbb{Q}_p]$ divides $h$ \cite{serrelocal}. Letting $h=p^{n-1}(p-1)m$, we have that
\[[\mathbb{Q}_p(\zeta_{p^n}):\mathbb{Q}_p]=\phi(p^n)=p^{n-1}(p-1)\]
so a choice of embedding of $\mathbb{Q}_p(\zeta_{p^n})$ gives as before a $C_{p^n}$ action on $\mathrm{E}_n$.

\begin{corollary}\label{cor:eonexample2}
 We have $\Phi(\mathrm{EO}_{p^{n-1}(p-1)m}(C_{p^n}))=p^{p^{n-1}m}$.
 \end{corollary}
\begin{proof}
 Again $\mathbb{Q}_p(\zeta_{p^n})/\mathbb{Q}_p$ is a totally ramified extension. One thus computes the valuations
\[\nu(\zeta_{p^n}-1)=\frac{1}{p^{n-1}(p-1)},\ldots,\nu(\zeta_{p^n}^{p^{n-1}}-1)=\frac{1}{p-1}\]
so that $N(C_{p^n})=p^{n-1}m$.
\end{proof}

\begin{remark}
    When $p=2$, the $\mathrm{ER(n)}$'s $\mathrm{K(n)}$-localize to $\mathrm{EO}_n(C_2)$'s, and thus the agreement of the numbers in Theorem \ref{thm:defectern} and Corollary \ref{cor:eonexample1} does not come as a surprise. In fact, Beaudry--Hill--Shi--Zeng construct genuine $C_{2^n}$-spectra known as the $D^{-1}\mathrm{BP}^{((G))}\langle m\rangle$'s whose fixed points $\mathrm{K(n)}$-localize to the $\mathrm{EO}_{2^{n-1}m}(C_{2^n})$'s of Corollary \ref{cor:eonexample2} \cite{BHSZ}. In this way the fixed point spectra $(D^{-1}\mathrm{BP}^{((G))}\langle m\rangle)^G$ give generalizations of the $\mathrm{ER(n)}$'s that capture larger cyclic 2-groups in the Morava stabilizer group. It can be shown that their chromatic defect coincides with the number given in Corollary \ref{cor:eonexample1}; we would like to return to these and their associated stacks in a future work.
\end{remark}

\section{\texorpdfstring{$\Z$}{Z}-indexed Adams-Novikov spectral sequences}\label{sec:7}
In their study of fp spectra, Mahowald and Rezk observed that the ASS of an fp spectrum $E$ of type n could be extended to a full plane SS converging to the homotopy groups of the chromatic localization $L_n^fE$. The essential idea is that, choosing a type $n+1$ complex $F$ so that $E\otimes F$ is an $\F_p$-module as in Definition \ref{def:fpspectra}, one can form an Adams tower for $E$ using $F$ in place of $\F_p$. Since $F$ is finite, the entire tower may be dualized, thus forming a $\Z$-indexed Adams tower, resulting in a full plane spectral sequence.

Our definition of a Wood-type $E$ gives $E$ precisely the properties necessary to emulate this construction with the Adams--Novikov spectral sequence. Unlike the Adams case, the $\Z$-indexed ANSS will instead converge to $0$. We develop the basics of these $\Z$-ANSSs in this section, relate their $E_2$-pages to various forms of Tate cohomology, and then plug in our motivating example, $\mathrm{ko}$.

There is an obvious generalization of the notion of Wood-types and $\Z$-indexed Adams spectral sequences to a general ring spectrum $E$ in place of $\mathrm{BP}$ or $\F_p$, producing full plane spectral sequences which converge to $L_n^fE$, whenever $E$ is an $E$-Wood-type via a finite type $n+1$ complex $F$, but we do not pursue this here.

\subsection{The ANSS for a Wood-type} We recall a general framework for Adams towers and their associated spectral sequences developed by Haynes Miller, see \cite{miller} as well as \cite[Section 4]{coctalos}. For an $\mathbb{E}_0$-algebra $R$, Miller defines an $R$-injective exactly as we define a weak $R$-module in Definition \ref{def:weakmodule}. Using this, he defines $R$-injective resolutions and relates them to $R$-based ASSs, from which the following may be deduced.

\begin{proposition}
    Let $R$ be an $\mathbb{E}_0$-algebra and $E$ a spectrum, and suppose we are given a tower
    \[
    \begin{tikzcd}
        \cdots\arrow[r]&E_{2}\arrow[d,"i_2"]\arrow[r]&E_1\arrow[d,"i_1"]\arrow[r]&E\arrow[d,"i_0"]\\
        &C_2&C_1&C_0
    \end{tikzcd}
    \]
    where $E_j=\mathrm{fib}(i_{j-1})$, $R\otimes i_j$ is a split monomorphism, and $C_j$ is a weak $R$-module, for all $j$. Then the spectral sequence obtained by applying $\pi_*(-)$ to this tower coincides with the $R$-based ASS of $E$ from the $E_2$-page on.
\end{proposition}

Suppose then that $F$ is an $\mathbb{E}_0$-algebra, and $E$ a spectrum, and form the $F$-based Adams tower of $E$:

\begin{equation}\label{equation:FASS}
    \begin{tikzcd}
        \cdots\arrow[r]&E\otimes \overline{F}\otimes\overline{F}\arrow[d,"i_2"]\arrow[r]&E\otimes \overline{F}\arrow[d,"i_1"]\arrow[r]&E\arrow[d,"i_0"]\\
        &E\otimes \overline{F}\otimes\overline{F}\otimes F&E\otimes \overline{F}\otimes F&E\otimes F
    \end{tikzcd}
\end{equation}
Suppose now that $R$ is a homotopy associative ring spectrum. If $E\otimes F$ is a weak $R$-module, and $F$ is an $R$-projective, it follows from the proposition that the associated spectral sequence recovers the $R$-based ASS of $E$ from the $E_2$-page on. Using that $F=\mathrm{T(n)}$ is a $R=\mathrm{BP}$-free, we have the following. 

\begin{corollary}
    Suppose $\Phi_p(E)\le n$, so that $E\otimes\mathrm{T(n)}$ is complex-orientable. Then the $\mathrm{T(n)}$-based $\mathrm{ASS}$ of $E$ coincides with the $\mathrm{ANSS}$ of $E$ from the $E_2$-page on.
\end{corollary}

We have the following further simplification when $E$ is Wood-type, setting now $F$ to be a finite $\mathrm{BP}$-projective and $R=\mathrm{BP}$. Here we shift $F$ so that its bottom cells are in degree zero and fix a cell $\mathbb{S}\to F$, thereby equipping $F$ with an $\mathbb{E}_0$-algebra structure.

\begin{corollary}\label{cor:woodtypeanss}
    Suppose that $E$ is Wood-type, and choose a finite $\mathrm{BP}$-projective $F$ so that $E\otimes F$ is complex-orientable. Then the $F$-based ASS of $E$ coincides with the ANSS of $E$ from the $E_2$-page on. 
\end{corollary}

\begin{example}
    Let $E=\mathrm{ko}$ and $F=C(\eta)$, i.e. the classical Wood equivalence. Corollary \ref{cor:woodtypeanss} gives an identification of $\mathrm{ANSS}(\mathrm{ko})$ with the $\eta$-Bockstein spectral sequence for $\mathrm{ko}$.

    Similar statements apply to the cases $E=\mathrm{tmf}$ with $F=D\mathcal{A}(1)$ at the prime 2 and $F=X_2$ at the prime 3 from Example \ref{example:woodtmf}. The identification of the $E_2$-page of the ANSS this provides corresponds to the stack-theoretic identifications of the descent SS for $\mathrm{tmf}$ as in \cite{bauer} using the $2$- and $3$-local covers of $\mathcal{M}_{ell}$ by $\mathcal{M}_1(3)$ and $\mathcal{M}_1(2)$ respectively.
 \end{example}

 \begin{remark}\label{rmk:relativeass}
     It is often the case, as in the example above, that for a Wood-type $\mathbb{E}_\infty$-ring spectrum $E$, the corresponding complex-oriented spectrum $T:=E\otimes F$ has the structure of an $E$-algebra, as in with $\mathrm{ko}\to\mathrm{ku}\simeq\mathrm{ko}\otimes C(\eta)$. In this case, $\mathrm{ANSS}(E)$ may be identified with relative Adams spectral sequence for the ring map $E\to T$ of \cite{bakerlazarev}. This has the advantage of endowing the spectral sequence of (\ref{equation:FASS}) with the structure of a spectral sequence of algebras from the $E_1$-page on.
 \end{remark}

\subsection{\texorpdfstring{$\Z$}{Z}-indexed ANSS and Tate cohomology} As above we fix a bottom cell $\mathbb{S}\to F$, and we let $\overline{F}:=\mathrm{fib}(\mathbb{S}\to F)$. The tower of (\ref{equation:FASS}) extends to the right whenever $F$ is dualizable:
\begin{equation}\label{equation:ZFASS}
    \begin{tikzcd}
        \cdots\arrow[r]&E\otimes \overline{F}\arrow[d,]\arrow[r]&E\arrow[d]\arrow[r]&E\otimes \mathbb D\overline{F}\arrow[d]\arrow[r]&\cdots\\
        &E\otimes \overline{F}\otimes F&E\otimes F&E\otimes \Sigma \mathbb DF
    \end{tikzcd}
\end{equation}

\begin{definition}
    If $E$ is Wood-type, fix a finite $\mathrm{BP}$-projective $F$ with $E\otimes F$ complex-orientable. The $\Z$-indexed Adams--Novikov spectral sequence of $E$ ($\Z\mathrm{-ANSS}(E)$) is the spectral sequence associated to the tower of (\ref{equation:ZFASS}).
\end{definition}

We will show now that the $\Z\mathrm{-ANSS}$ of a Wood-type $E$ almost completely determines the ANSS of $E$. For this, we will need some terminology introduced by Meier--Shi--Zeng \cite[Definition 2.5]{hfpsstate}.

\begin{definition}\label{def:isoofSSs}
    Let $f:\mathcal{E}\to\mathcal{E}'$ be a morphism of spectral sequences. We say $f$ is an \emph{isomorphism of spectral sequences in nonnegative filtrations} if the following statements hold:
    \begin{enumerate}
        \item On the $E_2$-page, $f$ induces an isomorphism in positive filtrations and an epimorphism in filtration zero.
        \item For all $r\ge 2$, every nonzero $d_r$-differential in $\mathcal{E}$ whose source is in nonnegative filtration is mapped via $f$ to a nonzero $d_r$-differential in $\mathcal{E}'$.
        \item For all $r\ge 2$, every nonzero $d_r$-differential in $\mathcal{E}'$ whose source is in nonnegative filtration is the image along $f$ of a nonzero $d_r$-differential in $\mathcal{E}$.
    \end{enumerate}
\end{definition}

In fact, conditions (2) and (3) of Definition \ref{def:isoofSSs} are redundant; they follow from condition (1). Indeed, this may be shown using an induction argument given in \cite[proof of Theorem 2.6]{hfpsstate}, which we now briefly recall (see also \cite[proof of Theorem 3.3]{vanishinglines}).

\begin{proposition}\label{prop:isoofSSs}
    Let $f:\mathcal{E}\to\mathcal{E}'$ be a morphism of spectral sequences such that, on the $E_2$-page, $f$ induces an isomorphism in positive filtrations and an epimorphism in filtration zero. Then $f$ is an \emph{isomorphism of spectral sequences} in nonnegative filtrations. 
\end{proposition}
\begin{proof}
     Conditions (2) and (3) of Definition \ref{def:isoofSSs} follow for $r=2$ immediately from condition (1) of Definition \ref{def:isoofSSs}, i.e. from the assumption of the proposition. We then let $r>2$ and suppose that conditions (2) and (3) of Definition \ref{def:isoofSSs} hold for all $r'<r$. 
     
     Suppose $d_r(x)=y$ is a nonzero differential in $\mathcal{E}$. We claim that $d_r(f(x))=f(y)$ is a nonzero differential in $\mathcal{E}'$. If not then $f(y)=0$ on the $E_r$-page. However, since $f(y)\neq0$ on the $E_2$-page, there must be a differential $d_{r'}(z)=f(y)$ in $\mathcal{E}'$ for $r'<r$. By induction using condition (3), this differential must lift to $\mathcal{E}$, contradicting that $y$ is nonzero on $E_r$. This establishes condition (2).

     Similarly, suppose $d_r(x)=y$ is a nonzero differential in $\mathcal{E}'$. Since $x$ admits a lift $\tilde{x}$ to $\mathcal{E}$ on the $E_2$-page, if the differential does not lift on the $E_r$-page, then there must be a nonzero differential $d_{r'}(\tilde{x})=z$ for $r'<r$. By induction using condition (2), it follows that $x$ supports a nonzero $d_{r'}$-differential in $\mathcal{E}'$, contradicting that $x$ survives to the $E_r$-page.
\end{proof}

\begin{theorem}\label{thm:ZANSS}
    The $\Z\mathrm{-ANSS}$ of a Wood-type $E$ has the following properties:
    \begin{enumerate}
        \item The $\Z\mathrm{-ANSS}$ is independent of the choice of finite $\mathrm{BP}$-projective $F$ from the $E_2$-page on.
        \item The natural map
    \[\mathrm{ANSS}(E)\to\Z\mathrm{-ANSS}(E)\]
    is an isomorphism of spectral sequences in nonnegative filtrations in the sense of Definition \ref{def:isoofSSs}. 
    \item The $\Z\mathrm{-ANSS}$ converges to zero.
    \end{enumerate}  
\end{theorem}
\begin{proof}
For (1), as in \cite[Lemma 4.8]{coctalos}, the $\mathrm{BP}$-injective resolutions 
    \[E\to E\otimes F_i\to\Sigma E\otimes\overline{F_i}\otimes F_i\to\cdots\]
    corresponding to the towers (\ref{equation:ZFASS}) for any pair of finite $\mathrm{BP}$-projectives $F_1,F_2$ with the property that $E\otimes F_i$ is $\mathrm{BP}$-injective are chain homotopy equivalent, since they are Adams resolutions. One may dualize the chain maps and chain homotopies to see that the resolutions of $E$
    \[\cdots\to \Sigma^{-1}E\otimes \mathbb D\overline{F_i}\otimes \mathbb DF_i\to E\otimes \mathbb DF_i\to E\]
    are chain homotopy equivalent. The two spectral sequences arise by splicing these resolutions together and thus must be isomorphic from the $E_2$-page on.

    For (2), note that condition (1) of Definition \ref{def:isoofSSs} follows by construction of the tower. The claim then follows immediately from Proposition \ref{prop:isoofSSs}.

    For (3), we first show that the colimit of the tower of (\ref{equation:ZFASS}) becomes zero after tensoring with $F$. Each of the downward maps in (\ref{equation:ZFASS}) admits a section after tensoring with $F$. Indeed, since $E\otimes F$ is complex-orientable, the map $E\otimes F\to E\otimes F\otimes F$ is a retract of the map $E\otimes F\otimes \mathrm{BP}\to E\otimes F\otimes F\otimes \mathrm{BP}$. Since $F$ is a $\mathrm{BP}$-projective, the map $F\otimes \mathrm{BP}\to F\otimes F\otimes \mathrm{BP}$ admits a section. It follows that all of the horizontal maps in (\ref{equation:ZFASS}) become zero after tensoring with $F$, as desired.
    
    Since $F$ is a type zero complex, a spectrum $X$ is zero if and only if $X\otimes F=0$, as follows from the thick subcategory theorem. Indeed the set of finite spectra $K$ such that $X\otimes K=0$ is a thick subcategory of finite spectra, and if it contains a type zero complex, it contains $\mathbb{S}$. It follows that the colimit of the tower of (\ref{equation:ZFASS}) is zero, so that the $\Z\mathrm{-ANSS}$ converges to zero.
\end{proof}

\begin{remark}
     Readers familiar with the homotopy fixed-point spectral sequence (HFPSS) and the Tate spectral sequence (TateSS) will notice similarities between these and the ANSS and $\Z$-ANSS of a Wood-type. Indeed, for a $G$-spectrum $E$, the canonical map $\mathrm{HFPSS}(E)\to\mathrm{TateSS}(E)$ is given on the $E_2$-page by the map
     \[H^*(G;\pi_*E)\to\widehat{H}^*(G;\pi_*E)\]
     from group cohomology to Tate cohomology. This map is an isomorphism in positive filtrations and a surjection in filtration zero, and hence the map $\mathrm{HFPSS}(E)\to\mathrm{TateSS}(E)$ is an isomorphism of spectral sequences in nonnegative filtrations in the sense of Definition \ref{def:isoofSSs}. 
     
     When $E^{tG}=0$, the Tate SS of $E$ converges to zero, as in condition (3) of Theorem \ref{thm:ZANSS}. This vanishing is often used to deduce vanishing lines and differentials in the HFPSS (see \cite{dannyvanishing} for example).
\end{remark}

We turn now to identifying the $E_2$ page of $\Z\mathrm{-ANSS}(E)$. In the following statements we will assume that $E$ is a Wood-type and choose a corresponding complex $F$ equipped with an $\mathbb{E}_1$-multiplication. Since the $\Z\mathrm{-ANSS}(E)$ does not depend on the choice of $F$, this results in no loss of generality since we may replace $F$ with its endomorphism ring $F\otimes \mathbb DF$. The $E_2$-page is computed via a certain localization in Hovey's category $\mathrm{Stable}(\mathrm{BP}_*\mathrm{BP})$-comodules, equivalently mod $\tau$ in Pstragowski's category $\mathrm{Syn}_{BP}$ \cite{piotr}. We will need to recall some terminology from \cite[Section 3.1]{mnn}.

\begin{definition}
    Let $\mathcal{C}$ be a presentably symmetric-monoidal stable $\infty$-category, and $A\in\mathrm{Alg}(\mathcal{C})$ a dualizable algebra. 
    \begin{itemize}
        \item Let $\mathcal{C}_{A\mathrm{-tors}}$ -- the category of $A$-torsion objects -- be the localizing subcategory of $\mathcal{C}$ generated by objects of the form $A\otimes X$, where $X$ is dualizable.
        \item Let $\mathcal{C}_{A^{-1}\mathrm{-local}}$ -- the category of $A^{-1}$-local objects -- be the subcategory of $\mathcal{C}$ consisting of objects $Y$ with the property that $\mathrm{Map}_{\mathcal{C}}(Z,Y)\simeq *$ for all $Z\in\mathcal{C}_{A\mathrm{-tors}}$.
    \end{itemize} 
\end{definition}

Any object $E\in\mathcal{C}$ fits into a natural cofiber sequence
    \[Z_{A^{-1}}E\to E\to L_{A^{-1}}E\]
    where $Z_{A^{-1}}E\in \mathcal{C}_{A\mathrm{-tors}}$ and $L_{A^{-1}}E\in \mathcal{C}_{A^{-1}\mathrm{-local}}$. Mathew, Naumann, and Noel give an explicit description of the localization functor $L_{A^{-1}}(-)$ as follows (see \cite[Proposition 3.5]{mnn}).

\begin{proposition}\label{prop:akhillocalization}
    Let $\mathrm{CB}(A)$ denote the cosimplicial object
\[A\implies A\otimes A\Rrightarrow A\otimes A\otimes A\cdots\]
    i.e. the cobar complex of $A$. Then for any $X\in\mathcal{C}$, the map $Z_{A^{-1}}E\to E$ is equivalent to the map
    \[|\mathbb D(\mathrm{CB}(A))\otimes E|\to E\]
    given by dualizing the cobar complex.
\end{proposition}

As above we fix an $\mathbb{E}_1$-ring $F$ such that $E\otimes F$ is complex-orientable, and we let 
\[\nu(-)/\tau:\Sp\to\mathrm{Stable}(\mathrm{BP}_*\mathrm{BP})\]
denote the functor sending a spectrum $E$ to the $\mathrm{BP}_*\mathrm{BP}$-comodule $\mathrm{BP}_*E$ regarded as a discrete object in $\mathrm{Stable}(\mathrm{BP}_*\mathrm{BP})$.

\begin{proposition}
    The $E_2$-page of the $\Z\mathrm{-ANSS}$ of a Wood-type $E$ is given by 
    \[E_2^{*,*}=\pi_{*,*}L_{(\nu F/\tau)^{-1}}(\nu E/\tau)\]
    in $\mathrm{Stable}(\mathrm{BP}_*\mathrm{BP})$.
\end{proposition}
\begin{proof}
    In \cite[Theorem 2.4]{mathewtate}, Mathew gives a similar categorical description of Tate cohomology, and the proof carries over to our setting without much change. The essential point is to show that $\pi_{*,0}L_{(\nu F/\tau)^{-1}}(\nu E/\tau)$ coincides with the zero line on our $E_2$-page, which is the cohomology of the complex
    \[\pi_*(E\otimes \mathbb DF)\to\pi_*(E\otimes F)\to\pi_*(\Sigma E\otimes\overline{F}\otimes F)\]
    Since $E\otimes F$ and therefore also $E\otimes \mathbb DF$ are weak $\mathrm{BP}$-modules, this is isomorphic to
    \[\Ext_{\mathrm{BP}_*\mathrm{BP}}^{0,*}(\mathrm{BP}_*(E\otimes \mathbb DF))\to\Ext_{\mathrm{BP}_*\mathrm{BP}}^{0,*}(\mathrm{BP}_*(E\otimes F))\to\Ext_{\mathrm{BP}_*\mathrm{BP}}^{0,*}(\mathrm{BP}_*(\Sigma E\otimes\overline{F}\otimes F))\]
    which is isomorphic to 
    \[\pi_{*,0}(\nu E/\tau\otimes \mathbb D\nu F/\tau)\to \pi_{*,0}(\nu E/\tau\otimes \nu F/\tau)\to \pi_{*,0}(\Sigma\nu E/\tau\otimes \overline{\nu F/\tau}\otimes\nu F/\tau)\]
    The kernel of the second map, by Corollary \ref{cor:woodtypeanss}, is isomorphic to $\Ext_{\mathrm{BP}_*\mathrm{BP}}^{0,*}(\mathrm{BP}_*(E))\cong\pi_{*,0}(\nu E/\tau)$. Since $\nu E/\tau$ and $\nu F/\tau$ are connective objects in $\mathrm{Stable}(\mathrm{BP}_*\mathrm{BP})$, it follows that 
    \[\pi_{*,0}(|\mathbb D(\mathrm{CB}(\nu F/\tau))\otimes \nu E/\tau|)\cong\pi_{*,0}(\nu E/\tau\otimes \mathbb D\nu F/\tau) \]
    and now the result follows from Proposition \ref{prop:akhillocalization}.
\end{proof}

The description of the $E_2$-page resembles Tate cohomology, as the latter is given by $\pi_*L_{k[G]^{-1}}(-)$ in the stable module category associated to a finite group $G$. However, our description rings a bit hollow at this level of generality, as one would hope for a description in terms of $L_{\mathrm{BP}_*\mathrm{BP}^{-1}}(-)$. The analogous description does hold in the Mahowald--Rezk context replacing $\mathrm{BP}$ with $\F_p$ (see \cite[Section 2]{mr}), where one may use the fact that $\mathcal{A}_*$ is projective in the category of finitely presented $\mathcal{A}_*$-comodules, and we see no direct analog of this in our setting.

In practice, however, we often find ourselves in the situation of Remark \ref{rmk:relativeass}, namely when $E$ is an $\mathbb E_\infty$ ring spectrum and $E\otimes F$ is a finite Adams-flat $E$-algebra. In this setting, the $E_2$-page of the $\Z\mathrm{-ANSS}(E)$ can be viewed more aptly as an instance of Tate cohomology. The proof of the following is exactly as in the previous proposition.

\begin{proposition}
    Let $E$ be an $E_\infty$-ring and $T\in\mathrm{Alg}(E)$ with the property that $\Gamma:=\pi_*(T\otimes_ET)$ is flat over $A=\pi_*T$. For any $X\in\mathrm{Mod}(E)$, the spectral sequence associated to the $\Z$-indexed tower
    \[
    \begin{tikzcd}
        \cdots\arrow[r]&X\otimes_E\overline{T}\arrow[d,]\arrow[r]&X\arrow[d]\arrow[r]&X\otimes_ED\overline{T}\arrow[d]\arrow[r]&\cdots\\
        &X\otimes_E\overline{T}\otimes_ET&X\otimes_ET&X\otimes_E\Sigma DT
    \end{tikzcd}
\]
has $E_2$-page given by $\pi_{*,*}L_{\Gamma^{-1}}(\nu X/\tau)$ in $\mathrm{Stable}(A,\Gamma)$.
\end{proposition}

Setting $T=E\otimes F$ and $X=E$ as in Example \ref{rmk:relativeass}, one recovers a description of the $E_2$-page of $\Z$-ANSS$(E)$ as $\pi_{*,*}L_{\Gamma^{-1}}\mathbb{1}$.

\subsection{The \texorpdfstring{$\Z$}{Z}-ANSS for \texorpdfstring{$\mathrm{ko}$}{ko}} We begin by calculating $\mathrm{ANSS}(\mathrm{ko})$, which by Corollary \ref{cor:woodtypeanss} can be identified with the $\eta$-Bockstein SS from the $E_2$-page on, equivalently, with the relative ASS for the map of commutative ring spectra $\mathrm{ko}\to \mathrm{ku}$, in the sense of \cite{bakerlazarev}. From this latter perspective, one sees that the spectral sequence has the structure of a spectral sequence of algebras, with associative $E_1$-page and commutative $E_r$-page for $r>1$. Since $\mathrm{ku}_*=\Z[u]$ with $|u|=2$, one has an additive isomorphism
\[E_1=\pi_*(\mathrm{ko}/\eta)[\eta]=\Z[u,\eta]\]
The class $\eta$ is a permanent cycle, and the differential $d_1$ is thus determined by the following.

\begin{proposition}\label{prop:d_1ANSS}
    In $\mathrm{ANSS}(\mathrm{ko})$, the $d_1$-differential satisfies
    \[d_1(u^n)=\begin{cases}2\eta u^{2k}&n=2k+1\\0&n=2k\end{cases}\]
\end{proposition}
\begin{proof}
    By the Leibniz rule, it suffices to establish $d_1(u)=2\eta$ and $d_1(u^2)=0$. These both follow from the Wood sequence: we have an exact sequence
    \[\pi_2\mathrm{ku}\xrightarrow{\partial}\pi_0\mathrm{ko}\xrightarrow{\eta}\pi_1 \mathrm{ko}\to 0\]
    Since $\pi_0\mathrm{ko}=\Z$ is generated by trivial bundles and $\pi_1 \mathrm{ko}=\Z/2$ via the Mobius bundle, we see that $\partial(u)=2$, which implies that $d_1(u)=2\eta$. Since $\partial$ is thus a monomorphism on $\pi_2\mathrm{ku}$, the exact sequence
    \[0\to\pi_1\mathrm{ko}\xrightarrow{\eta}\pi_2\mathrm{ko}\to\pi_2\mathrm{ku}\xrightarrow{\partial}\pi_0\mathrm{ko}\]
    implies $\pi_2 \mathrm{ko}$ is generated by classes divisible by $\eta$, which implies that $d_1(u^2)=0$, since $d_1$ is given by the composite 
    \[\pi_{4}\mathrm{ku}\xrightarrow{\partial}\pi_{2}\mathrm{ko}\to\pi_{2}\mathrm{ku}\]
    and the latter map has kernel those elements divisible by $\eta$.
\end{proof}

\NewSseqGroup\etatowers {} {
    \foreach \n in {0,...,20} {
    \class(\n+1,\n+1)
    }
    \foreach \m in {0,...,19} {
    \structline[red](\m,\m)(\m+1,\m+1)
    }
}

\NewSseqGroup\etatowersblue {} {
    \foreach \n in {0,...,20} {
    \class[blue](\n+1,\n+1)
    }
    \foreach \m in {0,...,19} {
    \structline[black](\m,\m)(\m+1,\m+1)
    }
}
\NewSseqGroup\etatowersred {} {
    \foreach \n in {0,...,20} {
    \class[red](\n+1,\n+1)
    }
    \foreach \m in {0,...,19} {
    \structline[black](\m,\m)(\m+1,\m+1)
    }
}

\begin{sseqdata}[name = ansskoattwo, Adams grading, classes = {fill, show name=below},
grid = go, xrange ={0}{20},yrange={0}{8},xscale=0.35,yscale=0.7,x tick step =2, y tick step =2,run off differentials = {->},struct lines = red ]

\class[rectangle](0,0)\class[rectangle,"u^2" {right,xshift=-0.25cm,yshift=-0.18cm}](4,0)\class[rectangle](8,0)\class[rectangle](12,0)\class[rectangle](16,0)\class[rectangle](20,0)

\class["\eta" {right,xshift=-0.25cm,yshift=-0.18cm}](1,1)
\structline(0,0)(1,1)
\etatowers(1,1)
\etatowers(4,0)
\etatowers(8,0)
\etatowers(12,0)
\etatowers(16,0)
\etatowers(20,0)
\foreach \i in {0,...,10} {
    \d3(4+\i,\i)
    }

\foreach \i in {0,...,10} {
    \d3(12+\i,\i)
    }

\foreach \i in {0,...,10} {
    \d3(20+\i,\i)
    }
\end{sseqdata}

\begin{figure}[!htbp]
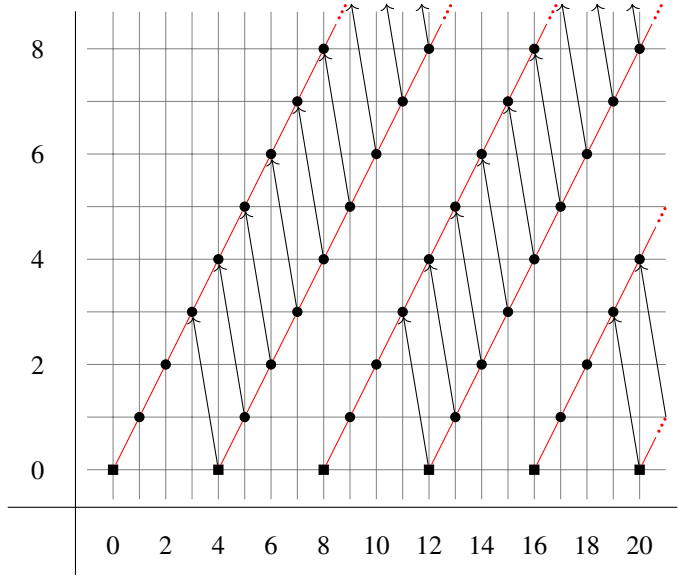

\centering
\printpage[name = ansskoattwo]
\caption{The ANSS for $\mathrm{ko}$.}
\label{koanss}
\end{figure}

\begin{sseqdata}[name = zansskoattwo, Adams grading, classes = {fill, show name=below},
grid = go, xrange ={-8}{20},yrange={-8}{8},xscale=0.2,yscale=0.35,x tick step =2, y tick step =2,run off differentials = {->},struct lines = black ]

\class(-9,-9)\class(-5,-9)\class(-1,-9)\class(3,-9)\class(7,-9)\class(11,-9)\class(15,-9)\class(19,-9)\class(-10,0)\class(21,0)\class(0,-10)\class(0,10)
\structline(0,-10)(0,10)
\structline(-10,0)(21,0)
\etatowers(-9,-9)
\etatowers(-5,-9)
\etatowers(-1,-9)
\etatowers(3,-9)
\etatowers(7,-9)
\etatowers(11,-9)
\etatowers(15,-9)
\etatowers(19,-9)

\foreach \i in {0,...,17} {
    \d3(-5+\i,-9+\i)
    }

\foreach \i in {0,...,17} {
    \d3(3+\i,-9+\i)
    }

\foreach \i in {0,...,17} {
    \d3(11+\i,-9+\i)
    }

\foreach \i in {0,...,17} {
    \d3(19+\i,-9+\i)
    }
\class(-6,-10)
\d3(-6,-10)
\class(-7,-11)
\d3(-7,-11)
\class(1,-11)
\d3(1,-11)
\class(2,-10)
\d3(2,-10)
\class(9,-11)
\d3(9,-11)
\class(10,-10)
\d3(10,-10)
\class(17,-11)
\d3(17,-11)
\class(18,-10)
\d3(18,-10)
    
\end{sseqdata}

\begin{figure}[!htbp]
\centering
\printpage[name = zansskoattwo]
\caption{The $\Z\mathrm{-ANSS}$ for $\mathrm{ko}$.}
\label{zkoanss}
\end{figure}

As a consequence of the Leibniz rule and the fact that $d_1(u^2)=0$, we see that in fact $u\eta=-\eta u$ on $E_1$. In any case it follows that $E_2=\Z[u^2,\eta]/(2\eta)$, as a ring. It follows from the Leibniz rule that either this spectral sequence collapses on $E_2$ or there is a differential $d_3(u^2)=\eta^3$. We will use the $\Z\mathrm{-ANSS}$ to show that this differential must happen. By construction $\Z\mathrm{-ANSS}(\mathrm{ko})$ satisfies
\[E_1=\Z[u,\eta^{\pm}]\]
as a module over $\mathrm{ANSS}(\mathrm{ko})$. It follows that $d_1$ is determined by the same formulae as in Proposition \ref{prop:d_1ANSS}, and we have the following.

\begin{proposition}
    The $E_2$-page of the $\Z\mathrm{-ANSS}(\mathrm{ko})$ is given as a module over that of $\mathrm{ANSS}(\mathrm{ko})$ by
    \[E_2=\F_2[u^2,\eta^{\pm}]\]
\end{proposition}

\begin{corollary}
    There is a differential $d_3(u^2)=\eta^3$ in both $\mathrm{ANSS}(\mathrm{ko})$ and $\Z\mathrm{-ANSS}(\mathrm{ko})$.
\end{corollary}
\begin{proof}
    If the claimed differential does not happen in $\mathrm{ANSS}(\mathrm{ko})$, then the spectral sequence collapses on $E_2$. Indeed if $d_3(u^2)=0$, then $u^2$ is a permanent cycle for degree reasons, as is clear from Figure 
    \ref{koanss}. Since $\mathrm{ANSS}(\mathrm{ko})$ is a multiplicative spectral sequence with $E_2$ generated as an algebra by $u^2$ and $\eta$, if $u^2$ is a permanent cycle, then there are no nonzero differentials in the spectral sequence, since $\eta$ is also a permanent cycle.    
    
    By $\eta$-linearity it follows that $\Z\mathrm{-ANSS}(\mathrm{ko})$ collapses on $E_2$. This contradicts Theorem \ref{thm:ZANSS}, namely that the latter spectral sequence must converge to zero.
\end{proof}

\printbibliography
\end{document}